\documentclass[reqno]{amsart}
\usepackage[marginratio=1:1]{geometry}

\usepackage[hidelinks]{hyperref}
\usepackage{calc}
\newsavebox\CBox
\newcommand\hcancel[2][0.5pt]{%
\ifmmode\sbox\CBox{$#2$}\else\sbox\CBox{#2}\fi%
  \makebox[0pt][l]{\usebox\CBox}%
  \rule[0.5\ht\CBox-#1/2]{\wd\CBox}{#1}}
\usepackage{blindtext}
\usepackage{color}
\usepackage{url}
\usepackage{comment}
\usepackage{float}
\usepackage{enumerate}
\usepackage{enumitem}   
\usepackage{cancel}
\usepackage{mathrsfs}
\usepackage{colonequals}
\usepackage{pdfpages}

\usepackage{amsmath,amsfonts,amsthm,bm}
\usepackage{mathrsfs}  
\usepackage{xcolor}
\usepackage{amsfonts}
\usepackage{amssymb}
\usepackage{amsmath}
\usepackage{mathtools}
\usepackage{mathrsfs}  
\numberwithin{equation}{section}
\usepackage[toc,page]{appendix}
\usepackage{tikz-cd}
\theoremstyle{definition}

\newtheorem{definicao}{Definition}[section]
\newtheorem{remark}[definicao]{Remark}

\theoremstyle{plain}

\newtheorem{teorema}[definicao]{Theorem}

\newtheorem{lema}[definicao]{Lemma}
\newtheorem{proposicao}[definicao]{Proposition}
\newtheorem{corolario}[definicao]{Corollary}

\usepackage{scalerel}

\definecolor{roxo}{rgb}{0.44, 0.16, 0.39}
\definecolor{ao(english)}{rgb}{0.0, 0.5, 0.0}
\definecolor{dmagenta}{RGB}{139, 0, 139}
\definecolor{dgreen}{RGB}{0,90,0}
\definecolor{navy}{RGB}{0,0,128}

\usepackage{stackengine}

\definecolor{iblue}{RGB}{0, 35, 194}

\address{$^{1}$Department of Mathematics, Imperial College London, London SW7 2AZ, UK}
\author[Giuseppe Tenaglia]{Giuseppe Tenaglia$^{1}$}
\email{\href{mailto:giuseppe.tenaglia20@imperial.ac.uk}{giuseppe.tenaglia20@imperial.ac.uk}}

\allowdisplaybreaks
\title{Abundance of typical horseshoes for the random standard map}
\date{September 2026}

\begin{document}
\raggedbottom
\begin{abstract}
{We introduce a notion of typical horseshoe, consisting of a pair of rectangles admitting Markov returns at times of positive lower density, with controlled hyperbolic geometry and orbits that shadow typical trajectories. We prove the abundance of typical horseshoes for the random standard map and establish several statistical properties of the system, including exponential mixing and large deviation estimates for the associated projective and two-point processes.}
\end{abstract}



\maketitle

\section{Introduction}

Random dynamical systems describe time-evolving motions whose future state depends on the present one and on the underlying noise. As deterministic ones, they may exhibit chaos through sensitive dependence on initial conditions: trajectories starting arbitrarily close can eventually separate. A statistical indicator of sensitive dependence are the Lyapunov exponents, which measure the asymptotic exponential growth rate of infinitesimal perturbations along typical trajectories. Positivity of the largest one  signals local instability. A natural question is whether this purely statistical instability gives rise to more geometric features that describe and organize chaotic dynamics.

In the absence of noise, a positive answer has been given by Katok. His seminal work \cite{PMIHES_1980__51__137_0} establishes that the dynamics of $C^{1+\alpha}$ non-uniformly hyperbolic diffeomorphisms can be approximated by uniformly hyperbolic invariant sets called \emph{horseshoes}. On these sets, a suitable iterate of the dynamics is topologically equivalent to a full shift over a finite number of symbols.

It is natural to ask whether analogous geometric structures can be constructed in random systems. The underlying difficulty is that deterministic horseshoe's constructions exploit the existence of periodic orbits with uniform geometric properties, whereas, in the random setting, periodic orbits do not
exists. In fact, return times fluctuate, and one needs to control how often
typical orbits return to a given set, while simultaneously exhibiting hyperbolic
behaviour and uniform geometric properties.

In \cite{LTTCircle,LTTGeneral}, we solved the above issues and established density of random horseshoes for non-uniformly expanding systems with diffusive noise. This class exhibits a more tractable version of the general difficulty: diffusive noise provides control on the statistics of the return times, and the absence of contracting directions in non-uniformly expanding systems simplifies the geometric control needed to construct random horseshoes.

In this work, we take a first step towards generalising our previous result to the non-uniformly hyperbolic setting, by establishing the existence of random horseshoes for the random standard map. This system takes the form
\begin{align*}
x_{n+1} &= L\cos\bigl(2\pi(x_n+\omega_n)\bigr)-y_n \pmod{1},\\
y_{n+1} &= x_n+\omega_n \pmod{1},
\end{align*}
where $L>0$ and $(\omega_n)_{n\ge0}$ is a sequence of independent and identically distributed random variables uniformly distributed in 
$(-L^{-\delta},L^{-\delta})$, for some $\delta>0$.
Its associated dynamics is non-uniformly hyperbolic: outside a small critical region of size $\sim L^{-1+\varepsilon}$, the expanding and contracting cone fields are strongly separated.

For this system, we introduce the notion of a \emph{typical horseshoe}. Informally, this consists of two nearby rectangles $I_1,I_2$ at a fixed scale, approximately aligned with the coordinate axes, such that any typical pair $(x_1,x_2),$ with $x_i\in I_i$ returns simultaneously to their respective rectangles at a sequence of times with positive lower density. Between consecutive selected returns, for each $i$, there exists a strip containing $x_i$ that crosses $I_i$ vertically and is mapped to a band crossing horizontally both rectangles simultaneously. For any fixed scale, Theorem \ref{thm:typical-horseshoes} establishes that, if $L$ is sufficiently large, every pair of rectangles outside the critical region satisfying the required geometric conditions forms a typical horseshoe.

Existence of horseshoes supported by fixed-scale rectangles outside the critical region is not by itself a novelty for the random standard map: such structures can be obtained by restricting to orbits that permanently avoid the critical region, following classical constructions \cite{Duarte1994}. The novelty of our main result, Theorem \ref{thm:typical-horseshoes}, is to show that these rectangles support horseshoes whose branches shadow typical pairs of trajectories—orbits that repeatedly enter the critical region where hyperbolicity is temporarily lost. Other approaches to horseshoe-like structures in random dynamical systems leverage positive entropy to construct symbolic models \cite{HuangLu2017,HuangZhang2023}, but do not establish geometric hyperbolicity.

The construction of typical horseshoes relies on several statistical properties of the random standard map that are also of independent interest. Together with the known results on positivity of Lyapunov exponents, expansion on average, and mixing of the one-point process, these provide a detailed quantitative picture of its statistical behaviour. In Theorem~\ref{mixingpm}, we establish annealed exponential mixing for the projective process, with a contraction factor of order $L^{-\delta}$.
In Theorem~\ref{mixingtwopointmotion}, we prove annealed exponential mixing for the two-point motion in a suitable weighted norm, with a contraction factor uniform in $L$ and $\delta$. These mixing estimates, combined with the techniques developed in \cite{LTTGeneral}, yield exponential large-deviation bounds for bounded observables, uniform in $L$ and $\delta$ (Theorem~\ref{thm:large_deviations}). Adapting the techniques of \cite[Proposition 7.2]{LTTCircle}, we also obtain exponential large-deviation estimates for finite-time Lyapunov exponents, with a contraction factor polynomial in $L^{-1}$ and uniform in $\delta$. Finally, we control the distribution of the stable direction and the frequency of its significant deviations from the vertical direction (Proposition~\ref{vertconc} and Corollary~\ref{Stabledef}).

Annealed exponential mixing of the projective process and the two-point motion has been established for related random shear models in \cite{BCZG2023,Zhang2026,LS26}, using a general framework based on controllability and Lyapunov methods. These models however involve independent random parameters for the two shears, whereas our model is driven by a single scalar noise whose amplitude decreases with $L$. Furthermore, the local minorization estimates underlying these arguments deteriorate in the large-parameter regime and do not yield the uniform mixing bounds obtained here.

As a further consequence, Theorem~\ref{mixingtwopointmotion}, combined with the argument of \cite[Proposition~4.6]{BCZG2023}, yields quenched exponential decay of correlations for Sobolev observables, with a deterministic decay exponent uniform in $L$ and $\delta$ and a random prefactor. Qualitative quenched exponential mixing already follows from expansion on average through \cite{DD24}; however, tracking the constants in that argument gives a contraction factor approaching one as $L\to\infty$. Similarly, the explicit decay exponent obtained in \cite{LS26} tends to zero as the kicking strength increases.

The technical machinery used in this work to construct typical horseshoes differs substantially from the one developed in \cite{LTTGeneral} for the non-uniformly expanding case. In that setting, our arguments crucially exploit the fact that the random times at which typical trajectories experience uniform expansion depend only on the past. In the hyperbolic setting, however, the information required to control the stable direction—and thereby construct random horseshoes—depends on the future.

A concrete illustration of this obstacle arises when trying to adapt the framework of \cite{DD24} to construct random horseshoes. Indeed, \cite{DD24} builds a finite-time Pesin theory for area-preserving diffeomorphisms to show that, at a specific time, two curves expand simultaneously while exhibiting points with long stable manifolds. Aside from controlling the exact alignment of these directions, this is precisely what our construction needs. To form a random horseshoe, however, such behavior must recur with positive density. Establishing this recurrence while accounting for future dependence would necessitate a detailed study of how stable directions depend on one another across trajectories, a task that lies beyond the scope of this work.

To resolve these issues, we construct two families of return times: the first controls the growth to large scale of horizontal curves centered at given marked points while ensuring that these points return to their respective reference rectangles; the second controls the geometry of the stable direction at the same marked points.  We then use quantitative density estimates to show that these families have an intersection of positive lower density, and use the resulting common return times to construct the random horseshoes. More precisely, our strategy proceeds  as follows:
\begin{enumerate}[nosep]
\item
Consider two admissible reference rectangles $I_1,I_2$, a typical pair of marked points $x_i\in I_i$, and curves $\gamma_i$ passing through $x_i$. Combining our statistical estimates with the finite-time Pesin theory of \cite{DD24}, we establish a sequence of times $n$ with positive lower density at which the Pesin charts along the trajectories of the marked points have polynomial size in $L^{-1}$ and shrink exponentially backwards in time. At these times, suitable pieces of the image curves reach a scale polynomial in $L^{-1}$ and are almost horizontal. During the next $K_1$ iterates, both trajectories avoid the critical region and return to their respective reference rectangles, while the image curves expand to cross both $I_1$ and $I_2$ horizontally. Critical avoidance continues for another $K_2$ iterates.
\item
We identify, for the marked points, stable times with density arbitrarily close to one for large $L$. At each such time, their stable direction is almost vertical, and we can construct a family of Pesin charts whose initial size is polynomial in $L^{-1}$ and whose sizes decrease exponentially fast along their future orbit.
\item
We retain those return times $h=n+K_1$ for which $h+K_2$ is a stable time for both marked trajectories. Our quantitative estimates show that, for large $L$, such times have positive lower density.

Between two sufficiently separated retained returns $h<h'$, we combine the Pesin charts starting at the stable time $h+K_2$ with those ending at the expanding time $h'-K_1$. This gives a compatible family of charts with polynomial size on $L^{-1}$ at both endpoints, independently of the length of the intervening orbit segment. Within these charts, horizontal curves grow forwards and vertical fibres grow backwards to polynomial scale in $L^{-1}$. This fact, along with the critical avoidance on the blocks $[h,h+K_2]$ and $[h'-K_1,h']$, allows us to construct vertical strips around the marked points that cross the corresponding reference rectangles vertically and whose images at time $h'$ simultaneously cross both reference rectangles horizontally.
\end{enumerate}
The horseshoe construction is carried out in Section~\ref{abundanceset}, following Sections~\ref{sec2}--\ref{revtsec}, where we establish the statistical estimates, study the distribution of stable directions, and adapt the finite-time Pesin framework of \cite{DD24} to our setting.

The paper is organised as follows. Section~\ref{sec2} states the main results and defines typical horseshoes. Sections~\ref{mixingderivatsec} and~\ref{sec:two_point_mixing} establish mixing for the projective process and the two-point motion, respectively. Section~\ref{ldevsec} proves the large-deviation estimates, and Section~\ref{stabledirsec} studies stable and unstable directions. Section~\ref{revtsec} develops the growth-to-large-scale estimates. Finally, Section~\ref{abundanceset} constructs simultaneous expanding returns and stable times and combines them to prove the abundance of typical horseshoes.

\section{Statement of the results}\label{sec2}

We study the random process $(x_n,y_n)_{n\ge0}$ on the
$2$-torus $\mathbb{T}^2$, defined by
\begin{equation}\label{randomstandardmap}
\begin{aligned}
x_{n+1} &= L\cos\bigl(2\pi(x_n+\omega_n)\bigr)-y_n \pmod{1},\\
y_{n+1} &= x_n+\omega_n \pmod{1},
\end{aligned}
\end{equation}
where $L>0$ and
$(\omega_n)_{n\ge0}$ is a sequence of independent and identically
distributed random variables with common distribution
$\operatorname{Unif}(-L^{-\delta},L^{-\delta}),$ with $\delta>0.$

Let $\Omega:=\left[-L^{-\delta},L^{-\delta}\right]^{\mathbb{N}_0}$ be the space of
noise sequences, equipped with the product probability measure
$\mathbb{P}$ whose coordinate marginals are uniform on
$\left[-L^{-\delta},L^{-\delta}\right]$. We denote expectation with respect to
$\mathbb{P}$ by $\mathbb{E}$.
Define the function $f(x):=L\cos(2\pi x).$ For each noise realization $\omega=(\omega_n)_{n\ge0}\in\Omega$, we define
\begin{equation}\label{maps}
f_{\omega}(x,y):=f_{\omega_0}(x,y)
=
\bigl(f(x+\omega_0)-y,\;x+\omega_0\bigr),
\end{equation}
where both coordinates are taken modulo $1$.

For $n\ge1$, we write
\[
f_\omega^n
:=
f_{\omega_{n-1}}\circ\cdots\circ f_{\omega_0},
\qquad
f_\omega^0:=\operatorname{id}.
\]
Equivalently, if $\vartheta:\Omega\to\Omega$ denotes the left shift, then
\[
f_\omega^n
=
f_{\vartheta^{n-1}\omega}\circ\cdots\circ
f_{\vartheta\omega}\circ f_\omega,
\qquad n\ge1.
\]
\subsection{Abundance of statistical properties}

Let $m$ denote the normalized Lebesgue measure on $\mathbb{T}^2$.
It is known \cite{BXY17} that $m$ is the unique ergodic stationary measure
for the random standard map. Consequently, by ergodicity, the
top Lyapunov exponent
\begin{align*}
\lambda_1(\omega,x)
:=
\lim_{n\to\infty}
\frac{1}{n}
\log \left\|Df_\omega^n(x)\right\|
\end{align*}
is constant $m\otimes\mathbb{P}$-almost surely. Moreover, in the
noise regime considered here, $\lambda_1>0$ for $L$ sufficiently
large \cite{BXY17}.

More importantly for our purposes, a sufficient condition for
the positivity of the Lyapunov exponent, usually referred to as
\emph{expansion on average}, was established for the random
standard map in \cite[Proposition 9]{BXY18}. We record below the version that will be used later.

\begin{teorema}\label{blubluaverage}
There exists $l_0\in \mathbb{N}$ such that, for every sufficiently small $\delta>0$, there exists $L_0(\delta)>0$ for which,
whenever $L>L_0(\delta)$ and $l\ge l_0$ 
\begin{equation}\label{exponavg}
\inf_{\substack{(x,y)\in\mathbb{T}^2\\ v\in\mathbb{S}^1}}
\mathbb{E}
\left[
\log\left|Df_\omega^l(x,y)\cdot v\right|
\right]
\ge
\frac34\log L.
\end{equation}
\end{teorema}

Throughout the paper, given a measurable space $E$, we denote by
$L^\infty(E)$ the Banach space of bounded measurable functions on $E$,
endowed with the supremum norm $\|\cdot\|_\infty$.

We define the Koopman operator associated to the random standard map by
\begin{equation}\label{koopmanonemotion}
\begin{aligned}
\mathcal P_1
&\colon L^\infty(\mathbb{T}^2)
\longrightarrow L^\infty(\mathbb{T}^2),\\
\mathcal P_1\phi(x,y)
&:=
\mathbb E\left[\phi\bigl(f_\omega(x,y)\bigr)\right].
\end{aligned}
\end{equation}
The following mixing estimate for $\mathcal{P}_1$ is essentially \cite[Theorem 2]{BXY18}.
\begin{teorema}\label{blublumixing}
There exist $K\in\mathbb N$ and $L_0>0$ such that, for every
$\delta>0$ sufficiently small and every $L>L_0$,
\begin{equation}\label{basemixing}
\sup_{(x,y),(x',y')\in\mathbb T^2}
\left|
\mathcal P_1^K\phi(x,y)-\mathcal P_1^K\phi(x',y')
\right|
\le
L^{-99/400}\|\phi\|_\infty,
\qquad
\forall\,\phi\in L^\infty(\mathbb T^2).
\end{equation}
In particular, there exists $C>0$ such that, for every $n\ge0$,
\begin{equation}\label{baseexpmixing}
\left\|
\mathcal P_1^n\varphi-\int_{\mathbb T^2}\varphi\,dm
\right\|_\infty
\le C L^{-\frac{99}{400}\lfloor n/K\rfloor}\|\varphi\|_\infty,
\qquad
\forall\,\varphi\in L^\infty(\mathbb T^2).
\end{equation}
\end{teorema}

In this work, we establish an analogous mixing property for the projective process associated with \eqref{randomstandardmap}.
Let $\mathbb{PT}^2$ denote the projectivized tangent bundle of $\mathbb T^2$, which we identify with
\begin{equation}\label{mathcalpt2}
\mathbb{PT}^2\simeq\mathbb T^2\times\mathbb P^1,
\end{equation}
where $\mathbb P^1$ denotes the real projective line,  viewed as the unit circle $\mathbb S^1$ with antipodal points identified. In particular, an element of $\mathbb P^1$ can be represented by a unit vector $v\in\mathbb S^1$, with $v$ and $-v$ representing the same direction.

The projective process is the Markov chain defined by
\begin{equation}\label{projectiveprocess}
(x_{n+1},y_{n+1},v_{n+1})
=
\left(
f_{\omega_n}(x_n,y_n),
\frac{Df_{\omega_n}(x_n,y_n)\cdot v_n}
{\left|Df_{\omega_n}(x_n,y_n)\cdot v_n\right|}
\right).
\end{equation}
Its Koopman operator is given by
\begin{equation}\label{Koopmanproh}
\begin{aligned}
\mathcal P_p
&\colon L^\infty(\mathbb{PT}^2)
\longrightarrow L^\infty(\mathbb{PT}^2),\\
\mathcal P_p\phi(x,y,v)
&:=
\mathbb E\left[
\phi\left(
f_\omega(x,y),
\frac{Df_\omega(x,y)\cdot v}{|Df_\omega(x,y)\cdot v|}
\right)
\right].
\end{aligned}
\end{equation}

We prove the following exponential mixing property.

\begin{teorema}\label{mixingpm}
For every $\delta>0$ sufficiently small, there exists $L_0>0$
such that, for every $L>L_0$, the projective dynamics
\eqref{projectiveprocess} admits a unique stationary probability
measure $\nu$, which is ergodic and mixing.
More precisely, there exists $M\in\mathbb N$, independent of
$L$ and $\delta$, such that, for every $n\ge0$,
\begin{equation}\label{projmixing}
\left\|
\mathcal P_p^n\phi-\int \phi d\nu \right\|_\infty
\le
2L^{-\frac{\delta}{2}\lfloor n/M\rfloor}
\|\phi\|_\infty,
\qquad
\forall\,\phi\in L^\infty(\mathbb{PT}^2).
\end{equation}
\end{teorema}

We do not claim that the contraction rate in \eqref{projmixing}
is optimal. In fact, we expect the contraction rate to be of
order $L^{-\gamma}$, for some $\gamma>0$ independent of $\delta$.
Although our argument does not yield such a rate, we believe
that it could be obtained by a refinement of the techniques
developed here, at the expense of substantially longer and more
involved estimates.

We finally turn to the two-point motion. Let
\[
\Delta:=
\left\{
(z,\bar z)\in\mathbb T^2\times\mathbb T^2:z=\bar z
\right\}
\]
denote the diagonal. The two-point motion associated with
\eqref{randomstandardmap} is the Markov chain on $(\mathbb T^2\times\mathbb T^2)\setminus\Delta$
defined by
\begin{equation}\label{twopointmotion}
(z_{n+1},\bar z_{n+1})
=
\left(f_{\omega_n}(z_n),f_{\omega_n}(\bar z_n)\right).
\end{equation}
Thus, the two-point motion describes the simultaneous evolution
of two distinct initial conditions under the same realization
of the noise.

Given a measurable function $V:(\mathbb T^2\times\mathbb T^2)\setminus\Delta
\longrightarrow[0,\infty),$ we define the weighted supremum norm
\begin{equation}\label{betanorm}
\|\phi\|_V
:=
\sup_{(z,\bar z)\in
(\mathbb T^2\times\mathbb T^2)\setminus\Delta}
\frac{|\phi(z,\bar z)|}{1+V(z,\bar z)}.
\end{equation}
We denote by
\begin{equation}\label{linfityc}
L^\infty_V\bigl((\mathbb T^2\times\mathbb T^2)\setminus\Delta\bigr)
:=
\left\{
\phi:(\mathbb T^2\times\mathbb T^2)\setminus\Delta\to\mathbb R:
\phi\text{ is measurable and }\|\phi\|_V<\infty
\right\}
\end{equation}
the corresponding weighted space.

Let $\mathcal P_2$ denote the Koopman operator associated with
the two-point motion, namely
\begin{equation}\label{koopmantwopont}
\mathcal P_2\phi(z,\bar z)
:=
\mathbb E\left[
\phi\bigl(f_\omega(z),f_\omega(\bar z)\bigr)
\right].
\end{equation}
Our main mixing result for the two-point motion is the following.

\begin{teorema}\label{mixingtwopointmotion}
There exist $C>0$, $\gamma\in(0,1)$, $\varepsilon_0>0$,
and $t\in\mathbb N$ such that, for every sufficiently small
$\delta>0$ and all $L$ sufficiently large, depending on $\delta$, if we set
\begin{equation}\label{VLdefinition}
V_L(z,\bar z)
:=
d(z,\bar z)^{-\frac{\varepsilon_0}{\log L}}
\mathbf1_{\{d(z,\bar z)\le L^{-t}\}}
+
\mathbf1_{\{d(z,\bar z)>L^{-t}\}},
\end{equation}
then we have, for every $n\ge1$,
\begin{equation}\label{2pmmixing}
\left\|
\mathcal P_2^n\phi
-\int_{\mathbb T^2\times\mathbb T^2}\phi\,d(m\otimes m)
\right\|_{V_L}
\le
C\gamma^n\|\phi\|_{V_L},
\qquad
\forall\,\phi\in
L^\infty_{V_L}\bigl((\mathbb T^2\times\mathbb T^2)\setminus\Delta\bigr).
\end{equation}
Equivalently, for every $(z,\bar z)\notin\Delta$,
\begin{align*}
\left|
\mathcal P_2^n\phi(z,\bar z)-\int\phi\,d(m\otimes m)
\right|
\le
C\gamma^n\bigl(1+V_L(z,\bar z)\bigr)\|\phi\|_{V_L}.
\end{align*}
\end{teorema}

Using the mixing estimates established in Theorems
\ref{mixingpm} and \ref{mixingtwopointmotion}, we obtain
exponential large-deviation bounds for bounded observables
of the projective process and the two-point motion,
with constants uniform in $L$ and $\delta$. 

\begin{teorema}\label{thm:large_deviations}
There exist constants $C,c>0$ such that, for every $\delta>0$
sufficiently small, there exists $L_0(\delta)>0$ for which
the following holds whenever $L\ge L_0(\delta)$.

Let $\nu=\nu_{L,\delta}$ denote the stationary probability
measure of the projective process. Then, for every bounded measurable
observable $H_p:\mathbb{PT}^2\to\mathbb R$, every initial state
$(x_0,y_0,v_0)\in\mathbb{PT}^2$, every $n\ge1$,
and every $\varepsilon>0$,
\begin{equation}\label{eq:projective_large_deviations}
\mathbb P\left(
\left|
\frac1n\sum_{j=0}^{n-1} H_p(x_j,y_j,v_j)-\int H_p\,d\nu
\right|\ge\varepsilon
\right)
\le
C\exp\left(-cn\frac{\varepsilon^2}{\|H_p\|^2_\infty}\right).
\end{equation}

Similarly, for every bounded measurable observable
$H_2:(\mathbb T^2\times\mathbb T^2)\setminus\Delta\to\mathbb R$,
every initial pair $(z_0,\bar z_0)\notin\Delta$, every $n\ge1$,
and every $\varepsilon>0$,
\begin{equation}\label{eq:two_point_large_deviations}
\mathbb P\left(
\left|
\frac1n\sum_{j=0}^{n-1} H_2(z_j,\bar z_j)
-\int H_2\,d(m\otimes m)
\right|\ge\varepsilon
\right)
\le
C\bigl(1+V_L(z_0,\bar z_0)\bigr)
\exp\left(-cn\frac{\varepsilon^2}{\|H_2\|^2_{\infty}}\right).
\end{equation}
\end{teorema}

Our final statistical result for the random standard map concerns large-deviation estimates for Lyapunov exponent.

\begin{teorema}\label{lyapunov_large_deviations}
For every $0<\alpha<24/25$, there exist $r_1,r_2>0$ and $L_0>1$ such that, for every sufficiently small $\delta>0$ and  every $L\ge L_0$,
\begin{equation}\label{lyapunov_ld_third_iterate}
\mathbb P\left\{
\log\left|Df_\omega^{3m}(x,y)\cdot v\right|
\le\alpha m\log L-C
\right\}
\le L^{-r_1m}e^{-r_2C}
\end{equation}
for every $(x,y)\in\mathbb T^2$, $v\in\mathbb S^1$, $m\ge1$, and $C\ge0$.
\end{teorema}

\begin{remark}
A consequence of Theorem~\ref{mixingtwopointmotion}, which we will not use in the proof of the horseshoe construction, is the following quenched mixing estimate. It follows from the argument of \cite[Proposition 4.6]{BCZG2023}, together with the uniform integrability of the weights $V_L$, so we omit the proof.

\begin{teorema}[Uniform quenched mixing]\label{cor:quenched_mixing}
For \(L\) and \(\delta\) as in Theorem \ref{mixingtwopointmotion}, for every $s,q>0$, there exist constants $C>0$ and $\gamma\in(0,1)$, independent of $L$ and $\delta$, and a nonnegative random variable $D=D_{L,\delta,s,q}$ with $\mathbb{E}[D^q]\leq C$, such that, for almost every $\omega$,
\[
\left|\int \varphi\,(\psi\circ f_\omega^n)\,dm\right|
\leq D(\omega)\gamma^n
\|\varphi\|_{H^s}\|\psi\|_{H^s}
\]
for all $n\geq 0$ and all mean-zero observables
$\varphi,\psi\in H^s(\mathbb{T}^2)$, where
$H^s(\mathbb{T}^2)$ denotes the Sobolev space of functions $\varphi\in L^2(\mathbb{T}^2)$ for which
\[
\|\varphi\|_{H^s}^2
:=\sum_{k\in\mathbb{Z}^2}(1+|k|^2)^s|\widehat{\varphi}(k)|^2<\infty,
\qquad
\widehat{\varphi}(k):=\int_{\mathbb{T}^2}\varphi(x)e^{-2\pi i k\cdot x}\,dm(x).
\]
\end{teorema}

\end{remark}

\subsection{Typical horseshoes}
In this subsection, we introduce the notion of typical horseshoes (Definition \ref{def:typical-horseshoe}) and state our main result on their existence (Theorem \ref{thm:typical-horseshoes}).
We begin by introducing the necessary geometric terminology.

Consider the cones
\[
\mathcal C^h
:=\left\{(v_x,v_y):|v_y|\le\frac{1}{1000}|v_x|\right\},
\qquad
\mathcal C^v
:=\left\{(v_x,v_y):|v_x|\le\frac{1}{1000}|v_y|\right\}.
\]
These cones have a dynamical significance. Indeed, by \eqref{randomstandardmap}, away from the critical set
\begin{equation}\label{eq:critical-set}
\mathcal C := \left\{0,\tfrac12\right\}\times\mathbb T.
\end{equation}
and for $L$ sufficiently large, the derivative of forward map preserves and expands $\mathcal C^h$,
whilst the derivative of the inverse map preserves and
expands $\mathcal C^v$.

We say that a regular $C^1$ curve is called almost horizontal,
respectively almost vertical, if its tangent lines lie in
$\mathcal C^h$, respectively $\mathcal C^v$.

A \emph{curvilinear rectangle} is a closed topological disk
bounded by two almost horizontal sides and two almost vertical
sides, meeting only at their endpoints.
\begin{definicao}[$\chi$-horseshoe position]
\label{def:horseshoe-position}
Let $\chi>0$ be sufficiently small. Two disjoint curvilinear
rectangles $I,\bar I$ are in $\chi$-horseshoe position if they
are contained in a common rectangular neighbourhood $U$ of
diameter at most $50\chi$ and satisfy the following conditions:
\begin{enumerate}[nosep]
\item Each horizontal side has arclength between $2\chi$ and
$4\chi$, and each vertical side has arclength between $10\chi$
and $20\chi$.
\item The projections onto the $y$-axis of the four almost
vertical sides of $I$ and $\bar I$ coincide.
\item The rectangles stay at distance at least $\chi$ from
the critical set:
$\operatorname{dist}_{\mathbb T^2}(I\cup\bar I,\mathcal C)\ge\chi$.
\end{enumerate}
\end{definicao}
For a pair of rectangles $I,\bar I$ in $\chi$-horseshoe position, choose curvilinear rectangles 
$I_c\subset I$ and $\bar I_c\subset\bar I$, each of area at least $\chi^2$ and at distance at least $\chi/4$ from the boundary of its respective
rectangle. We call $I_c$ and $\bar I_c$ the \emph{cores} of $I$ and $\bar I$, respectively.

Furthermore, choose slightly larger curvilinear rectangles $I^+,\bar I^+$, contained in $U$ and disjoint, with sides almost
parallel to the corresponding sides of $I$ and $\bar I$, respectively, such that $I\Subset\operatorname{int}I^+,$ and $\bar I\Subset\operatorname{int}\bar I^+.$ We call $I^+,\bar I^+$ the outer covers of $I,\bar I$.

Let $\gamma$ be a curve, and let $I,\bar I$ be rectangles in $\chi$-horseshoe position, for some $\chi>0$. We say that $\gamma$ \emph{crosses $I$ horizontally} if it contains an almost horizontal arc joining the left and right sides of $I$, whose interior lies in the interior of $I$. The crossing is \emph{with margin} if this arc extends along $\gamma$ to a horizontal crossing of $I^+$.
The corresponding definitions apply to $\bar I$ and $\bar I^+$. We call $\gamma$ a \emph{simultaneous horizontal crossing with margin} of $I,\bar I$ if it contains a connected almost horizontal arc that crosses both rectangles with margin. Vertical crossings  are defined by exchanging the coordinates.

Let $I$ be a curvilinear rectangle. A curvilinear rectangle $S$ is an  \emph{horizontal strip in $I$} if its left and right boundaries are subarcs of the corresponding boundaries of $I$, and its top and bottom boundaries are almost horizontal arcs crossing $I$ horizontally. The strip $S$ \emph{crosses $I$ with margin} if it
is the intersection with $I$ of a horizontal strip in $I^+$. 
Furthermore, $S$ \emph{is a horizontal strip crossing $I,\bar I$ horizontally and simultaneously with margin} if,
for each $J\in\{I,\bar I\}$, the intersections $S\cap J$
and $S\cap J^+$ are horizontal strips in $J$ and $J^+$,
respectively.

\emph{Vertical strips and their crossings with margin} are defined similarly by exchanging the coordinates.

Before giving a formal definition of typical horseshoes, let us give an intuitive idea. We say that two rectangles in $\chi$-horseshoe position form a typical horseshoe if, for almost every noise realization and almost every pair of initial points in the two cores, there exists a sequence of return times of positive lower density at which both trajectories return to their respective cores.  Between consecutive such times, certain vertical strips in each rectangle are mapped to horizontal strips crossing both rectangles,
as illustrated in Figure~\ref{fig:typical-horseshoe}.

\begin{figure}[H]
    \centering
    \includegraphics[width=0.85\linewidth]{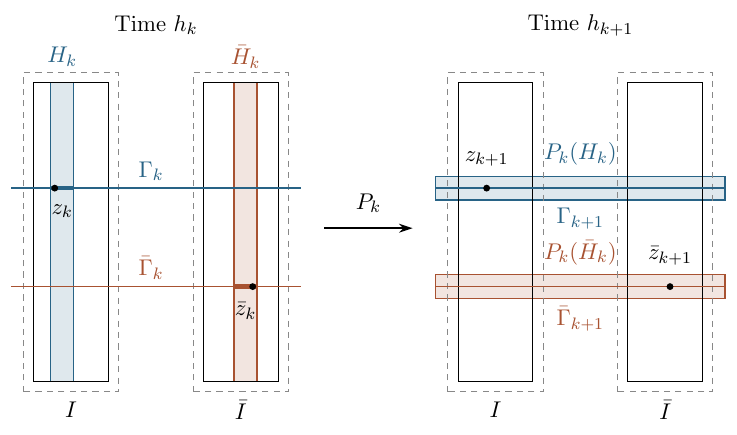}
    \caption{The map $P_k$ sends the vertical strips $H_k\subset I$
and $\bar H_k\subset\bar I$ to the horizontal strips $P_k(H_k)$
and $P_k(\bar H_k)$, each crossing both rectangles with margin.}
    \label{fig:typical-horseshoe}
\end{figure}

\begin{definicao}[Typical horseshoe]
\label{def:typical-horseshoe}
Let $(I,\bar I)$ be in $\chi$-horseshoe position, for some
$\chi>0$. We call $(I,\bar I)$ a \emph{typical horseshoe}
if there exist constants $d_*,\beta_*>0$ and a set
$\widetilde\Omega\subset\Omega$, with
$\mathbb P(\widetilde\Omega)=1$, such that the following holds.

For every $\omega\in\widetilde\Omega$ and for
$(m\otimes m)$-almost every
$(z,\bar z)\in I_c\times\bar I_c$ which do not have the same
$y$-coordinate, let $\gamma,\bar\gamma$ be the horizontal
line segments through $z,\bar z$ crossing $I,\bar I$,
respectively.
There exists a strictly increasing sequence of non-negative integers $\{h_k\}_{k\ge0}$ satisfying
\begin{equation}\label{posdensstm}
\liminf_{N\to\infty}
\frac{\#\{k\ge0:h_k(\omega)<N\}}{N}\ge d_*,
\end{equation}
with the following properties. Set
\[
n_k:=h_{k+1}-h_k,
\qquad
P_k:=f_{\vartheta^{h_k}\omega}^{n_k},
\qquad
z_k:=f_\omega^{h_k}(z),
\qquad
\bar z_k:=f_\omega^{h_k}(\bar z).
\]
Then, we require
\begin{enumerate}
\item\label{item1}
For every $k\ge0$, $z_k\in I_c$ and $\bar z_k\in\bar I_c$,
and there exist curves $\gamma_{h_k},\bar\gamma_{h_k}$ with
\[
z\in\gamma_{h_k}\subset\gamma,
\qquad
\bar z\in\bar\gamma_{h_k}\subset\bar\gamma,
\]
such that
\[
\Gamma_k:=f_\omega^{h_k}(\gamma_{h_k}),
\qquad
\bar\Gamma_k:=f_\omega^{h_k}(\bar\gamma_{h_k})
\]
are both simultaneous horizontal crossings with margin
of $I,\bar I$.

\item\label{item2} For every $k\ge0$, there are  $\Gamma'_k\subset\Gamma_k$ and
$\bar\Gamma'_k\subset\bar\Gamma_k$ satisfying
\[
P_k(\Gamma'_k)=\Gamma_{k+1},
\qquad
P_k(\bar\Gamma'_k)=\bar\Gamma_{k+1},
\]
and vertical strips $H_k,\bar H_k$ which admit a foliation of $C^1$ curves
\[
H_k=\bigsqcup_{p\in\Gamma'_k}W_k(p),
\qquad
\bar H_k=\bigsqcup_{p\in\bar\Gamma'_k}\bar W_k(p)
\]
by almost vertical crossings of $I,\bar I$, respectively,
satisfying
\[
W_k(p)\cap\Gamma'_k=\{p\},
\qquad
\bar W_k(p)\cap\bar\Gamma'_k=\{p\}
\]
for $p\in\Gamma'_k$ and $p\in\bar\Gamma'_k$, respectively. At time $h_{k+1}$, each of the images $P_k(H_k)$ and $P_k(\bar H_k)$ are horizontal strips crossing both $I$ and $\bar I$ with margin.

\item\label{itemthree} For every $q\in H_k\cup\bar H_k$ and every
nonzero $v\in\mathcal C^h$,
\begin{equation}\label{formula1}
DP_k(q)v\in\operatorname{int}\mathcal C^h,
\qquad
\|DP_k(q)v\|\ge L^{\beta_*n_k}\|v\|,
\end{equation}
and for every $q\in P_k(H_k\cup\bar H_k)$ and every nonzero $w\in\mathcal C^v$,
\begin{equation}\label{formula2}
DP_k^{-1}(q)w\in\operatorname{int}\mathcal C^v,
\qquad
\|DP_k^{-1}(q)w\|\ge L^{\beta_*n_k}\|w\|.
\end{equation}
\end{enumerate}
\end{definicao}

Note that item~\ref{itemthree} implies that the horizontal thickness of $H_k$ and the vertical thickness of $P_k(H_k)$ are at most of order
$L^{-\beta_*n_k}$.

\begin{remark}
Definition \ref{def:typical-horseshoe} also gives the Markov
structure familiar from deterministic horseshoes.
Indeed, let $(I,\bar I)$ be a typical horseshoe.
Fix $\omega\in\widetilde\Omega$ and a pair
$(z,\bar z)\in I_c\times\bar I_c$ for which the definition
provides a sequence $(h_k)_{k\ge0}$ satisfying
items \ref{item1}--\ref{itemthree}.
Then, for every $k\ge0$, one can choose vertical sub-strips 
\[
R_k^1,R_k^2\subset H_k,
\qquad
\bar R_k^1,\bar R_k^2\subset\bar H_k
\]
whose images satisfy
\[
\begin{aligned}
P_k(R_k^1)&=S_k^1\subset I,
&
P_k(R_k^2)&=S_k^2\subset\bar I,\\
P_k(\bar R_k^1)&=\bar S_k^1\subset I,
&
P_k(\bar R_k^2)&=\bar S_k^2\subset\bar I,
\end{aligned}
\]
where $S_k^1,\bar S_k^1$ are horizontal strips crossing $I$,
and $S_k^2,\bar S_k^2$ are horizontal strips crossing $\bar I$.

This Markov structure yields a one-sided symbolic coding
along the times $(h_k)_{k\ge0}$. Namely, writing
$I_1:=I$ and $I_2:=\bar I$, every sequence
$a=(a_k)_{k\ge0}\in\{1,2\}^{\mathbb N_0}$ is realised by
an orbit $(x_k)_{k\ge0}$ satisfying
\[
x_k\in I_{a_k},
\qquad
x_{k+1}=P_k(x_k)
\qquad\text{for all }k\ge0.
\]
\end{remark}

We are ready to state the main result of this work.
\begin{teorema}\label{thm:typical-horseshoes}
For $\delta>0$ sufficiently small, and $\chi>0$, there exists
$L_0$ such that, for $L\ge L_0$, every pair of rectangles
$I,\bar I$ in $\chi$-horseshoe position is a typical horseshoe. 
\end{teorema}

\subsection{Proof strategy and organization of the paper} 
The proof of Theorem~\ref{thm:typical-horseshoes} builds on several intermediate results, including the statistical properties stated in the previous section. Their role in the horseshoe construction becomes fully apparent only in the final section. We therefore give an overview here, explaining how these results fit together and contribute to the proof.

The core of our construction lies in Section~\ref{abundanceset}, in which we prove that the abundance of typical horseshoes follows from suitable lower density estimates for two families of dynamically relevant times: \emph{Young times} and \emph{stable times}. The rest of the paper is devoted, essentially, to produce the preliminary estimates needed to deduce such  bounds.

Establishing the lower density of Young times requires most of the work.
To this end, we combine the large deviation estimates of
Theorem~\ref{thm:large_deviations}, which follow from the mixing results
in Theorems~\ref{mixingpm} and~\ref{mixingtwopointmotion}, with the large
deviation estimates for Lyapunov exponents in
Theorem~\ref{lyapunov_large_deviations}. These results occupy the paper
through Section~5. We also use the estimates on the invariant measure
of the projective process in Proposition~\ref{proposition1} and the
tail estimates for the stopping times $\tau_j(\gamma)$ defined
in~\eqref{taugamma}. Establishing these tail estimates is the main
purpose of Section~\ref{revtsec}.

To estimate the lower density of stable times, we use the large
deviation estimates for Lyapunov exponents in
Theorem~\ref{lyapunov_large_deviations} and the estimates on the
stable direction in Corollary~\ref{Stabledef}.

\section{Mixing of the projective process}\label{mixingderivatsec}

In this section, we study the projective motion associated to the random standard map. To keep track of the angular direction, we use the slope coordinate
$s=\tan\theta\in\mathbb R\cup\{\infty\}\simeq\mathbb P^1$, so that a finite slope $s$ represents the direction of $(1,s)$, and $s=\infty$ represents the vertical direction. 

Let $\mathcal P_p$ denote the Koopman operator associated with
the projective process, as defined in \eqref{Koopmanproh}.
For each initial condition
$(x_0,y_0,s_0)\in\mathbb{PT}^2$,
let $\bar{\mathcal P}^n_{p,(x_0,y_0,s_0)}$
denote the corresponding $n$-step transition law.
Thus, for every bounded measurable function
$\phi:\mathbb{PT}^2\to\mathbb R$,
\begin{equation}\label{transition_formula}
\mathcal P_p^n\phi(x_0,y_0,s_0)
=
\int_{\mathbb{PT}^2}
\phi(x,y,s)\,
\bar{\mathcal P}^n_{p,(x_0,y_0,s_0)}(dx,dy,ds).
\end{equation}

For probability measures $\mu,\mu'$ on $\mathbb{PT}^2$,
we define their total variation distance by
\[
\|\mu-\mu'\|_{\mathrm{TV}}
:=
\sup_A|\mu(A)-\mu'(A)|,
\]
where the supremum is taken over all measurable sets.

The aim of this section is to prove the following result.

\begin{teorema}\label{mixing_proj_motion}
There exist $C>0$ and $M\in\mathbb N$, such that, for all $\delta>0$ sufficiently small and all
sufficiently large $L$,
\begin{align*}
\left\|
\bar{\mathcal P}^M_{p,(x_0,y_0,s_0)}
-
\bar{\mathcal P}^M_{p,(x_0',y_0',s_0')}
\right\|_{\mathrm{TV}}
\le
CL^{-\delta},
\qquad
\forall\,(x_0,y_0,s_0),(x_0',y_0',s_0')\in\mathbb{PT}^2.
\end{align*}
\end{teorema}
Note that this implies Theorem \ref{mixingpm} due to \eqref{transition_formula}
\subsection{Sketch of the proof} Note that the projective process has a skew-product structure: the base evolves according to the random standard map, while the fibre records the tangent direction. By \eqref{basemixing}, the base distribution is close to the Lebesgue measure after a finite number of iterates, uniformly in the initial condition. The remaining
question is whether the direction also loses memory of its initial
value. To understand this, look at the evolution of $s_n$. Define
\begin{equation}\label{allthegis}
g_1(x,y):=f(y)-x,
\qquad
g_2(x,y):=f'(g_1(x,y)).
\end{equation}
Writing $\varphi(a,s):=1/(a-s)$, we have that
\begin{equation}\label{recursion}
\begin{aligned}
s_n
&=\varphi(f'(y_n),s_{n-1}),\\
s_{n-1}
&=\varphi(g_2(x_n,y_n),s_{n-2}),\\
s_{n-2}
&=\varphi\bigl(f'(x_{n-3}+\omega_{n-3}),s_{n-3}\bigr).
\end{aligned}
\end{equation}
The key observation is that if $s_{n-2}$ is close to the horizontal direction and both $|f'(y_n)|$ and $|g_2(x_n,y_n)|$ are large, then $s_n$ is very close to
\begin{equation}\label{slope}
s^*(y):=\frac{1}{f'(y)},
\end{equation}
evaluated at $y=y_n$. This is the expected behaviour, since $|f|$ and $|f'|$ are typically of polynomial order in $L$, while the estimates in \cite[Lemma~10]{BXY17} support the picture that tangent directions are predominantly horizontal. Formula \eqref{slope} then suggests that
\begin{equation}\label{ennandem}
\operatorname{Law}(s_{n+m}\mid x_{n+m},y_{n+m})
\approx
\operatorname{Law}\bigl(
s_m(x_n,y_n,s^*(y_n))
\mid x_{n+m},y_{n+m}
\bigr),
\end{equation}
where $s_m(x,y,s)$ denotes the slope after $m$ further iterates
starting from $(x,y,s)$. If $n$ is large enough that also $(x_n,y_n)$ has already mixed,
we expect this conditional law to be approximately independent
of the initial condition as well.

The argument suggested by \eqref{ennandem} can be made rigorous using the Chapman-Kolmogorov equations. Indeed, let $K$ be as in \eqref{basemixing}, so that the base distribution is approximately Lebesgue after $K$. Chapman-Kolmogorov gives 
\begin{align*}
\bar{\mathcal P}^{4+K}_{p,(x_0,y_0,s_0)}(A)
-
\bar{\mathcal P}^{4+K}_{p,(x_0',y_0',s_0')}(A) \sim \int_{\mathbb T^2}\int_{\mathbb P^1}
\bar{\mathcal P}^4_{p,(x,y,s)}(A)
\bigl(q^K_{(x_0,y_0,s_0)}-q^K_{(x_0',y_0',s_0')}\bigr)
(ds\mid x,y)\,dm(x,y), 
\end{align*}
where $q^K_{(x_0,y_0,s_0)}(ds\mid x,y)$ denotes the conditional law of the slope after $K$ steps given $(x_K,y_K)=(x,y).$ This is essentially choosing $n=K$ and $m=4$ in \eqref{ennandem}, where the choice of $m=4$ is motivated by the control we have on the fourth-iterate of the projective dynamics due to our construction in subsection \ref{prel}.

We then split the integral on the right-hand side into three contributions:

\begin{itemize}
\item We first bound the contribution from base points in the
bad set
\begin{equation}\label{badset}
C:=
\left\{
(x,y)\in\mathbb T^2:
|f'(y)|<L^{2/3}
\ \text{or}\
|g_2(x,y)|<L^{2/3}
\right\},
\end{equation}
outside which both derivatives are large. This corresponds
to \eqref{i1}.

\item For $(x,y)\notin C$, we bound the contribution from
slopes outside the small interval
\begin{equation}\label{biuay}
B(y):=\{s\in\mathbb R:|s-s^*(y)|<L^{-4/3}\}
\end{equation}
around $s^*(y)$. This corresponds to \eqref{i2}.

\item Finally, for $(x,y)\notin C$ and $s\in B(y)$, we show
that the four-step transition law starting from $(x,y,s)$
is close to the one starting from $(x,y,s^*(y))$.
This controls the remaining contribution in \eqref{i3}
\end{itemize}

This argument is formalized in the following Lemma

\begin{lema}\label{reductionLemma}
Let $K\ge4$ be as in \eqref{basemixing}.   For
$(x_0,y_0,s_0),(x_0',y_0',s_0')\in\mathbb{PT}^2$, define
\begin{equation}\label{i1}
I_1
:=
\sup_A
\left|
\int_C\int_{\mathbb P^1}
\Delta_{p,A}(x,y,s)
\bigl(q^K_{(x_0,y_0,s_0)}-q^K_{(x_0',y_0',s_0')}\bigr)
(ds\mid x,y)\,dm(x,y)
\right|,
\end{equation}
\begin{equation}\label{i2}
I_2
:=
\sup_A
\left|
\int_{C^c}\int_{\mathbb P^1\setminus B(y)}
\Delta_{p,A}(x,y,s)
\bigl(q^K_{(x_0,y_0,s_0)}-q^K_{(x_0',y_0',s_0')}\bigr)
(ds\mid x,y)\,dm(x,y)
\right|,
\end{equation}
\begin{equation}\label{i3}
I_3
:=
\sup_A
\left|
\int_{C^c}\int_{B(y)}
\Delta_{p,A}(x,y,s)
\bigl(q^K_{(x_0,y_0,s_0)}-q^K_{(x_0',y_0',s_0')}\bigr)
(ds\mid x,y)\,dm(x,y)
\right|,
\end{equation}
where
\[
\Delta_{p,A}(x,y,s)
:=
\bar{\mathcal P}^4_{p,(x,y,s)}(A)
-
\bar{\mathcal P}^4_{p,(x,y,s^*(y))}(A),
\]
with $s^*(y)$ as in \eqref{slope}. Then
\[
\left\|
\bar{\mathcal P}^{4+K}_{p,(x_0,y_0,s_0)}
-
\bar{\mathcal P}^{4+K}_{p,(x_0',y_0',s_0')}
\right\|_{\mathrm{TV}}
\le
O(L^{-99/400})+I_1+I_2+I_3.
\]
\end{lema}

\begin{proof}
Let $K$ be as in \eqref{basemixing}. Then, for any measurable set
$A\subset\mathbb{PT}^2$, we have
\begin{align*}
\bar{\mathcal P}^{4+K}_{p,(x_0,y_0,s_0)}(A)
&=
\int
\bar{\mathcal P}^4_{p,(x,y,s)}(A)
q^K_{(x_0,y_0,s_0)}(ds\mid x,y)
\bar{\mathcal P}^K_{1,(x_0,y_0)}(dx,dy)\\
&=
\int_{\mathbb T^2}\int_{\mathbb P^1}
\bar{\mathcal P}^4_{p,(x,y,s)}(A)
q^K_{(x_0,y_0,s_0)}(ds\mid x,y)\,dm(x,y)\\
&\quad+
\int_{\mathbb T^2}\int_{\mathbb P^1}
\bar{\mathcal P}^4_{p,(x,y,s)}(A)
q^K_{(x_0,y_0,s_0)}(ds\mid x,y)
\bigl[
\bar{\mathcal P}^K_{1,(x_0,y_0)}(dx,dy)-m(dx,dy)
\bigr].
\end{align*}
Here $\bar{\mathcal P}^K_{1,(x_0,y_0)}$ is the base transition measure and $q^K_{(x_0,y_0,s_0)}(ds\mid x,y)$ denote the conditional law of the slope as previously introduced. The above identity, along with \eqref{basemixing}, implies that,
if $(x_0',y_0',s_0')$ is another initial condition,
\begin{align*}
&\left|
\bar{\mathcal P}^{4+K}_{p,(x_0,y_0,s_0)}(A)
-
\bar{\mathcal P}^{4+K}_{p,(x_0',y_0',s_0')}(A)
\right|\\
&\le
O(L^{-99/400})\\
&\quad+
\left|
\int_{\mathbb T^2}\int_{\mathbb P^1}
\bar{\mathcal P}^4_{p,(x,y,s)}(A)
\bigl(q^K_{(x_0,y_0,s_0)}-q^K_{(x_0',y_0',s_0')}\bigr)
(ds\mid x,y)\,dm(x,y)
\right|.
\end{align*}
Using
\[
\int
\bigl(q^K_{(x_0,y_0,s_0)}-q^K_{(x_0',y_0',s_0')}\bigr)
(ds\mid x,y)=0,
\]
we get
\begin{align*}
&\int_{\mathbb P^1}
\bar{\mathcal P}^4_{p,(x,y,s)}(A)
\bigl(q^K_{(x_0,y_0,s_0)}-q^K_{(x_0',y_0',s_0')}\bigr)
(ds\mid x,y)\\
&=
\int_{\mathbb P^1}
\bigl(
\bar{\mathcal P}^4_{p,(x,y,s)}(A)
-
\bar{\mathcal P}^4_{p,(x,y,s^*(y))}(A)
\bigr)
\bigl(q^K_{(x_0,y_0,s_0)}-q^K_{(x_0',y_0',s_0')}\bigr)
(ds\mid x,y).
\end{align*}
Consequently,
\begin{align*}
&\left|
\int_{\mathbb T^2}\int_{\mathbb P^1}
\bigl(
\bar{\mathcal P}^4_{p,(x,y,s)}(A)
-
\bar{\mathcal P}^4_{p,(x,y,s^*(y))}(A)
\bigr)
\bigl(q^K_{(x_0,y_0,s_0)}-q^K_{(x_0',y_0',s_0')}\bigr)
(ds\mid x,y)\,dm(x,y)
\right|\\
&\le I_1+I_2+I_3,
\end{align*}
using \eqref{badset} and \eqref{biuay}.
Taking the supremum over $A$ concludes the proof.
\end{proof}

Because $m(C)\le L^{-1/3}$, we have
\begin{equation}\label{i1lessimlminus}
I_1\lesssim L^{-1/3}.
\end{equation}

It remains to establish the following estimates.

\begin{lema}\label{finalmixinglema}
We have
\begin{equation}\label{i2final}
I_2
\lesssim
L^{-\delta}+L^{1/100+3\delta-1/16}
\lesssim L^{-\delta},
\end{equation}
and
\begin{equation}\label{i3final}
I_3\lesssim L^{-\delta}.
\end{equation}
\end{lema}

Together with the reduction lemma and \eqref{i1lessimlminus},
these estimates prove Theorem~\ref{mixing_proj_motion}.

The remainder of this section is devoted to proving
Lemma~\ref{finalmixinglema}. In Subsection~\ref{prel}, we describe the geometry of the set
of noise sequences satisfying $f_\omega^4(x_0,y_0)=(x,y)$ and obtain estimates uniform in the initial and terminal base points. We then use these estimates in the following two subsections to bound $I_2$ and $I_3$,
respectively.
\subsection{Geometric estimates}\label{prel}

The aim of this subsection is to study the geometry of the
noise sequences satisfying $f_\omega^4(x_0,y_0)=(x,y)$,
for prescribed $(x_0,y_0),(x,y)\in\mathbb T^2$.
For $\delta$ sufficiently small and $L$ sufficiently large,
we show that this set is nonempty and can be parametrized as $\bigl(
\omega_0,\omega_1,
\omega_2(\omega_0,\omega_1),
\omega_3(\omega_0,\omega_1)
\bigr),$ with $\omega_0$ and $\omega_1$ varying over  a large number of admissible intervals with controlled length. The underlying idea is that these intervals are sufficiently numerous and sufficiently long that, even after conditioning on $(x_m,y_m)=(x,y)$ for $m\ge4$, enough randomness remains to average effectively over the slope $s_m$ and obtain
the estimates of $I_2$ and $I_3$.

Fix $(x_0,y_0),(x,y)\in\mathbb T^2$ and
$(\omega_0,\omega_1)\in[-L^{-\delta},L^{-\delta}]^2$.
Observe that, if a pair
$(\omega_2,\omega_3)\in[-L^{-\delta},L^{-\delta}]^2$
completes $(\omega_0,\omega_1)$ to a noise realization $\omega$
such that $f_\omega^4(x_0,y_0)=(x,y)$, then this pair is unique.
Indeed, fixing $(\omega_0,\omega_1)$ determines $(x_2,y_2)$,
and the endpoint condition necessarily gives
\[
\omega_2=f(y)-x-x_2,
\qquad
\omega_3=y-f(f(y)-x)+y_2.
\]

This motivates the definition, for $(x_0,y_0),(x,y)\in\mathbb T^2$,
of the set of admissible initial noise pairs
\begin{equation}\label{omegax0y0xy}
\Omega_{(x_0,y_0)}^{(x,y)}
:=
\left\{
(\omega_0,\omega_1)\in[-L^{-\delta},L^{-\delta}]^2:
\exists\,(\omega_2,\omega_3)\in[-L^{-\delta},L^{-\delta}]^2
\text{ such that }f_\omega^4(x_0,y_0)=(x,y)
\right\}.
\end{equation}
Let
\begin{equation}\label{ax0y0xy}
A_{(x_0,y_0)}^{(x,y)}
:=
\left\{
\omega_0:
\exists\omega_1,\;
(\omega_0,\omega_1)\in\Omega_{(x_0,y_0)}^{(x,y)}
\right\}
\end{equation}
be its projection onto the first coordinate, and, for
$\omega_0\in A_{(x_0,y_0)}^{(x,y)}$, define the fibre
\begin{equation}\label{bixoyoxy}
B_{(x_0,y_0)}^{(x,y)}(\omega_0)
:=
\left\{
\omega_1:
(\omega_0,\omega_1)\in\Omega_{(x_0,y_0)}^{(x,y)}
\right\}.
\end{equation}

\begin{proposicao}\label{four-step-trans}
For every $(x_0,y_0),(x,y)\in\mathbb T^2$,  there exists
$\tilde A_{(x_0,y_0)}^{(x,y)}
\subset A_{(x_0,y_0)}^{(x,y)}$
with the following properties:
\begin{enumerate}
\item\label{coolitem} The set $\tilde A_{(x_0,y_0)}^{(x,y)}$ has $\gtrsim L^{99/100-\delta}$ connected components,
each of length $\gtrsim L^{-1-\delta}$.

\item\label{remainingclaim}
For every $\omega_0\in\tilde A_{(x_0,y_0)}^{(x,y)}$,
the set $B_{(x_0,y_0)}^{(x,y)}(\omega_0)$ contains
$\gtrsim L^{99/100-2\delta}$ pairwise disjoint intervals,
each of length $\gtrsim L^{-1-\delta}$.
\end{enumerate}
In addition,
\begin{equation}\label{measurethesitmates}
\begin{aligned}
\operatorname{Leb}\bigl(\tilde A_{(x_0,y_0)}^{(x,y)}\bigr)
&\asymp L^{-2\delta},\\
\operatorname{Leb}\bigl(
A_{(x_0,y_0)}^{(x,y)}
\setminus\tilde A_{(x_0,y_0)}^{(x,y)}
\bigr)
&\lesssim L^{-3\delta}.
\end{aligned}
\end{equation}
\end{proposicao}

The proof is given in Appendix~\ref{app:geometric-estimates}.

A consequence of Proposition~\ref{four-step-trans}, used in the estimates of both $I_2$ and $I_3$, is that when sampling $(\omega_0,\omega_1) \in \Omega_{(x_0,y_0)}^{(x,y)},$ the probability of extracting  $\omega_0\in A_{(x_0,y_0)}^{(x,y)}$ for which the fibre $B_{(x_0,y_0)}^{(x,y)}(\omega_0)$ has small Lebesgue measure is small.  More precisely, if 
\begin{equation}\label{quuuuuu}
Q
=
\mathbb P\left(
\,\cdot\,\middle|\Omega_{(x_0,y_0)}^{(x,y)}
\right).
\end{equation}
denotes the conditional measure in $\Omega_{(x_0,y_0)}^{(x,y)},$ then by Proposition~\ref{from_normalized_to_others}, its marginal $Q_0$ on the $\omega_0$-coordinate has density 
\begin{equation}\label{pomega0}
p(\omega_0)
:=
\frac{
\bigl|B_{(x_0,y_0)}^{(x,y)}(\omega_0)\bigr|
\mathbf1_{A_{(x_0,y_0)}^{(x,y)}}(\omega_0)
}{
\bigl|\Omega_{(x_0,y_0)}^{(x,y)}\bigr|
},
\end{equation}
and we have the following result.
\begin{proposicao}\label{marginals}
Let $(x_0,y_0),(x,y)\in \mathbb{T}^2.$ Fix a sufficiently small constant $c_0>0$ and define
\begin{equation}\label{biiiiii}
B'{}_{(x_0,y_0)}^{(x,y)}
:=
\left\{
\omega_0\in A_{(x_0,y_0)}^{(x,y)}:
\bigl|B_{(x_0,y_0)}^{(x,y)}(\omega_0)\bigr|
\le c_0L^{-1/100-3\delta}
\right\}.
\end{equation}
Then
\[
Q_0(B'{}_{(x_0,y_0)}^{(x,y)})\lesssim L^{-\delta}.
\]
\end{proposicao}

\begin{proof}
By \eqref{pomega0},
\[
Q_0(B'{}_{(x_0,y_0)}^{(x,y)})
=
\frac{
\int_{B'{}_{(x_0,y_0)}^{(x,y)}}
\bigl|B_{(x_0,y_0)}^{(x,y)}(\omega_0)\bigr|\,d\omega_0
}{
\int_{A_{(x_0,y_0)}^{(x,y)}}
\bigl|B_{(x_0,y_0)}^{(x,y)}(\omega_0)\bigr|\,d\omega_0
}.
\]
By Proposition \ref{four-step-trans} and \eqref{biiiiii},
choosing $c_0$ sufficiently small, we have $B'{}_{(x_0,y_0)}^{(x,y)}
\subset
A_{(x_0,y_0)}^{(x,y)}
\setminus\tilde A_{(x_0,y_0)}^{(x,y)},$ which,
 by \eqref{measurethesitmates}, implies
\begin{equation}\label{lebbprime}
\operatorname{Leb}\bigl(
B'{}_{(x_0,y_0)}^{(x,y)}
\bigr)
\lesssim L^{-3\delta}.
\end{equation}

Furthermore, by Proposition \ref{four-step-trans}
and \eqref{measurethesitmates},
\begin{align*}
\int_{A_{(x_0,y_0)}^{(x,y)}}
\bigl|B_{(x_0,y_0)}^{(x,y)}(\omega_0)\bigr|\,d\omega_0
&\ge
\int_{\tilde A_{(x_0,y_0)}^{(x,y)}}
\bigl|B_{(x_0,y_0)}^{(x,y)}(\omega_0)\bigr|\,d\omega_0\\
&\gtrsim
L^{-1/100-3\delta}L^{-2\delta}.
\end{align*}
Therefore, using the above, \eqref{biiiiii} and \eqref{lebbprime}, we obtain
\[
Q_0(B'{}_{(x_0,y_0)}^{(x,y)})
\lesssim
\frac{
L^{-1/100-3\delta}L^{-3\delta}
}{
L^{-1/100-3\delta}L^{-2\delta}
}
=
L^{-\delta}.
\]
\end{proof}

For the estimate of $I_3$, we also need local control of the fibres. The following proposition shows that, outside a small exceptional set, each admissible $\omega_1$ belongs to an interval of controlled length.

\begin{proposicao}\label{uniform_goodness}
Fix $\omega_0$ and assume that $f_\omega^4(x_0,y_0)=(x,y)$ for some
$(\omega_1,\omega_2,\omega_3)\in[-L^{-\delta},L^{-\delta}]^3$.
Define the enlarged critical set
\begin{equation}\label{aug}
\begin{aligned}
\mathcal C_3^{\mathrm{aug}}
:={}&
\left\{
\omega_3\in[-L^{-\delta},L^{-\delta}]:
\left|f'\bigl(f(f(y)-x)-y+\omega_3\bigr)\right|
\le2L^{1/2}
\right\}\\
&\cup [L^{-\delta}-4L^{-1-\delta},L^{-\delta}]
\cup [-L^{-\delta},-L^{-\delta}+4L^{-1-\delta}].
\end{aligned}
\end{equation}
Then:
\begin{enumerate}
\item[(i)] The set of $\omega_1$ for which there exist
$\omega_2$ and $\omega_3\in\mathcal C_3^{\mathrm{aug}}$
satisfying $f_\omega^4(x_0,y_0)=(x,y)$
has Lebesgue measure bounded by
\[
\operatorname{Leb}
\{\omega_1\text{ associated to bad }\omega_3\}
\lesssim L^{-1/2}.
\]

\item[(ii)] If $\omega_3\notin\mathcal C_3^{\mathrm{aug}}$
and $\omega_1$ is at distance at least $cL^{-1-\delta}$
from the boundary of the noise interval, then there exists
$c'>0$ such that either
\[
(\omega_1,\omega_1+c'L^{-1-\delta})
\subset B_{(x_0,y_0)}^{(x,y)}(\omega_0)
\]
or
\[
(\omega_1-c'L^{-1-\delta},\omega_1)
\subset B_{(x_0,y_0)}^{(x,y)}(\omega_0).
\]
In particular, every such $\omega_1$ belongs to a connected
component of the fibre of length at least $c'L^{-1-\delta}$.
\end{enumerate}
\end{proposicao}

\subsection{Estimate on \texorpdfstring{$I_2$}{I2}}\label{estimateofi2}

In this subsection, we estimate the quantity $I_2$ in \eqref{i2}.
Throughout this and the following subsection, we write
$s_i(x_0,y_0,s_0)$ for the slope after $i$ iterates starting
from $(x_0,y_0,s_0)$, suppressing its dependence on the noise
sequence $\omega$. 
First, note that
\begin{equation}\label{firsti2}
I_2
\lesssim
\sup_{(x_0,y_0,s_0)}
\operatorname*{ess\,sup}_{(x,y)\in C^c}
\mathbb P\left(
|s_K(x_0,y_0,s_0)-s^*(y)|>L^{-4/3}
\,\middle|\,
f_\omega^K(x_0,y_0)=(x,y)
\right).
\end{equation}

The recursion for $s_n$ in \eqref{recursion} implies that, if
$s_K(x_0,y_0,s_0)$ is far from $s^*(y),$ then
$s_{K-2}(x_0,y_0,s_0)$ is not horizontal enough.
Quantitatively, we can write
\begin{align*}
&\mathbb P\left(
|s_K(x_0,y_0,s_0)-s^*(y)|>L^{-4/3}
\,\middle|\,
f_\omega^K(x_0,y_0)=(x,y)
\right)\le
\mathbb P\left(
|s_{K-2}(x_0,y_0,s_0)|>L^{-1/16}
\,\middle|\,
f_\omega^K(x_0,y_0)=(x,y)
\right).
\end{align*}
Conditioning on the past sigma algebra up to time $K-5$ and reindexing the remaining
four noise variables, Proposition \ref{fromrmtogoodset} gives
\begin{align*}
&\mathbb P\left(
|s_{K-2}(x_0,y_0,s_0)|>L^{-1/16}
\,\middle|\,
f_\omega^K(x_0,y_0)=(x,y)
\right)\\
&\le
\sup_{(x_0',y_0',s_0')\in\mathbb{PT}^2}
\mathbb P\left(
|s_2(x_0',y_0',s_0')|>L^{-1/16}
\,\middle|\,
\Omega_{(x_0',y_0')}^{(x,y)}
\right),
\end{align*}
where $\Omega_{(x_0',y_0')}^{(x,y)}$ is defined as in
\eqref{omegax0y0xy}.
Relabel the initial condition as $(x_0,y_0,s_0)$, and
let $B'{}_{(x_0,y_0)}^{(x,y)}$ be as in \eqref{biiiiii} and $Q_0$ as in
\eqref{pomega0}. Then, we have
\begin{align*}
&\mathbb P\left(
|s_2(x_0,y_0,s_0)|>L^{-1/16}
\,\middle|\,
\Omega_{(x_0,y_0)}^{(x,y)}
\right)\\
&=
\int_{B'{}_{(x_0,y_0)}^{(x,y)}}
\mathbb P\left(
|s_2(x_0,y_0,s_0)|>L^{-1/16}
\,\middle|\,
\Omega_{(x_0,y_0)}^{(x,y)},\omega_0
\right)\,dQ_0(\omega_0)\\
&\quad+
\int_{(B'{}_{(x_0,y_0)}^{(x,y)})^c}
\mathbb P\left(
|s_2(x_0,y_0,s_0)|>L^{-1/16}
\,\middle|\,
\Omega_{(x_0,y_0)}^{(x,y)},\omega_0
\right)\,dQ_0(\omega_0)\\
&\lesssim
L^{-\delta}
+
\sup_{\omega_0\in A_{(x_0,y_0)}^{(x,y)}\setminus B'{}_{(x_0,y_0)}^{(x,y)}}
\mathbb P\left(
|s_2(x_0,y_0,s_0)|>L^{-1/16}
\,\middle|\,
\Omega_{(x_0,y_0)}^{(x,y)},\omega_0
\right),
\end{align*}
where in the last line we used Proposition \ref{marginals}.

For every
$\omega_0\in A_{(x_0,y_0)}^{(x,y)}\setminus B'{}_{(x_0,y_0)}^{(x,y)}$,
the one-step recurrence in \eqref{recursion} and the formula for the density of $\mathbb{P}\left(\cdot \middle|\,
\Omega_{(x_0,y_0)}^{(x,y)},\omega_0 \right)$ in Proposition \ref{from_normalized_to_others} give
\begin{align*}
&\mathbb P\left(
|s_2(x_0,y_0,s_0)|>L^{-1/16}
\,\middle|\,
\Omega_{(x_0,y_0)}^{(x,y)},\omega_0
\right)\\
&\le
\frac{
\operatorname{Leb}
\left\{
\omega_1\in B_{(x_0,y_0)}^{(x,y)}(\omega_0):
|f'(x_1(\omega_0)+\omega_1)-s_1(x_0,y_0,s_0)|<L^{1/16}
\right\}
}{
|B_{(x_0,y_0)}^{(x,y)}(\omega_0)|
}\\
&\lesssim
L^{1/100+3\delta}
\operatorname{Leb}
\left\{
\omega_1\in B_{(x_0,y_0)}^{(x,y)}(\omega_0):
|f'(x_1(\omega_0)+\omega_1)-s_1(x_0,y_0,s_0)|<L^{1/16}
\right\}\\
&\lesssim
L^{1/100+3\delta-1/16},
\end{align*}
where in the last line we argue as in \eqref{standardfsin}.

\subsection{Estimate on \texorpdfstring{$I_3$}{I3}}

In this subsection we prove \eqref{i3final}. Note that $I_3\lesssim H_3,$ where
\begin{equation}\label{i3lesseq}
H_3
:=
\sup_{(x_0,y_0)\notin C}
\sup_{s_0,s_0'\in B(y_0)}
\left\|
\bar{\mathcal P}^{4}_{p,(x_0,y_0,s_0)}
-
\bar{\mathcal P}^{4}_{p,(x_0,y_0,s_0')}
\right\|_{\mathrm{TV}},
\end{equation}
and $\bar{\mathcal P}^4_{1,(x_0,y_0)}(x,y)$ denotes the law of the base process after four steps. Let
$q^4_{(x_0,y_0,s_0)}(\cdot\mid x,y)$ denote the conditional law of
$s_4(x_0,y_0,s_0)$ given $(x_4,y_4)=(x,y)$. Then, by \eqref{i3lesseq},
\[
H_3
\le
\sup_{(x_0,y_0)\notin C}
\operatorname*{ess\,sup}_{(x,y)\in\mathbb T^2}
\sup_{s_0,s_0'\in B(y_0)}
\left\|
q^4_{(x_0,y_0,s_0)}(\cdot\mid x,y)
-
q^4_{(x_0,y_0,s_0')}(\cdot\mid x,y)
\right\|_{\mathrm{TV}}.
\]

Conditioned on $(x_4,y_4)=(x,y)$, the last two projective
updates are a deterministic function of $s_2(x_0,y_0,s_0)$.
Since measurable push-forwards do not increase total variation
distance, we have
\begin{align*}
&\left\|
\mathbb P\left(
s_4(x_0,y_0,s_0)\in\cdot
\,\middle|\,
f_\omega^4(x_0,y_0)=(x,y)
\right)
-
\mathbb P\left(
s_4(x_0,y_0,s_0')\in\cdot
\,\middle|\,
f_\omega^4(x_0,y_0)=(x,y)
\right)
\right\|_{\mathrm{TV}}\\
&\le
\left\|
\mathbb P\left(
s_2(x_0,y_0,s_0)\in\cdot
\,\middle|\,
\Omega_{(x_0,y_0)}^{(x,y)}
\right)
-
\mathbb P\left(
s_2(x_0,y_0,s_0')\in\cdot
\,\middle|\,
\Omega_{(x_0,y_0)}^{(x,y)}
\right)
\right\|_{\mathrm{TV}},
\end{align*}
where $\Omega_{(x_0,y_0)}^{(x,y)}$ is as in
\eqref{omegax0y0xy}.
We compare the conditional laws of $s_2(x_0,y_0,s_0)$ and
$s_2(x_0,y_0,s_0')$ by averaging over $\omega_1$ in the
fibre $B_{(x_0,y_0)}^{(x,y)}(\omega_0)$.
For this comparison, we restrict to those $\omega_0$
for which the fibre has sufficiently large Lebesgue
measure and $|s_1(x_0,y_0,s_0)-s_1(x_0,y_0,s_0')|\le L^{-4/3}$.
Accordingly, let $B'{}_{(x_0,y_0)}^{(x,y)}$ be as in \eqref{biiiiii} and define
\begin{equation}\label{ineedtis}
\begin{aligned}
\tilde B
&:=
B'{}_{(x_0,y_0)}^{(x,y)}\cup
\left\{
\omega_0\in A_{(x_0,y_0)}^{(x,y)}:
|s_1(x_0,y_0,s_0)-s_1(x_0,y_0,s_0')|>L^{-4/3}
\right\},\\
\tilde G
&:=
A_{(x_0,y_0)}^{(x,y)}\setminus\tilde B.
\end{aligned}    
\end{equation}
Conditioning on $\omega_0$, we have
\begin{align*}
&\left\|
\mathbb P\left(
s_2(x_0,y_0,s_0)\in\cdot
\,\middle|\,
\Omega_{(x_0,y_0)}^{(x,y)}
\right)
-
\mathbb P\left(
s_2(x_0,y_0,s_0')\in\cdot
\,\middle|\,
\Omega_{(x_0,y_0)}^{(x,y)}
\right)
\right\|_{\mathrm{TV}}\\
&\le
Q_0(B'{}_{(x_0,y_0)}^{(x,y)})
+Q_0\bigl(|s_1(x_0,y_0,s_0)-s_1(x_0,y_0,s_0')|>L^{-4/3}\bigr)\\
&\quad+
\sup_{\omega_0\in\tilde G}
\left\|
\mathbb P\left(
s_2(x_0,y_0,s_0)\in\cdot
\,\middle|\,
\Omega_{(x_0,y_0)}^{(x,y)},\omega_0
\right)
-
\mathbb P\left(
s_2(x_0,y_0,s_0')\in\cdot
\,\middle|\,
\Omega_{(x_0,y_0)}^{(x,y)},\omega_0
\right)
\right\|_{\mathrm{TV}},
\end{align*}
where $Q_0$ is as in \eqref{pomega0}.

The first term is $O(L^{-\delta})$ by Proposition
\ref{marginals}. Furthermore, by \eqref{pomega0},
\[
Q_0\bigl(|s_1(x_0,y_0,s_0)-s_1(x_0,y_0,s_0')|>L^{-4/3}\bigr)
=
\frac{
\displaystyle
\int_{\{|s_1(x_0,y_0,s_0)-s_1(x_0,y_0,s_0')|>L^{-4/3}\}}
|B_{(x_0,y_0)}^{(x,y)}(\omega_0)|\,d\omega_0
}{
|\Omega_{(x_0,y_0)}^{(x,y)}|
}.
\]
By \eqref{recursion} and the fact that
$|s_0-s_0'|\le2L^{-4/3}$, we have
\[
\operatorname{Leb}
\{\omega_0:|s_1(x_0,y_0,s_0)-s_1(x_0,y_0,s_0')|>L^{-4/3}\}
\lesssim L^{-1/2}.
\]
The above inequality, along with $|B_{(x_0,y_0)}^{(x,y)}(\omega_0)|\le2L^{-\delta}$
and the estimate on the measure of
$\Omega_{(x_0,y_0)}^{(x,y)}$ in Proposition
\ref{four-step-trans}, gives
\[
Q_0\bigl(|s_1(x_0,y_0,s_0)-s_1(x_0,y_0,s_0')|>L^{-4/3}\bigr)
\lesssim
L^{-1/2+1/100+4\delta}.
\]
Formula \eqref{i3final} now follows from the following
proposition.

\begin{proposicao}\label{Llarge}
For $L$ sufficiently large,
\begin{align*}
\sup_{\omega_0\in\tilde G}
\left\|
\mathbb P\left(
s_2(x_0,y_0,s_0)\in\cdot
\,\middle|\,
\Omega_{(x_0,y_0)}^{(x,y)},\omega_0
\right)
-
\mathbb P\left(
s_2(x_0,y_0,s_0')\in\cdot
\,\middle|\,
\Omega_{(x_0,y_0)}^{(x,y)},\omega_0
\right)
\right\|_{\mathrm{TV}}
\lesssim L^{-\delta}.
\end{align*}
\end{proposicao}
The proof is deferred to Appendix~\ref{app:slope-comparison}.

\section{Mixing of the two-point motion}\label{sec:two_point_mixing}

In this section, we prove annealed mixing for the two-point
motion, viewed as a Markov chain on the non-compact space
$(\mathbb T^2\times\mathbb T^2)\setminus\Delta$.
Our approach is to establish geometric ergodicity using
Harris' theorem. Let us first recall that, given a time-homogeneous Markov chain
$(Z_n)_{n\ge0}$ on a measurable space $(X,\mathcal B)$,
and writing $\mathbb P$ and $\mathbb E$ for probability and expectation on the underlying probability space, its $n$-step transition law is the family of probability measures
\[
\bar P_x^{\,n}(A):=\mathbb P(Z_n(x)\in A),
\qquad \forall  A\in\mathcal B,
\]
where $x\in X$ denotes the chosen initial conditions,
and its Koopman operator is defined by
\[
P\phi(x):=\mathbb E[\phi(Z_1(x))]
\]
for bounded measurable functions $\phi:X\to\mathbb R$. By the Markov property, the iterates of $P$ satisfy
\[
P^n\phi(x)
=
\mathbb E[\phi(Z_n(x))]
=
\int_X\phi(y)\,\bar P_x^{\,n}(dy),
\qquad n\ge0.
\]

The following quantitative version
of Harris' theorem is proved
\cite[Theorems~1.2 and~1.3]{HM11}.
\begin{teorema}\label{HM_theorem}
Consider a time-homogeneous Markov chain on a measurable
space $(X,\mathcal B)$, with Koopman operator $P$ and
transition laws $\bar P_x^{\,n}$.
Suppose there exist a measurable function $V:X\to[0,\infty)$,
an integer $N\ge1$, and constants $\gamma,h\in(0,1)$,
$K,R>0$, and $B\ge1$, with
\begin{equation}\label{harris_threshold}
R>\frac{2K}{1-\gamma},
\end{equation}
such that the following conditions hold:
\begin{enumerate}
\item
\begin{equation}\label{harris_drift}
P^NV(x)\le\gamma V(x)+K,
\qquad x\in X.
\end{equation}
\item There exists a probability measure $\xi$ on $X$ such that
\begin{equation}\label{harris_minorization}
\bar P_x^{\,N}(\cdot)\ge h\,\xi(\cdot),
\qquad \forall x\in\{V\le R\}.
\end{equation}
\item The operator $P$ is bounded on $L^\infty_V(X)$, with
\begin{equation}\label{normoperator}
\sup_{\|\phi\|_V=1}\|P\phi\|_V\le B,
\end{equation}
where the weighted space and its norm are defined in
\eqref{linfityc} and \eqref{betanorm}.
\end{enumerate}
Then the chain admits a unique invariant probability measure
$\pi$, satisfying
\begin{equation}\label{integrability}
\int_X V\,d\pi\le\frac{K}{1-\gamma}.
\end{equation}
Moreover, there exist constants $C>0$ and $\theta\in(0,1)$,
depending only on $N,\gamma,K,h,R,B$, such that
\begin{equation}\label{harris_mixing}
\left\|P^n\phi-\int_X\phi\,d\pi\right\|_V
\le C\theta^n\|\phi\|_V,
\qquad n\ge0,
\end{equation}
for every $\phi\in L^\infty_V(X)$.
\end{teorema}
Condition~\eqref{harris_drift} is called a \emph{drift condition},
and $V$ is called a \emph{Lyapunov function}.
Condition~\eqref{harris_minorization} is called a
\emph{minorization condition}, and the set $\{V\le R\}$
is called a \emph{small set}.
For the two-point motion on
$X=(\mathbb T^2\times\mathbb T^2)\setminus\Delta$,
the Koopman operator is $\mathcal P_2$, defined in
\eqref{koopmantwopont}, and we denote the $n$-step transition
law starting from $(z,\bar z)$ by
$\bar{\mathcal P}^{\,n}_{2,(z,\bar z)}$.
The following theorem establishes the drift and minorization
estimates needed to apply Theorem~\ref{HM_theorem} to this
process and is the main technical result of this section.

\begin{teorema}\label{what_harier_wants}
There exist constants $\varepsilon_0,K>0$,
$\alpha,h\in(0,1)$, and integers $l,t,r\ge1$, satisfying
\begin{equation}\label{two_point_threshold}
e^{2t\varepsilon_0}>\frac{2K}{1-\alpha},
\end{equation}
such that, for all sufficiently large $L$ and all sufficiently
small $\delta$, the function $V_L$ defined in
\eqref{VLdefinition} satisfies the following properties:
\begin{enumerate}
\item For every $(z,\bar z)\in
(\mathbb T^2\times\mathbb T^2)\setminus\Delta$,
\begin{equation}\label{two_point_drift}
\mathcal P_2^{jl}V_L(z,\bar z)
\le \alpha^j V_L(z,\bar z)+K,
\qquad j\ge0.
\end{equation}

\item The sublevel set
\begin{equation}\label{two_point_sublevel}
\mathcal C
:=\{V_L\le e^{2t\varepsilon_0}\}
=\{(z,\bar z):d(z,\bar z)\ge L^{-2t}\}
\end{equation}
satisfies
\begin{equation}\label{two_point_minorization}
\bar{\mathcal P}^{\,rl}_{2,(z,\bar z)}(A)
\ge h\,m_{C_{\mathrm{good}}}(A),
\qquad (z,\bar z)\in\mathcal C,
\end{equation}
for every measurable set $A$, where $C_{\mathrm{good}}$
is defined in \eqref{defCc} and
\begin{equation}\label{two_point_minorizing_measure}
m_{C_{\mathrm{good}}}(A)
:=\frac{(m\otimes m)(A\cap C_{\mathrm{good}})}
{(m\otimes m)(C_{\mathrm{good}})}.
\end{equation}
\end{enumerate}
\end{teorema}

In Subsection~4.1, we show that the above result
follows from the drift estimate in Proposition~\ref{Lyapunov function}
and the minorization estimate in
Proposition~\ref{small_set_on_C}.
Their proofs are given in Appendix~D and
Subsection~\ref{subsec:two_point_smallset}, respectively.

\begin{proof}[Proof of Theorem \ref{mixingtwopointmotion}]
We apply Theorem~\ref{HM_theorem} to $\mathcal P_2$ on
$X=(\mathbb T^2\times\mathbb T^2)\setminus\Delta$,
with Lyapunov function $V_L$ and parameters
\begin{equation}\label{two_point_harris_identification}
N=rl,\qquad \gamma=\alpha^r,\qquad
R=e^{2t\varepsilon_0},\qquad \xi=m_{C_{\mathrm{good}}}.
\end{equation}
The minorization constant $h$ in \eqref{harris_minorization} is the same $h$ in \eqref{two_point_minorization}.
The drift condition follows from \eqref{two_point_drift}
with $j=r$, while the minorization condition is precisely
\eqref{two_point_minorization}.
Finally, since $\alpha^r\le\alpha$,
\eqref{two_point_threshold} gives
\[
R=e^{2t\varepsilon_0}
>\frac{2K}{1-\alpha}
\ge\frac{2K}{1-\alpha^r},
\]
which verifies \eqref{harris_threshold}.

It remains to verify \eqref{normoperator}.
Since $f_\omega^{-1}$ is $(2\pi L+1)$-Lipschitz, writing $p=\varepsilon_0/\log L$, the definition of $V_L$ gives
\[
\begin{aligned}
V_L(f_\omega z,f_\omega\bar z)
&\le \max\left\{1,d(f_\omega z,f_\omega\bar z)^{-p}\right\}\\
&\le (2\pi L+1)^p\max\left\{1,d(z,\bar z)^{-p}\right\}\\
&\le (2\pi L+1)^p L^{tp}V_L(z,\bar z)\\
&= e^{t\varepsilon_0}(2\pi L+1)^pV_L(z,\bar z).
\end{aligned}
\]
for all sufficiently large $L$. 
Consequently,
\[
|\mathcal P_2\phi|
\le\|\phi\|_{V_L}\mathcal P_2(1+V_L)
\le e^{(t+2)\varepsilon_0}\|\phi\|_{V_L}(1+V_L),
\]
and hence
\begin{equation}\label{two_point_operator_bound}
\|\mathcal P_2\|_{V_L\to V_L}
\le e^{(t+2)\varepsilon_0}.
\end{equation}
Thus \eqref{normoperator} holds with $B=e^{(t+2)\varepsilon_0}$.

Each map $f_\omega\times f_\omega$ preserves $m\otimes m$, so this is the invariant probability measure identified by Harris' theorem.
The desired estimate \eqref{2pmmixing} now follows from
\eqref{harris_mixing}.
\end{proof}

\subsection{Geometric ergodicity of the two-point motion}
\label{subsec:two_point_geometric_ergodicity}

Theorem~\ref{what_harier_wants} follows from
Propositions~\ref{Lyapunov function} and~\ref{small_set_on_C}
below. The first establishes the drift condition
\eqref{two_point_drift}, while the second establishes the
minorization condition \eqref{two_point_minorization}.

To prove the drift estimate in \eqref{two_point_drift},  we combine the contraction of
the Lyapunov function near the diagonal with a uniform
bound away from it. Our choice of a Lyapunov function
singular at the diagonal is inspired by the framework of
\cite[Section~4]{BCZG2023}; in our setting, however, the
mechanism producing the drift is the expansion on average
in Theorem~\ref{blubluaverage}, which pushes nearby points away
from the diagonal.

Set $\Delta_t:=\{(z,\bar z)\in X:d(z,\bar z)\le L^{-t}\}$
and define
\begin{equation}\label{gittipi}
G_{t,p}(z,\bar z)
:=
d(z,\bar z)^{-p}\mathbf1_{\{d(z,\bar z)\le L^{-t}\}}
+
\mathbf1_{\{d(z,\bar z)>L^{-t}\}}.
\end{equation}
Thus $V_L=G_{t,\varepsilon_0/\log L}$, with
$\varepsilon_0>0$ chosen sufficiently small as in the
following proposition.

\begin{proposicao}\label{Lyapunov function}
Let $l_0$ as in Theorem \ref{blubluaverage}. For every fixed $l\ge l_0$, there exists
$\varepsilon_*>0$ such that, for every
$0<\varepsilon_0\le\varepsilon_*$ and $t\ge2l+1$,
the following estimates hold for all sufficiently large
$L$ and all sufficiently small $\delta>0$, with
$p:=\varepsilon_0/\log L$:
\begin{equation}\label{first_lyapunov}
\mathcal P_2^lG_{t,p}(z,\bar z)
\le\alpha G_{t,p}(z,\bar z),
\qquad (z,\bar z)\in\Delta_t,
\end{equation}
where $\alpha:=1-\varepsilon_0/2\in(0,1)$, and
\begin{equation}\label{concoutsidediagonal}
\sup_{(z,\bar z)\in X\setminus\Delta_t}
\mathcal P_2^lG_{t,p}(z,\bar z)
\le
1+O\left(L^{-\frac{\lambda}{2}(l-3)}\right),
\end{equation}
where $\lambda>0$ is given by
Proposition~\ref{smalldev}.
In particular, given $l$ and $\varepsilon_0$, we can
choose $t$ such that \eqref{two_point_threshold} and
\eqref{two_point_drift} hold with $K=2$.
\end{proposicao}

The proofs of \eqref{first_lyapunov} and
\eqref{concoutsidediagonal} are deferred to
Appendix~\ref{app:two_point_lyapunov}. The last assertion follows directly from these estimates.
Indeed, after fixing $l$ and $\varepsilon_0$, and hence
$\alpha$, choose $t\ge2l+1$ sufficiently large that
\begin{equation}\label{two_point_parameter_choice}
e^{2t\varepsilon_0}>
\frac{4}{(1-\alpha)^2},
\end{equation}
which gives \eqref{two_point_threshold} with $K=2$.
Furthermore, \eqref{first_lyapunov} and
\eqref{concoutsidediagonal} imply, for $L$ sufficiently large,
\begin{equation}\label{two_point_one_block_drift}
\mathcal P_2^lV_L
\le\alpha V_L+2(1-\alpha)
\qquad\text{on }X.
\end{equation}
Consequently, iterating this inequality gives
\begin{equation}\label{two_point_iterated_drift}
\mathcal P_2^{jl}V_L
\le
\alpha^jV_L+2(1-\alpha)\sum_{q=0}^{j-1}\alpha^q
\le
\alpha^jV_L+2,
\qquad j\ge0,
\end{equation}
which is \eqref{two_point_drift} with $K=2$.

Recall that the sublevel set defined in
\eqref{two_point_sublevel} is
\[
\mathcal C
=
\{(z,\bar z)\in X:d(z,\bar z)\ge L^{-2t}\}
=
\{G_{t,\varepsilon_0/\log L}\le e^{2t\varepsilon_0}\}.
\]
The following proposition establishes the minorization
condition \eqref{two_point_minorization} on $\mathcal C$.
\begin{proposicao}\label{small_set_on_C}
Fix $l$ and $t$ as in
Proposition~\ref{Lyapunov function}.
There exist an even integer $r\ge1$ and $\bar h>0$
such that
\begin{equation}\label{two_point_smallset_estimate}
\bar{\mathcal P}^{\,rl}_{2,(z,\bar z)}(\cdot)
\ge
\bar h\,m_{C_{\mathrm{good}}}(\cdot),
\qquad (z,\bar z)\in\mathcal C.
\end{equation}
\end{proposicao}

In Subsection~\ref{subsec:two_point_smallset}, we explain
the strategy for proving Proposition~\ref{small_set_on_C},
deferring the technical estimates to
Appendices~\ref{subsec:two_point_extension}
and~\ref{subsec:two_point_local_minorization}.

\begin{proof}[Proof of Theorem~\ref{what_harier_wants}]
Let $l_0$ be as in Theorem \ref{blubluaverage}. Fix an integer $l\ge l_0$ and choose $\varepsilon_0>0$
as in Proposition~\ref{Lyapunov function}.
By its last assertion, we can choose $t\ge2l+1$ such
that \eqref{two_point_threshold} and
\eqref{two_point_drift} hold with
$\alpha=1-\varepsilon_0/2$ and $K=2$.

For these choices of $l,t$,
Proposition~\ref{small_set_on_C} gives an even integer
$r\ge1$ and $\bar h>0$ satisfying
\eqref{two_point_smallset_estimate}.
Taking $h:=\bar h$ gives
\eqref{two_point_minorization} and concludes the proof.
\end{proof}

\subsection{The small-set construction}\label{subsec:two_point_smallset}

The proof of Proposition~\ref{small_set_on_C} relies on two
ingredients:
\begin{itemize}
\item A preliminary minorization estimate on an auxiliary set $S$;
\item Entrance and return estimates that allow us to extend this minorization to $\mathcal C$.
\end{itemize}

The local minorization is stated in
Proposition~\ref{First_step_small_set}, with its proof in
Appendix~\ref{subsec:two_point_local_minorization}.
The entrance and return estimates are stated in
Proposition~\ref{uniform_bound_on_T}, with their proof in
Appendix~\ref{subsec:two_point_extension}.

For $c>0$, define
\begin{equation}\label{defSc}
S:=
\left\{
(z_0,\bar z_0)\in (\mathbb{T}^2\times \mathbb{T}^2)\setminus \Delta :
\begin{array}{l}
\min\{|\sin(2\pi x_0)|,|\sin(2\pi\bar x_0)|\}\ge c,\\
|\cos(\pi(x_0+\bar x_0))\sin(\pi(\bar x_0-x_0))|\ge c
\end{array}
\right\}.
\end{equation}
Define the target set
\begin{equation}\label{defCc}
C_{\mathrm{good}}:=
\left\{
(z_0,\bar z_0)\in (\mathbb{T}^2\times \mathbb{T}^2)\setminus \Delta:
\begin{aligned}
&\min\{|\sin(2\pi A)|,|\sin(2\pi\bar A)|\}\ge10c,\\
&|\cos(\pi(A+\bar A))
\sin(\pi(\bar A-A))|\ge10c
\end{aligned}
\right\},
\end{equation}
where
\begin{equation}\label{two_point_inverse_initial_coordinates}
\begin{aligned}
y_1&:=f(y_0)-x_0,
&
\bar y_1&:=f(\bar y_0)-\bar x_0,\\
A&:=f(y_1)-y_0,
&
\bar A&:=f(\bar y_1)-\bar y_0.
\end{aligned}
\end{equation}

\begin{proposicao}\label{First_step_small_set}
For every $c>0,$ sufficiently small, there exists $\kappa>0$, independent of $L,\delta$, such that,
for every initial pair $(z_0,\bar z_0) \in S$ 
\begin{equation}\label{foundlowerdense}
\frac{d\bar{\mathcal P}^{\,6}_{2,(z_0,\bar z_0)}}{d(m\otimes m)}
(u,v,\bar u,\bar v)
\ge\kappa
\end{equation}
for $(m\otimes m)$-almost every
$(u,v,\bar u,\bar v)\in C_{\mathrm{good}}$.
Consequently, for some $b_0>0$ independent of $L,\delta$,
\begin{equation}\label{two_point_local_minorization}
\bar{\mathcal P}^{\,6}_{2,(z_0,\bar z_0)}(\cdot)
\ge b_0\,m_{C_{\mathrm{good}}}(\cdot),
\qquad (z_0,\bar z_0)\in S,
\end{equation}
where $m_{C_{\mathrm{good}}}$ is defined in
\eqref{two_point_minorizing_measure}.
\end{proposicao}

The next proposition provides the entrance and return estimates needed to extend this minorization to $\mathcal C$.

\begin{proposicao}\label{uniform_bound_on_T}
For every $c,\eta>0$ sufficiently small such that $S\subset\{d_{\mathbb T}(x,\bar x)>\eta\}$, there exist $T\in\mathbb N$, $b\in(0,1)$, and $d_*,p_1>0$,
independent of $L,\delta$, such that, for all sufficiently
large $L$, the following hold:
\begin{enumerate}
\item\label{itemforgotten}
For every $(z_0,\bar z_0)\in\mathcal C$, there is an integer
$0\le n=n(z_0,\bar z_0)\le T$ satisfying
\begin{equation}\label{prop18}
\mathbb P\bigl(d(z_n,\bar z_n)>\eta\bigr)\ge b.
\end{equation}

\item\label{thm:twostep_smallset}
If $d_{\mathbb T}(x_0,\bar x_0)\ge\eta/4$, then
\begin{equation}\label{lemma3}
\mathbb P\bigl((z_2,\bar z_2)\in S\bigr)\ge d_*.
\end{equation}

\item\label{lem:S_to_eta_one_step}
If $(z_0,\bar z_0)\in S$, then
\[
\mathbb P\bigl(d_{\mathbb T}(x_1,\bar x_1)>\eta\bigr)\ge p_1.
\]
Consequently, by \eqref{thm:twostep_smallset},
\begin{equation}\label{odd number}
\mathbb P\bigl((z_3,\bar z_3)\in S\bigr)\ge\widetilde p,
\end{equation}
where $\widetilde p:=p_1d_*>0$.
\end{enumerate}
\end{proposicao}

\begin{proof}[Proof of Proposition \ref{small_set_on_C}]

Fix $\eta>0$ sufficiently small, and then choose
$c>0$ sufficiently small so that
Propositions~\ref{uniform_bound_on_T}
and~\ref{First_step_small_set} apply.
In particular, we have $S\subset\{d_{\mathbb T}(x,\bar x)>\eta\}.$

By \eqref{two_point_local_minorization}, it suffices to find an even
integer $r$ and $d_0>0$, independent of $L,\delta$, such that
\begin{equation}\label{eq:return_to_S}
\mathbb P\bigl(z_{rl-6},\bar z_{rl-6}\in S\bigr)\ge d_0,
\qquad
(z,\bar z)\in\mathcal C.
\end{equation}
Indeed, the Markov property then gives
\begin{equation}\label{two_point_minorization_composition}
\begin{aligned}
\bar{\mathcal P}^{\,rl}_{2,(z,\bar z)}(A)
&\ge
b_0\,\mathbb P\bigl((z_{rl-6},\bar z_{rl-6} )\in S\bigr)
m_{C_{\mathrm{good}}}(A)\\
&\ge b_0d_0\,m_{C_{\mathrm{good}}}(A).
\end{aligned}
\end{equation}

Fix $(z_0,\bar z_0)\in\mathcal C$. We distinguish four cases.

\medskip
\noindent\textbf{Case 1: $(z_0,\bar z_0)\in S$.}

Since $S\subset\{d_{\mathbb T}(x,\bar x)>\eta\}$,
formulas \eqref{lemma3} and \eqref{odd number} give uniformly
positive probabilities of returning to $S$ after two and three steps,
respectively. Every integer at least two has the form $2a+3b$ with
$a,b\ge0$. Concatenating these returns therefore gives a positive
lower bound for returning at any prescribed sufficiently large time.

\medskip
\noindent\textbf{Case 2: $(z_0,\bar z_0)\notin S$ and
$d_{\mathbb T}(x_0,\bar x_0)>\frac{\eta}{4}$.}

By \eqref{lemma3}, $\mathbb P\bigl((z_2,\bar z_2)\in S\bigr)\ge d_*.$ After these two steps, argue as in Case 1.

\medskip
\noindent\textbf{Case 3: $L^{-2t}\le d(z_0,\bar z_0)\le\eta$.}

By Proposition \ref{uniform_bound_on_T}, some $1\le n\le T$ satisfies $\mathbb P\bigl(d(z_n,\bar z_n)>\eta\bigr)\ge b.$ Since $y_n-\bar y_n=x_{n-1}-\bar x_{n-1}$ modulo one,
\begin{equation}\label{stupidinclusion}
\begin{aligned}
\{d(z_n,\bar z_n)>\eta\}
\subset{}&
\left\{d_{\mathbb T}(x_n,\bar x_n)>\frac\eta4\right\}\\
&\cup
\left\{d_{\mathbb T}(x_{n-1},\bar x_{n-1})>\frac\eta4\right\}.
\end{aligned}
\end{equation}
Thus, at one of the two deterministic times $n-1,n$, the horizontal
separation exceeds $\eta/4$ with probability at least $b/2$.
We may then apply Case 2 and concatenate returns as in Case 1.

\medskip
\noindent\textbf{Case 4: $d(z_0,\bar z_0)>\eta$ and
$d_{\mathbb T}(x_0,\bar x_0)\le\frac{\eta}{4}$.}

After one step, $d(z_1,\bar z_1)
\ge\frac{\eta}{2\pi L+1}>L^{-2t}$ for $L$ sufficiently large. Proposition \ref{uniform_bound_on_T} item \ref{itemforgotten},
applied at time one, gives a uniformly positive probability of exceeding separation $\eta$ within the next $T$ steps.
Using \eqref{stupidinclusion}, we argue as in Case 3.
\end{proof}

\section{Large deviation estimates for the random standard map}\label{ldevsec}

In this section, we prove the large deviation estimates for bounded observables of the projective process and the two point motion stated in Theorem~\ref{thm:large_deviations}, together with those for the top Lyapunov exponent stated in Theorem~\ref{lyapunov_large_deviations}.

\begin{proof}[Proof of Theorem~\ref{thm:large_deviations}]

We first prove the large deviation estimates for bounded observables of the projective process in \eqref{eq:projective_large_deviations}. 

Let $\mathbb{PT}^2$ be as in \eqref{mathcalpt2}. For a bounded measurable observable $H_p\colon\mathbb{PT}^2\to\mathbb R$ and $\xi\in\mathbb C$, define the tilted operator
\begin{equation}\label{tiltedop}
\begin{aligned}
\mathcal P_{p,\xi}
&\colon L^\infty(\mathbb{PT}^2)\to L^\infty(\mathbb{PT}^2),\\
\mathcal P_{p,\xi}h
&:=\mathcal P_p(e^{\xi H_p}h),
\end{aligned}
\end{equation}
where $\mathcal P_p$ is the Koopman operator of the projective motion defined in \eqref{Koopmanproh}. 

By the analytic perturbation framework developed in
\cite[Section 3]{LTTGeneral} and the uniform spectral gap
established in Theorem~\ref{mixingpm} for fixed $\delta$
and sufficiently large $L$, formula \eqref{eq:projective_large_deviations} follows once we show that $\xi\mapsto\mathcal P_{p,\xi}$ is analytic in a neighbourhood of $0$ whose size is independent of $L$ and $\delta$. In fact, this family is entire: the expansion
\[
\mathcal P_{p,\xi}(h)
=
\sum_{k=0}^{\infty}
\frac{\xi^k}{k!}\mathcal P_p(H_p^k h)
\]
converges in operator norm for every $\xi\in\mathbb C$, since
\begin{equation}\label{auxiliarysum}
\|\mathcal P_p(H_p^k h)\|_\infty
\le \|H_p\|_\infty^k\|h\|_\infty.
\end{equation}

The large deviations estimate for the two point motion in \eqref{eq:two_point_large_deviations} follows by the same argument, applied to the Koopman operator $\mathcal P_2$ defined in \eqref{koopmantwopont}, acting on the weighted space $L^\infty_{V_L}\bigl((\mathbb T^2\times\mathbb T^2)\setminus\Delta\bigr)$ defined in \eqref{linfityc}, with $V_L$ as in \eqref{VLdefinition}. In this setting, \eqref{two_point_operator_bound} yields the analogue of \eqref{auxiliarysum} in the weighted norm $\|\cdot\|_{V_L}$ defined in \eqref{betanorm}, ensuring convergence of the corresponding expansion in operator norm.
\end{proof}

\begin{proof}[Proof of Theorem~\ref{lyapunov_large_deviations}]

Fix $(x,y,v)\in \mathbb{PT}^2$. For each $i\ge1$, define
\[
X_i
:=
\log \left\|
Df^3_{\vartheta^{3(i-1)}\omega}
\left(f^{3(i-1)}_{\omega}(x,y)\right)
\frac{Df^{3(i-1)}_{\omega}(x,y)\cdot v}
{\|Df^{3(i-1)}_{\omega}(x,y)\cdot v\|}
\right\|.
\]

We claim that
\begin{equation}\label{elementary}
\mathbb P\left\{
X_i\ge\frac{24}{25}\log L
\,\middle|\,\mathcal F_{i-1}
\right\}
\ge1-C_*L^{\delta-1/100},
\end{equation}

Indeed, the sine sub-level
estimate~\eqref{standardfsin} applied with $a=\frac{49}{50}$  gives
\begin{equation}\label{plsfiniscitutto}
\sup_{z\in\mathbb T^2}\sup_{v\in S^1}
\mathbb P\left(
\angle(Df_\omega(z)v,(1,0))>L^{-\frac{49}{50}}
\right)
\le C L^{\delta-\frac{1}{100}},
\end{equation}
showing that any initial vector becomes almost horizontal in
one iterate, up to discarding probability mass of order
$L^{\delta-1/100}$. During this first iterate, its norm may
contract at most of order $L^{-1}$.
At each of the following two iterates, we may discard an
additional probability mass of order $L^{\delta-1/100}$ to
ensure that $(x_j+\omega_j,y_j)$ lies at distance at least $L^{-1/100}$ from the critical set. On this event, the derivative preserves almost horizontality and expands almost horizontal vectors by a factor comparable to $L^{99/100}$. Combining these events over the first three iterates-first making the vector almost horizontal, then expanding it by avoiding the critical region in the next two iterates-gives \eqref{elementary}.

Define the auxiliary random variables
\[
\xi_i
:=
\begin{cases}
\dfrac{24}{25}\log L,
& \text{if $X_i \ge \dfrac{24}{25}\log L$},\\[0.3em]
-4\log L,
& \text{otherwise}.
\end{cases}
\]
The above definition, along with the crude lower bound $X_i\ge-4\log L$, implies 
\begin{equation}\label{smallbutstrung}
X_i\ge\xi_i,
\qquad \forall i\ge1.
\end{equation}

Let $A_n:=\#\left\{1\le i\le n:\xi_i=-4\log L\right\}$
count the blocks among the first $n$ for which
$X_i<\frac{24}{25}\log L$. Then
\begin{equation}\label{decompxiixix}
\sum_{i=1}^n\xi_i
=
\left(
\frac{24}{25}n-\frac{124}{25}A_n
\right)\log L.
\end{equation}

By \eqref{smallbutstrung}, the above formula implies
\[
\mathbb P\left\{
\sum_{i=1}^nX_i
\le
\alpha n\log L-C
\right\}
\le
\mathbb P\left\{
A_n\ge k_\alpha n+\beta\frac{C}{\log L}
\right\},
\]
where
$k_\alpha
:=
\frac{25}{124}
\left(
\frac{24}{25}-\alpha
\right)>0,$ and $\beta:=\frac{25}{124}.$ Fix $\delta_*<1/100$ such that $0<\delta\le\delta_*$. An exponential-moment argument combined with the conditional estimate \eqref{elementary} implies \eqref{lyapunov_ld_third_iterate}, with $r_1:=\frac{\xi k_\alpha}{2},\,\, r_2:=\xi\beta,$ for any fixed $0<\xi<\frac{1}{100}-\delta_*.$

\end{proof}

\section{Distribution of stable and unstable directions}\label{stabledirsec}

In this section, we study the invariant measures of the projective
processes associated with the random standard map and its inverse,
in order to estimate the distributions of the unstable and stable
directions of the forward dynamics, respectively.

Let us represent any projective direction $v \in\mathbb P^1$ by the unique angle $\theta\in[0,\pi)$ satisfying $v=(\cos\theta,\sin\theta).$ For the projective process of the random standard map, a simple
adaptation of the proof of \cite[Lemma~9]{BXY17} yields the following
concentration estimate.

\begin{proposicao}\label{proposition1}
Let $\nu$ be the mixing stationary measure of the projective process
given by Theorem~\ref{mixingpm}. Then there exist constants
$r,s>0$ such that, for all sufficiently
large $L$,
\[
\nu\bigl\{
(x,y,\theta):\theta \notin B_{L^{-s}}(0)
\bigr\}
\le L^{-r}.
\]
\end{proposicao}
\begin{proof}
We follow the proof of \cite[Lemma~9]{BXY17}, step $2$. Replacing their fixed angular cutoff by $L^{-s}$ changes only the lower bound on $\rho$ used after \cite[(10)]{BXY17}, from
$\rho\gtrsim L^2$ to $\rho\gtrsim L^{2-2s}$. With this change,
the remaining estimates yield the desired result for $s>0$
sufficiently small.
\end{proof}

\subsection{Distribution of the stable direction}

Let $f_\omega$ be as in \eqref{maps}. The inverse maps $f_{\omega}^{-1},$ whose formula can be found in \eqref{inverseformula}, induces the iterations 
\begin{equation}\label{inversedyn}
x_{i+1}=y_i-\omega_i,
\qquad
y_{i+1}=f(y_i)-x_i.
\end{equation}
Its associated projective dynamics then satisfies
\[
\tan\theta_{i+1}=f'(y_i)-\cot\theta_i.
\]

Let $\mathcal P_p^-$ denote the Koopman operator of the inverse projective process. For each initial state
$(x,y,\theta)\in\mathbb{PT}^2$, let
$\bar{\mathcal P}^{-,n}_{p,(x,y,\theta)}$ denote the distribution of this process after $n$ steps.

\begin{lema}\label{inverseone}
Fix $(x_0,y_0,\theta_0)\in\mathbb{PT}^2$, and define
\begin{equation}\label{hacca}
H^-(\omega_0,\omega_1,\omega_2):= f^{-1}_{\omega_2}\circ f^{-1}_{\omega_1} \circ f_{\omega_0}^{-1}(x_0,y_0,\theta_0)=(x_3,y_3,\theta_3),
\end{equation}
Then, if $\theta_2\neq0$, we have that
\begin{equation}\label{density}
\frac{
d\bar{\mathcal P}^{-,3}_{p,(x_0,y_0,\theta_0)}
}{
dx\,dy\,d\theta
}(x_3,y_3,\theta_3)
=
\frac{L^{3\delta}}{8}
\left(
\sum_{\omega_0\in\mathcal E^-(x_3,y_3,\theta_3)}
\frac{1}{|f''(B(x_0,y_0)+\omega_0)|}
\right)
\frac{1}{\cos^2\theta_3},
\end{equation}
where
\[
B(x_0,y_0):=f\bigl(f(y_0)-x_0\bigr)-y_0,
\]
and $\mathcal E^-(x_3,y_3,\theta_3)$ is the set of values
of $\omega_0$ for which \eqref{hacca} admits a solution.
Moreover, $\#\mathcal E^-(x_3,y_3,\theta_3)\le2$.
\end{lema}

\begin{proof}
A direct computation gives
\[
y_1=f(y_0)-x_0,
\qquad
y_2=f(y_1)-y_0+\omega_0=B(x_0,y_0)+\omega_0.
\]
Thus
\[
x_3=y_2-\omega_2,
\qquad
y_3=f(y_2)-y_1+\omega_1,
\qquad
\tan\theta_3=f'(y_2)-\cot\theta_2,
\]
where $\theta_2$ is deterministic.
Using \eqref{hacca}, the Jacobian matrix is
\[
DH^-=
\begin{pmatrix}
1&0&-1\\
f'(y_2)&1&0\\
\partial_{y_2}\theta_3&0&0
\end{pmatrix}.
\]
Differentiating
\[
\tan\theta_3=f'(y_2)-\cot\theta_2
\]
with respect to $y_2$ gives, assuming $\theta_2\neq0$,
\[
\det DH^-(\omega_0,\omega_1,\omega_2)
=
f''(y_2)\cos^2\theta_3.
\]
This proves \eqref{density}.

To conclude, note that
$\tan\theta_3=f'(y_2)-\cot\theta_2$
determines $y_2$ up to at most two solutions, and each $y_2$
uniquely determines
\[
\omega_0=y_2-B(x_0,y_0),
\qquad
\omega_2=y_2-x_3,
\qquad
\omega_1=y_3-f(y_2)+y_1.
\]
\end{proof}

The argument used to prove Proposition~\ref{proposition1},
together with the inverse density formula in
Lemma~\ref{inverseone}, gives the following estimate.

\begin{proposicao}\label{vertconc}
There exist constants $\bar s,\bar r>0$ such that, for all
sufficiently large $L$, every stationary measure $\nu^-$
of the inverse projective process satisfies
\begin{equation}\label{invfinal}
\nu^-\Bigl\{
(x,y,\theta):\theta\notin B_{L^{-\bar s}}\Bigl(\frac{\pi}{2}\Bigr)
\Bigr\}
\le L^{-\bar r}.
\end{equation}
\end{proposicao}

We leave the details to the reader.

By Oseledets' theorem, the unit determinant of our maps, and the
positivity of the top Lyapunov exponent established in
\cite{BXY17}, the random standard map admits an almost surely
defined, future-measurable stable direction $e^s(\omega,x,y)\in\mathbb P^1.$

The following result, which is classical (see \cite{Viana14}),
connects the projective motion of the inverse dynamics with the
distribution of the stable direction.

\begin{teorema}\label{corr}
Define a probability measure $\nu^s$ on $\mathbb{PT}^2$ by
\begin{equation}\label{nuesse}
\nu^s(B)
:=
(\mathbb P\otimes m)
\bigl\{
(\omega,x,y):(x,y,e^s(\omega,x,y))\in B
\bigr\}.    
\end{equation}
Then, for every bounded measurable function
$\phi\colon\mathbb{PT}^2\to\mathbb R$,
\[
\int_{\mathbb{PT}^2}\mathcal P_p^-\phi\,d\nu^s
=
\int_{\mathbb{PT}^2}\phi\,d\nu^s.
\]
In particular, $\nu^s$ is stationary for the inverse projective process.
\end{teorema}

We now apply Proposition~\ref{vertconc} to estimate the
probability that $e^s$ lies outside a neighbourhood of the
vertical direction.

\begin{corolario}\label{Stabledef}
Let $\bar s,\bar r>0$ be as in Proposition~\ref{vertconc}. Then
\begin{equation}\label{jesus}
(\mathbb P\otimes m)
\Bigl\{
(\omega,x,y):
\operatorname{dist}\bigl(e^s(\omega,x,y),\tfrac{\pi}{2}\bigr)
\ge L^{-\bar s}
\Bigr\}
\le L^{-\bar r}.
\end{equation}

Consequently, by Birkhoff's ergodic theorem, for
$(\mathbb P\otimes m)$-almost every $(\omega,x,y)$,
\begin{equation}\label{christ}
\limsup_{n\to\infty}
\frac{
\left|
\left\{
1\le k\le n:
\operatorname{dist}\bigl(
e^s(\Theta^k(\omega,x,y)),\tfrac{\pi}{2}
\bigr)
\ge L^{-\bar s}
\right\}
\right|
}{n}
\le L^{-\bar r},
\end{equation}
where $\Theta:\Omega\times\mathbb T^2\to\Omega\times\mathbb T^2$
denotes the skew product map $\Theta(\omega,x,y)=\bigl(\vartheta(\omega),f_\omega(x,y)\bigr).$
\end{corolario}

\begin{proof}
Let $\nu^s$ be as in \eqref{nuesse}. By Theorem~\ref{corr}, $\nu^s$ is stationary for the inverse
projective process. Formula \eqref{jesus} follows by \eqref{invfinal} and \eqref{nuesse} whilst \eqref{christ} follows from the ergodicity of $\Theta$ with respect to $\mathbb P\otimes m$.
\end{proof}

\section{Growth to large scale in the random standard map}\label{revtsec}

In this section, we build on the techniques developed in \cite{DD24} to quantify how frequently small curves grow to
large scale and to determine the scale they reach.

Throughout this section, we work with the third iterate
\begin{equation}\label{thirditerate}
F_\omega:=f_\omega^3,
\end{equation}

where $f_\omega$ is defined in \eqref{maps}. Given $\omega \in \Omega$ and $z\in \mathbb{T}^2,$  write 
\begin{equation}\label{sequenceofmatrices}
A_n(\omega,z) = DF^n_{\omega}(z).
\end{equation}
We set $A_0(\omega,z)=\mathrm{Id}$.
We write $A_n=A_n(\omega,z)$ when the dependence on $\omega$ and $z$
is understood.

We first revisit the notions of reverse-tempered norms and
reverse-tempered splittings, introduced in \cite{DD24}. There, they are formulated in terms of the inverse cocycle; here we use a formulation involving the forward derivative products in \eqref{sequenceofmatrices}.

\begin{definicao}\label{reversetemperedness}
Let $R\ge0$ and $\alpha,\varepsilon>0$. We say that $N\ge1$ is an
$(R,\alpha,\varepsilon)$-reverse-tempered time for the sequence
$\{A_n(\omega,z)\}_{n=0}^N$ if
\begin{equation}\label{eq:reverse-temperedness}
\|A_m\|
\ge e^{-R}
L^{-\varepsilon(N-m)}
L^{\alpha(m-k)}
\|A_k\|,
\qquad 0\le k\le m\le N.
\end{equation}
\end{definicao}

\begin{definicao}\label{def:reverse-tempered-splitting}
We say that the sequence $\{A_n(\omega,z)\}_{n=0}^N$ admits a
$(C,\alpha,\varepsilon)$-reverse-tempered splitting if there exist
unit vectors $s_N,u_N\in \mathbb {PT}^2$ such that, for every
$0\le k\le m\le N$,
\begin{equation}\label{eq:stable-splitting}
\frac{\|A_m s_N\|}{\|A_k s_N\|}
\le
e^{C}L^3L^{3\varepsilon(N-k)}L^{-\alpha(m-k)},
\end{equation}
\begin{equation}\label{eq:unstable-splitting}
\frac{\|A_m u_N\|}{\|A_k u_N\|}
\ge
e^{-C}L^{-3}L^{-3\varepsilon(N-k)}L^{\alpha(m-k)},
\end{equation}
and
\begin{equation}\label{eq:angle-splitting}
\angle(A_m s_N,A_m u_N)
\ge
e^{-C}L^{-3}L^{-3\varepsilon(N-k)}.
\end{equation}
\end{definicao}

The next Proposition establishes that if $N$ is a
reverse-tempered time for the sequence $\{A_n\}_{n=0}^N$, then this sequence admits a reverse-tempered splitting. This is essentially \cite[Proposition 4.6]{DD24}.

\begin{proposicao}\label{thm1}
Fix $R\ge0$ and $\alpha>0$. There exist $C_*>0$ and
$\bar\varepsilon>0$ such that, for every
$0<\varepsilon<\bar\varepsilon$, there exists $L_0>0$
with the following property: for every $L>L_0$, if $N$ is an
$(R,\alpha,\varepsilon)$-reverse-tempered time for the sequence
$\{A_n(\omega,z)\}_{n=0}^N$, then this sequence admits a
$(3R+C_*,\alpha,\varepsilon)$-reverse-tempered splitting.
The vectors $s_N,u_N$ can be chosen orthogonal, with $s_N$
a most contracted singular direction of $A_N$.
\end{proposicao}

The next proposition is an application of \cite[Proposition A.13]{DD24}, and   shows that a curve sufficiently transverse to the
stable direction grows to a scale bounded below by a fixed power of
$L^{-1}$, while its curvature remains bounded by a fixed power of $L$.

\begin{proposicao}\label{hyperbolict_times}
Fix $C\ge0$, $\alpha>0$, and $\ell_0,K_0>0$. There exist
$\varepsilon_0,\varepsilon',a,b>0$ and $n_0,L_0\ge1$ such that the
following holds whenever $L\ge L_0$ and
$0<\varepsilon<\varepsilon_0$.

Suppose that, for some $\omega \in \Omega,$ $z \in 
\mathbb{T}^2,$ and $N\ge n_0$ the sequence of matrices $\{A_n(\omega,z)\}_{n=0}^N$ as in \eqref{sequenceofmatrices} admits a
$(C,\alpha,\varepsilon)$-reverse-tempered splitting, with stable and
unstable directions $s_N,u_N$, respectively. Let
\[
\gamma:[-\ell_0,\ell_0]\longrightarrow\mathbb T^2
\]
be an embedded $C^2$ curve parametrized by arclength, satisfying
\[
\gamma(0)=z,\qquad
\|\gamma''\|_\infty\le K_0,\qquad
\angle(\gamma'(0),s_N)\ge L^{-\varepsilon'N}.
\]
Then $F_\omega^N(\gamma)$ contains a subcurve $\gamma_N$ through $F_\omega^N(z)$ which, when parametrized by arclength, satisfies
$\|\gamma_N''\|_\infty\le L^a$ and extends for a length of at least
$L^{-b}$ on each side of $F_\omega^N(z)$.
\end{proposicao}
\begin{proof}
We follow the proof of \cite[Proposition~A.13]{DD24},
keeping track of the dependence on $L$. Throughout,
$K$ denotes constants independent of $L,N,\varepsilon$.
Applying \cite[Lemma~A.1]{DD24} using our quantitative splitting, we obtain charts $L_i$ satisfying
\begin{equation}\label{godhelpmepls}
\|L_i^{-1}\|\le\sqrt2,
\qquad
\|L_i\|\le KL^6L^{6\varepsilon(N-i)},    
\end{equation}
in place of \cite[(A.36)]{DD24}.  Writing $G_i$ for
the centred maps in these coordinates, the above estimates
together with $\|D^2F_\omega\|_\infty\le KL^5$ and the
computation in \cite[Lemma~A.2]{DD24} give $\|D^2G_i\|_\infty
\le KL^{23}L^{18\varepsilon(N-i)}$ as a  coarse counterpart of \cite[(A.37)]{DD24}.  To obtain the bound on the nonlinear terms of $G_i$ corresponding to \cite[(A.42)]{DD24}, we replace the radii
in \cite[(A.38)]{DD24} by $r_i=L^{-B}L^{-.9\alpha(N-i)},$ for 
$0 \le i\le N$ and  $B>0$ sufficiently large. With this choice of radii, the graph-transform argument
in Step~1 of \cite[Proposition~A.13]{DD24} applies,
yielding a final curve of size comparable to $L^{-B}$
in the adapted coordinates, and comparable to a larger power of $L^{-1}$ in the original coordinates by \eqref{godhelpmepls} (for further details, see the analogous graph-transform argument in Appendix~G).

To prove the estimate on the $C^2$ norm, we follow Step~2 of
\cite[Proposition~A.13]{DD24}. By choosing $\varepsilon,\varepsilon'$ sufficiently small and then
$n_0$ sufficiently large, we may restrict to the first case
considered there, namely that there exists $N_0\ll N$ such that, for every $N_0\le i\le N$, the retained subcurve of
$F_\omega^i(\gamma)$ is represented in the chart at time $i$ by  graph with slope at most the fixed cutoff $\varepsilon_0$ in \cite[Proposition~A.12]{DD24}. By the slope estimate corresponding to \cite[(A.44)]{DD24}, for all sufficiently
large $L$ we may take
\[
N_0=\left\lceil
\frac{2}{\alpha}\bigl(7+(6\varepsilon+\varepsilon')N\bigr)
\right\rceil,
\]
which is independent of $L$.  In the curvature estimate, the dependence on
$L$ enters through the constant $\Lambda$ in \cite[(A.48)]{DD24}.
Since the derivative bounds for $F_\omega$ and its inverse are
polynomial in $L$, we may take $\Lambda=O(\log L)$. Following then the computation in
\cite[(A.47)--(A.49)]{DD24} and returning to the original coordinates by \eqref{godhelpmepls}, we obtain the desired curvature bound.
\end{proof}
Fix $0<\alpha<24/25$, $\ell_0,K_0>0$ and $C>0$.
Let $\varepsilon_0>0$ be the constant provided by
Proposition~\ref{hyperbolict_times} for these parameter
choices, and choose $0<\varepsilon<\varepsilon_0$.
 Given  an embedded $C^2$ curve
$\gamma:[-\ell_0,\ell_0]\to\mathbb T^2$, parametrized by arclength, with $\gamma(0)=z$ and $\|\gamma''\|_\infty\le K_0$, set
$\tau_0(\gamma)=0$ and define recursively
\begin{equation}\label{taugamma}
\tau_j(\gamma)(\omega):=
\inf\left\{
m>\tau_{j-1}(\gamma)(\omega):
\begin{array}{l}
m\ge n_0,\quad
\{A_n(\omega,z)\}_{n=0}^{m}\text{ admits a }
(C,\alpha,\varepsilon)\text{-reverse-tempered splitting}\\[1mm]
\text{with most contracting direction }s_m
\text{ such that }
\angle(\gamma'(0),s_m)\ge L^{-\varepsilon'm}
\end{array}
\right\},
\end{equation}
where  $n_0,\varepsilon'$ are obtained from Proposition \ref{hyperbolict_times} after fixing $\alpha,C,\ell_0,K_0$,

By Proposition~\ref{hyperbolict_times}, for every $j\ge1$
such that $\tau_j(\gamma)<\infty$, there exists a subcurve
$\gamma_j\subset\gamma$ containing $z$ such that
$F_\omega^{\tau_j(\gamma)}(\gamma_j)$ extends for a length
of at least $L^{-b}$ on each side of
$F_\omega^{\tau_j(\gamma)}(z)$ and has curvature bounded
by $L^a$, where $a,b$ are as in
Proposition~\ref{hyperbolict_times}.

The remainder of this section is devoted to the proof of the following theorem.

\begin{teorema}\label{thmgrowth}
There exist $0<\alpha<24/25,$ $C>0$ such that, for any $l_0,K_0>0,$ one can choose $\varepsilon>0$ such that the following holds: for any fixed $\delta,\delta_0>0$ sufficiently small, there exist constants $c_{\delta_0}>0$,
$L_0\ge1$, and $n_{\delta_0}\in\mathbb N$ with the following property. For every embedded $C^2$ curve
$\gamma:[-\ell_0,\ell_0]\to\mathbb T^2$, parametrized by arclength,
with $\gamma(0)=z$ and $\|\gamma''\|_\infty\le K_0$,
the times $\tau_k(\gamma)$ defined in \eqref{taugamma} using the choices of $C,\alpha,\varepsilon$ above satisfy
\begin{equation}\label{putafkcnformula}
\mathbb P\left(
\#\{k\ge1:\tau_k(\gamma)\le n\}
\le (1-\delta_0)n
\right)
\le L^{-c_{\delta_0}n},
\end{equation}
for all $L\ge L_0$ and $n\ge n_{\delta_0}$.
\end{teorema}

For each prescribed choice of $\ell_0,K_0$, the constants
$\alpha,C,\varepsilon$ are chosen as specified in the proof
of Theorem~\ref{thmgrowth}.
The constants $n_0$ and $\varepsilon'$ are obtained from
Proposition~\ref{hyperbolict_times}, with $\varepsilon'$ decreased
if necessary, completing the choice of parameters defining $\tau_j$.

The proof of Theorem~\ref{thmgrowth} proceeds in two steps:
\begin{itemize}
\item
In Subsection~\ref{tailofrev}, Proposition~\ref{prop:first-net-tail} establishes exponential tail for  the stopping times defined
in \eqref{def:first-net-good-time}. These are a modification of the reverse-tempered times introduced in \cite{DD24}. As a consequence of Proposition~\ref{prop:first-net-tail},
the waiting time for the first
$(C',\alpha,\varepsilon)$-reverse-tempered time has an
exponential tail, for some fixed $C'>0$.
\item
In Subsection~\ref{concludethmgrowth}, we prove
Lemma~\ref{lem:S-j-tail}, which establishes exponential
tail estimates for the stopping times in
\eqref{eq:S-j-def-correct}. At such times, the derivative sequence starting at $z$ is reverse tempered, and its most contracting direction is sufficiently separated from a prescribed
direction $v$. Theorem \ref{thmgrowth} then follows choosing $z=\gamma(0),$ $v=\gamma'(0)$ and applying an exponential moment argument.
\end{itemize}

\subsection{Tail of reverse-tempered times}\label{tailofrev}

Let us fix a subset  $\mathcal N\subset\mathbb P^1$ consisting of eight equally spaced
directions satisfying
\[
\min_{w\in\mathcal N}\angle(v,w)\le\frac{\pi}{16}
\quad\text{for every }v\in\mathbb P^1.
\]
\begin{definicao}\label{def:net-good-time}
For $\omega\in\Omega$, $z\in\mathbb T^2$, and $N\ge1$,
we say that $N$ is a
$(\mathcal N;C,\alpha,\varepsilon)$-reverse tempered time for the sequence $\{A_i(\omega,z)\}_{i=0}^N$  if,
for every $v\in\mathcal N$ and every $0\le k<m\le N$,
\[
\|A_m(\omega,z)v\|
\ge
e^C L^{-\varepsilon(N-m)}
L^{\alpha(m-k)}
\|A_k(\omega,z)v\|.
\]
When $C=0$, we say that $N$ is an
$(\mathcal N;\alpha,\varepsilon)$-reverse-tempered time.
\end{definicao}
The importance of such times is motivated by the following lemma, whose proof is left to the reader.
\begin{lema}\label{lem:net-implies-rev-temp}
Assume that $n$ is an
$(\mathcal N;\alpha,\varepsilon)$-reverse-tempered time for
$\{A_i(\omega,z)\}_{i=0}^n.$ Then there exists $C'>0$ such that
$n$ is also $(C',\alpha,\varepsilon)$-reverse-tempered
time as in Definition \ref{reversetemperedness}
\end{lema}
For $\omega \in \Omega,$ $z\in \mathbb{T}^2,$ and an integer
$N\ge0$, define
\begin{equation}\label{def:first-net-good-time}
\bar\tau_1(N)(\omega,z)
:=
\inf\left\{
n>N:
n\text{ is an }(\mathcal N;\alpha,\varepsilon)
\text{-reverse-tempered time for $\{A_i(\omega,z)\}_{i=0}^n$}
\right\}.
\end{equation}
\begin{proposicao}\label{prop:first-net-tail}
Let $\alpha,\varepsilon>0$ satisfy
\begin{equation}\label{admissibleparams}
0<\varepsilon<\frac{24}{25}-\alpha.
\end{equation}
There exist constants $\bar A\ge1$ and
$\bar r_1,\bar r_2>0$ such that, for $\delta>0$ sufficiently
small, all sufficiently large $L$, and all integers $m,n\ge0$
and $D_0\ge0$,
\begin{equation}\label{vindipldev}
\mathbb P\left(
\bar\tau_1(m)>m+n
\,\middle|\,
\substack{
\{A_i(\omega,z)\}_{i=0}^{m}\text{ is}\\
(\mathcal N;-D_0,\alpha,\varepsilon)\text{-reverse tempered}
}
\right)
\le
\bar A L^{-\bar r_1 n}e^{\bar r_2D_0}.
\end{equation}
\end{proposicao}

To prove Proposition~\ref{prop:first-net-tail}, we track how
the reverse-temperedness constant changes when one more matrix
is appended to the sequence.
In order to do that, we use the following lemma, an immediate
consequence of \cite[Lemma~4.15]{DD24}.

\begin{lema}\label{lem:concat-rev-temp}
Suppose that $\{A_m(\omega,z)\}_{m=0}^M$ is
$(\mathcal N;C,\alpha,\varepsilon)$-reverse tempered,
with $C\le0$, and that
\begin{equation}\label{eq:singleton-c}
\|A_{M+1}(\omega,z)v\|
\ge
L^{\alpha+c}\|A_M(\omega,z)v\|,
\qquad v\in\mathcal N,
\end{equation}
for some $c\in\mathbb R$.
Then $\{A_m(\omega,z)\}_{m=0}^{M+1}$ is
$(\mathcal N;C_+,\alpha,\varepsilon)$-reverse tempered,
where
\begin{equation}\label{eq:concat-recursion-simple}
C_+
:=
\min\{C+\varepsilon\log L,\ C+c\log L\}.
\end{equation}
\end{lema}

\begin{proof}[Proof of Proposition~\ref{prop:first-net-tail}]
Given $z\in\mathbb T^2$, let $\alpha,\varepsilon>0$ satisfy
\eqref{admissibleparams}, and let $\bar\tau_1(m)$ be as in
\eqref{def:first-net-good-time}. Suppose that the sequence
$\{A_i(\omega,z)\}_{i=0}^{m}$ is
$(\mathcal N;-D_0,\alpha,\varepsilon)$-reverse tempered,
for some $D_0>0$.
Let $C_0=-D_0$ and, for $j\ge1$, let $C_j$ denote the
optimal reverse-temperedness constant for the sequence
$\{A_i(\omega,z)\}_{i=0}^{m+j}$ with respect to $\mathcal N$.
If
\begin{equation}\label{eq:good-step}
\|A_{m+j+1}(\omega,z)v\|
\ge L^{24/25}\|A_{m+j}(\omega,z)v\|
\qquad \forall v\in\mathcal N,
\end{equation}
by \eqref{eq:concat-recursion-simple} we have
\begin{equation}\label{eq:good-recursion}
C_{j+1}
\ge
C_j+\varepsilon\log L.
\end{equation}
Otherwise, we use the lower bound
\begin{equation}\label{eq:bad-step}
\|A_{m+j+1}(\omega,z)v\|
\ge L^{-4}\|A_{m+j}(\omega,z)v\|
\qquad \forall v\in\mathcal N,
\end{equation}
leading to
\begin{equation}\label{eq:bad-recursion-simple}
C_{j+1}
\ge
C_j-(4+\alpha)\log L.
\end{equation}
By \eqref{elementary}, for some constant $K>0$,
\begin{equation}\label{preelementary}
\mathbb P\left(
\|A_{m+j+1}(\omega,z)v\|
\ge L^{24/25}\|A_{m+j}(\omega,z)v\|
\quad \forall v\in\mathcal N
\,\middle|\,
\mathcal F_{m+j}
\right)
\ge
1-K L^{\delta-\frac{1}{100}},
\end{equation}
with
\[
\mathcal F_m:=\sigma(\omega_0,\ldots,\omega_{3m-1}),
\qquad m\ge1.
\]
Let $Q_n$ be the number of times, between $m+1$ and $m+n$,
the event in \eqref{eq:good-step} does not happen.
Since $C_0=-D_0$, the sequence $D_i:=-C_i$ satisfies,
for $n\ge1$, on the event $\{\bar\tau_1(m)>m+n\}$,
\begin{equation}\label{eq:defect-bound}
0\le D_n
\le
D_0-(n-Q_n)\varepsilon\log L+Q_n(4+\alpha)\log L.
\end{equation}
Consequently,
\begin{equation}\label{eq:bad-count-threshold}
Q_n
>
\frac{\varepsilon}{4+\alpha+\varepsilon}n
-
\frac{D_0}{(4+\alpha+\varepsilon)\log L}.
\end{equation}
The probabilistic estimate in \eqref{preelementary}, along with
an exponential moments argument, implies \eqref{vindipldev}.
\end{proof}

\subsection{Concluding the proof of Theorem \ref{thmgrowth}}\label{concludethmgrowth}

Set $S_0=0$ and, for $v\in\mathbb P^1$, define recursively 
\begin{equation}\label{eq:S-j-def-correct}
S_j(\omega,z,v)
:=
\inf\left\{
m>\max\{S_{j-1},N_0\}:
\begin{array}{l}
m\text{ is an }(\mathcal N;\alpha,\varepsilon)
\text{-reverse-tempered time,}\\[2mm]
\|A_m(\omega,z)w\|\ge L^{a_*m}
\text{ for every }w\in B_{L^{-\varepsilon'm}}(v)
\end{array}
\right\},
\end{equation}
for some given $N_0 \in \mathbb{N}$, $\varepsilon'>0$, and $a_*\in(0,24/25)$ is the constant provided by
Proposition~\ref{prop:strong-direction} below.
Thus $S_j$ is the first time after $\max\{S_{j-1},N_0\}$ to be a $(\mathcal N;\alpha,\varepsilon)$-reverse-tempered time for the sequence $\{A_n(\omega,z)\}_{n=0}^{S_j}$ and such that every direction in the ball $B_{L^{-\varepsilon'S_j}}(v)$ has expanded by at least $L^{a_*S_j}$. 

For these times we have the following tail estimates.

\begin{lema}\label{lem:S-j-tail}
Let $\alpha,\varepsilon$ as in \eqref{admissibleparams} and , given $c_0$  as in Proposition~\ref{prop:strong-direction} choose
$\varepsilon'$ sufficiently small that
\begin{equation}\label{eq:delta-condition-final}
0<c_0\varepsilon'<1.
\end{equation}
Then, for $N_0$ sufficiently large, there exist constants $C,\eta>0$ such that, for all sufficiently large $L$
and every $i\ge1$,
\begin{equation}\label{eq:S1-tail}
\mathbb P(S_1>N_0+i)\le CL^{-\eta i},
\end{equation}
and, for every $j\ge2$,
\begin{equation}\label{eq:S-j-tail}
\mathbb P\left(
S_j>S_{j-1}+i
\,\middle|\,
\mathcal F_{S_{j-1}}
\right)
\le
CL^{-\eta i}.
\end{equation}
\end{lema}

We first show how Lemma~\ref{lem:S-j-tail} implies
Theorem~\ref{thmgrowth}.

\begin{proof}[Proof of Theorem \ref{thmgrowth}]
Fix $K_0,\ell_0$ and a curve $\gamma$ satisfying the conditions in the statement of Theorem \ref{thmgrowth}. Choose $0<\alpha<24/25$, let  $C>3C'+C_*$, where $C'$
is as in Lemma~\ref{lem:net-implies-rev-temp}, and $C_*$ is given, along with $\bar \varepsilon_0,$  applying Proposition~\ref{thm1} with our previous choice of $\alpha$ and $R=C'$.  Choose $\varepsilon<\bar \varepsilon_0$ satisfying \eqref{admissibleparams} and let $n_0,\varepsilon',\varepsilon_0$ be the
constants provided by Proposition~\ref{hyperbolict_times} for the
above choices of $\ell_0,K_0,\alpha,C$. By decreasing $\varepsilon$
and $\varepsilon'$ if necessary, we may assume that
$\varepsilon<\varepsilon_0$ and that $\varepsilon'$ satisfies
\eqref{eq:delta-condition-final}.

Set $v:=\gamma'(0)$ and $z:=\gamma(0)$.
Lemma~\ref{lem:S-j-tail} estimates the tail of 
$\{S_j(\omega,z,v)\}_{j=0}^{\infty }$ satisfying \eqref{eq:S-j-def-correct},
with $\varepsilon,\alpha,\varepsilon'$ as fixed above and
$N_0\ge n_0$ chosen sufficiently large. 

By Lemma~\ref{lem:net-implies-rev-temp} and
Proposition~\ref{thm1}, the derivative sequence at each $S_j$,
$j\ge1$, admits a $(C,\alpha,\varepsilon)$-reverse-tempered
splitting with most contracting direction $s_{S_j}$.
Since area preservation gives $\|A_{S_j}s_{S_j}\|\le1$,
the expansion condition in \eqref{eq:S-j-def-correct} implies
that $s_{S_j}\notin B_{L^{-\varepsilon'S_j}}(v)$, and hence
\[
\angle(\gamma'(0),s_{S_j})\ge L^{-\varepsilon'S_j}.
\]
As $S_j>N_0\ge n_0$, every $S_j$ is therefore one of the
times $\tau_k(\gamma)$, so that
\begin{align*}
\#\{j \colon \tau_j \le n\} \ge \#\{j \colon S_j \le n\}.
\end{align*}

Formula \eqref{putafkcnformula} then follows from a standard
exponential-moment argument using the tail estimates in
Lemma~\ref{lem:S-j-tail}.
\end{proof}

To prove Lemma ~\ref{lem:S-j-tail}, we need the following Proposition, which controls the regularity of the stable direction. Its proof follows the argument of \cite[Lemma~4.13 and Proposition~4.14]{DD24}, using our large-deviation
estimate, Theorem \ref{lyapunov_large_deviations}, in place of
\cite[Proposition~4.4]{DD24}.

\begin{proposicao}\label{prop:strong-direction}
There exist $a_*\in(0,24/25)$ and $c_0,h,L_0,\beta_0>0$ such that,
for every $\beta\ge\beta_0$, if $n\ge c_0\beta,$ then, for all $L>L_0$, we have
\begin{equation}\label{eq:strong-direction-bad}
\sup_{z\in\mathbb T^2}\sup_{v\in\mathbb P^1}
\mathbb P\left(
\exists w\in B_{L^{-\beta}}(v):
\|A_n(\omega,z)w\|<L^{a_*n}
\right)
\le
2L^{-h\beta}.
\end{equation}
\end{proposicao}

We are now ready to prove ~\ref{lem:S-j-tail}.

\begin{proof}[Proof of Lemma \ref{lem:S-j-tail}]

For $S_1$, formula \eqref{eq:S1-tail} follows from \cite[Proposition~8.4]{DD24}, using
\eqref{vindipldev} and Proposition~\ref{prop:strong-direction}
in place of their recovery and direction estimates, respectively, and $N_0$ sufficiently large that $\varepsilon'N_0\ge\beta_0.$

Fix $j\ge2$. Define
\[
s_i^{(j)}
:=
\inf\left\{
m\ge S_{j-1}+\left\lfloor\frac i2\right\rfloor+1:
m\text{ is an }(\mathcal N;\alpha,\varepsilon)
\text{-reverse-tempered time}
\right\}.
\]
Then, by the properties of $S_{j-1}$ in \eqref{eq:S-j-def-correct}, 
formula \eqref{vindipldev} applies with $D_0=0$ and yields
\begin{equation}\label{eq:sij-tail}
\begin{aligned}
\mathbb P\left(
s_i^{(j)}>S_{j-1}+i
\,\middle|\,
\mathcal F_{S_{j-1}}
\right)
&\le
\bar A\sum_{\ell=0}^{\lfloor i/2\rfloor}
L^{-\bar r_1(i-\ell)}\\
&\le C_1L^{-\bar r_1i/2},
\end{aligned}
\end{equation}
for some $C_1>0.$
If $S_j>S_{j-1}+i$, then either no $(\mathcal N; \alpha,\varepsilon)$-reverse-tempered time occurs in the window  $\left(S_{j-1}+\frac i2,\;S_{j-1}+i\right],$ or such a time occurs but the ball-expansion condition fails.
Therefore
\begin{equation}\label{eq:union-bound-ball-expansion}
\begin{aligned}
\mathbb P\left(
S_j>S_{j-1}+i
\,\middle|\,
\mathcal F_{S_{j-1}}
\right)
&\le
\mathbb P\left(
s_i^{(j)}>S_{j-1}+i
\,\middle|\,
\mathcal F_{S_{j-1}}
\right)
\\[1mm]
&\quad+
\mathbb P\left(
\begin{array}{l}
\exists w\in B_{L^{-\varepsilon's_i^{(j)}}}(v):
\|A_{s_i^{(j)}}w\|<L^{a_*s_i^{(j)}},\\[1mm]
s_i^{(j)}\le S_{j-1}+i
\end{array}
\,\middle|\,
\mathcal F_{S_{j-1}}
\right)
\\[1mm]
&\le C_1 L^{-\bar r_1 i/2}
\\[1mm]
&\quad+
\sum_{m=\lfloor i/2\rfloor+1}^{i}
\mathbb P\left(
\begin{array}{l}
\exists w\in B_{L^{-\varepsilon'(m+S_{j-1})}}(v):\\[1mm]
\|A_{m+S_{j-1}}w\|<L^{a_*(m+S_{j-1})}
\end{array}
\,\middle|\,
\mathcal F_{S_{j-1}}
\right).
\end{aligned}
\end{equation}
where in the last inequality we used \eqref{eq:sij-tail}. We now estimate the second term. Write
\[
v_{S_{j-1}}:=\widehat A_{S_{j-1}}(v),
\qquad
A_m^{(S_{j-1})}:=A_m\bigl(\vartheta^{3S_{j-1}}\omega,F_\omega^{S_{j-1}}(z)\bigr),
\]
where $\widehat A_{S_{j-1}}$ denotes the projective action of
$A_{S_{j-1}}.$
By \eqref{eq:S-j-def-correct}, $A_{S_{j-1}}$ expands every
direction in $B_{L^{-\varepsilon'S_{j-1}}}(v)$ by at least
$L^{a_*S_{j-1}}$. Since it preserves the areas of triangles
spanned by pairs of vectors, its projective action contracts
angular distances on this ball by a factor of at most
$C_2L^{-2a_*S_{j-1}}$, for some $C_2>0$. Consequently,
\[
\widehat A_{S_{j-1}}\left(B_{L^{-\varepsilon'(m+S_{j-1})}}(v)\right)
\subset
B_{C_2L^{-\varepsilon'm-(\varepsilon'+2a_*)S_{j-1}}}(v_{S_{j-1}}),
\]
and
\[
\|A_{S_{j-1}}w\|
\ge
L^{a_*S_{j-1}}>1,
\qquad
\text{for every }w\in B_{L^{-\varepsilon'(m+S_{j-1})}}(v).
\]
Hence
\begin{equation}\label{tbf}
\begin{aligned}
&\mathbb P\left(
\exists w\in B_{L^{-\varepsilon'(m+S_{j-1})}}(v):
\|A_{m+S_{j-1}}w\|<L^{a_*(m+S_{j-1})}
\,\middle|\,
\mathcal F_{S_{j-1}}
\right)\\
&\qquad\le
\mathbb P\left(
\exists w\in B_{C_2L^{-\varepsilon'm-(\varepsilon'+2a_*)S_{j-1}}}(v_{S_{j-1}}):
\|A_m^{(S_{j-1})}w\|<L^{a_*m}
\,\middle|\,
\mathcal F_{S_{j-1}}
\right).
\end{aligned}
\end{equation}
Let $\beta_0,c_0$ as in Proposition \ref{prop:strong-direction}
If $m\ge m_*:=\lceil c_0\beta_0\rceil$, set $\bar\beta:=\max\{\beta_0,\varepsilon'm\}.$ Since $S_{j-1}\ge N_0$, for $N_0$ and $L$ sufficiently large we have
\[
B_{C_2L^{-\varepsilon'm-(\varepsilon'+2a_*)S_{j-1}}}(v_{S_{j-1}})
\subset B_{L^{-\bar \beta}}(v_{S_{j-1}}).
\]
Proposition~\ref{prop:strong-direction}, applied with
$\beta =\bar \beta$ and $n=m,$ therefore bounds the RHS of
\eqref{tbf} by $2L^{-h\varepsilon'm}$.

For $1\le m<m_*$, choose $a_*<b_*<24/25$.
By Theorem~\ref{lyapunov_large_deviations}, conditionally on
$\mathcal F_{S_{j-1}}$, we have
\[
\|A_m^{(S_{j-1})}v_{S_{j-1}}\|\ge L^{b_*m}
\]
except on a set of probability at most $L^{-rm}$, for some $r>0$.
Using $\|A_m^{(S_{j-1})}\|\le L^{4m}$, we obtain
\[
\|A_m^{(S_{j-1})}w\|
\ge
L^{b_*m}
-
C_2L^{4m-\varepsilon'm-(\varepsilon'+2a_*)S_{j-1}}
\ge
L^{a_*m}
\]
for all $w$ in the transported neighborhood of $v_{S_{j-1}}$, for $N_0$ and $L$ sufficiently
large and $m<m_*$

Thus each term in \eqref{eq:union-bound-ball-expansion} is bounded
by $2L^{-\eta_0m}$, for some $\eta_0>0$. Combining this with
\eqref{eq:sij-tail} and \eqref{eq:union-bound-ball-expansion}, we obtain
\[
\mathbb P\left(
S_j>S_{j-1}+i
\,\middle|\,
\mathcal F_{S_{j-1}}
\right)
\le
C_1L^{-\bar r_1i/2}
+
2\sum_{m=\lfloor i/2\rfloor+1}^{i}L^{-\eta_0m}
\le
CL^{-\eta i},
\]
for some $C>0.$
\end{proof}

\section{Abundance of typical horseshoes}\label{abundanceset}

In this section we prove Theorem~\ref{thm:typical-horseshoes}. To do so, we first establish the asymptotic frequency of two
families of random times: \emph{Young times} and
\emph{stable times}. We then combine their statistical
and dynamical properties to construct the random horseshoes. Since the construction involves several steps, before introducing the precise definitions and statements we briefly explain the role of Young and stable times and how they are combined to prove Theorem~\ref{thm:typical-horseshoes}.

Throughout this  discussion, time is measured in iterates of
$F_\omega=f_\omega^3$. When we refer to the critical set of
$F_\omega$, we mean the set of points $z$ such that
$f_\omega^i(z)\in\mathcal C$ for some $i\in\{0,1,2\}$,
where $\mathcal C$ is defined in \eqref{eq:critical-set}.

\medskip
\noindent\textit{Young times.}  Consider two disjoint reference rectangles $I,\bar I$,
situated away from the critical set, and two curves
$\gamma_0,\bar\gamma_0$ with marked points
$z_0\in\gamma_0$ and $\bar z_0\in\bar\gamma_0$.

Informally, $n$ is a Young time if there exist small subarcs $z_0\in\gamma_0^{(n)}\subset\gamma_0$ and $\bar z_0\in\bar\gamma_0^{(n)}\subset\bar\gamma_0,$ satisfying the following  properties:
\begin{enumerate}
\item

The curves $F_\omega^n(\gamma_0^{(n)})$ and
$F_\omega^n(\bar\gamma_0^{(n)})$ are almost horizontal,
have lengths bounded below by a fixed power of $L^{-1}$,
and are uniformly separated both from each other and
from the critical set.

\item
During the next $K_1$ iterates, for some $K_1>0$, the orbits of the marked points avoid the critical set, while the image subarcs grow to a fixed scale and remain almost horizontal. At the end of this block,
$z_{n+K_1}\in I$ and $\bar z_{n+K_1}\in\bar I$.
Thus, for $\chi$ sufficiently small relative to the scale reached, each of $F_\omega^{n+K_1}(\gamma_0^{(n)})$ and
$F_\omega^{n+K_1}(\bar\gamma_0^{(n)})$ crosses both
$I$ and $\bar I$.

\item
During the following $K_2$ iterates, for some $K_2>0,$ the orbits of the marked points continue to avoid the critical set.
\end{enumerate}

Young times are formally introduced in Definition~\ref{youngtimes}, after some preliminary work. Leveraging the mixing of the two-point motion and the large-deviation estimates, Proposition~\ref{prop:young-times} establishes an exponential bound on the probability of
observing too few Young times. 

\medskip
\noindent\textit{Stable times.}
A stable time $j$ is a time at which the stable direction at $z_j$ is almost vertical and the future derivative cocycle satisfies a hyperbolicity condition known as
temperedness (see Definition~\ref{def:tempered-norms}): for every $n\ge1$, the time $n$ is a reverse-tempered time for the derivative cocycle of the inverse dynamics, started at time $j+n$ and followed backwards to time $j$
(see Definition~\ref{reversetemperedness}).

We will show that, along typical trajectories, stable times
have density arbitrarily close to one when $L$ is
sufficiently large.

\medskip
\noindent\textit{From Young and stable times to typical horseshoes.}
Fix $\chi>0$ sufficiently small and let $I,\bar I$
be two rectangles in $\chi$-horseshoe position, and let $I_c,\bar I_c$ denote their cores, as defined right after Definition \ref{def:horseshoe-position}.
Given any $(z_0,\bar z_0)\in I_c\times\bar I_c$
with distinct $y$-coordinates,  consider the horizontal
curves $\gamma_0,\bar\gamma_0$ passing through
$z_0,\bar z_0$ and crossing $I,\bar I$, respectively.

We consider the Young times associated with these curves
and marked points, together with the stable times of both
trajectories. We then select a sequence $\{\bar h_j\}_{j\ge0}$
such that each $\bar h_j=n+K_1$ for some Young time $n$,
$\bar h_j+K_2$ is a stable time for both trajectories,
and consecutive times are sufficiently separated.

We will prove in Proposition~\ref{prop:horseshoe-times}
that, for almost every noise realisation and almost every $(z_0,\bar z_0)\in I_c\times\bar I_c$, these times have positive lower density and satisfy all the geometric
and hyperbolic properties required in
Definition~\ref{def:typical-horseshoe}.
Since the times $\bar h_j$ are counted in iterates of
$F_\omega=f_\omega^3$, the corresponding sequence
of horseshoe times in that definition is $\{3\bar h_j\}_{j\ge0}$.

This section is organised as follows. In
Subsection~\ref{yt}, we introduce Young times and
establish their main properties. In
Subsection~\ref{stable-times}, we define stable times
and prove the corresponding estimates. Finally, in
Subsection~\ref{finito}, we combine these results
to prove Theorem~\ref{thm:typical-horseshoes}.

\subsection{Young times}\label{yt}

Let $z_0=(x_0,y_0),\bar z_0=(\bar x_0,\bar y_0) \in \mathbb{T}^2,$ $z_0\neq \bar z_0$ and write  $z_i=(x_i,y_i):=F_\omega^i(z_0)$ and $\bar z_i=(\bar x_i,\bar y_i):=F_\omega^i(\bar z_0),$ where $F_{\omega}$ is the third iterate of the random standard map as in \eqref{thirditerate}.
For fixed $\ell_0,K_0>0,$ let $\gamma_0,\bar\gamma_0 \colon [-\ell_0,\ell_0]\to \mathbb{T}^2$ be two curves satisfying  $\gamma(0)=z_0,\,\bar \gamma(0)=\bar z_0$ and $||\gamma''||_{\infty}, ||\bar \gamma''||_{\infty} \le K_0.$ By Proposition~\ref{hyperbolict_times} and
Theorem~\ref{thmgrowth}, there exist constants
$C,\alpha,\varepsilon,\varepsilon'>0$ such that, for every
$\delta_0>0$, there exist $c_{\delta_0}>0$ and
$L_0=L_0(\delta_0)>0$ for which the stopping times
$\tau_j(\gamma_0)$ and $\tau_j(\bar\gamma_0)$ defined in
\eqref{taugamma} using these constants satisfy
\eqref{putafkcnformula} whenever $L>L_0$. For $\gamma\in\{\gamma_0,\bar\gamma_0\}$, write
\[
H_n(\gamma)
:=
\{\tau_j(\gamma):j\ge1\}\cap\{1,\ldots,n\}.
\]
Let $v_0$ be the tangent direction of $\gamma_0$ at $z_0$,
and define $\bar v_0$ similarly. Denote their evolution under the
projective dynamics of $F_\omega$ by $v_i,\bar v_i$. Given $\eta_1,\eta_2>0$, define the set of simultaneous hyperbolic times associated to $(\gamma_0,\bar \gamma_0)$ by
\begin{equation}\label{simultaneoushyptimes}
\widetilde H_n(\gamma_0,\bar \gamma_0)
:=
H_n(\gamma_0)\cap H_n(\bar\gamma_0)\cap
\left\{
1\le i\le n:
\begin{array}{l}
d(z_i,\bar z_i)>\eta_1,\\[2mm]
v_i,\bar v_i\in B_{\eta_2}(0),\\[2mm]
z_i,\bar z_i\notin\mathcal G_{2\eta}
\end{array}
\right\},
\end{equation}
where $0\in\mathbb P^1$ denotes the horizontal direction, and, for $\eta>0,$ the set $G_{\eta}$ denotes the critical neighbourhood
\begin{equation}\label{criticaln}
\mathcal G_\eta
:=
\left\{
(x,y) \in\mathbb T^2:
d(x,0)<\eta
\ \text{or}\
d\left(x,\frac12\right)<\eta
\right\}.
\end{equation}

\begin{proposicao}\label{prop:hyperbolic-times-density}
For every $\bar\varepsilon\in(0,1)$, if $\eta,\eta_1,\eta_2>0$ are
sufficiently small and $0<\eta_2<1/10$ is fixed, there exist
$L_0,C,c>0$ such that, for every $L\ge L_0$ and $n\ge1$,
\begin{equation}\label{hyptimes}
\mathbb P\left(
\#\widetilde H_n(\gamma_0,\bar \gamma_0)\le(1-\bar\varepsilon)n
\right)
\le Ce^{-cn}.
\end{equation}
\end{proposicao}

\begin{proof}
Theorem~\ref{thmgrowth} gives that the density of
$H_n(\gamma_0)$ and $H_n(\bar\gamma_0)$ is arbitrarily
close to one, with exponential tails.
Consequently, for every $\bar\delta_0>0$ sufficiently small,
there exist $c_{\bar\delta_0}>0$ and $n_{\bar\delta_0}\ge1$
such that, for $L$ sufficiently large and $n\ge n_{\bar\delta_0}$,
\begin{equation}\label{firsth}
\mathbb P\left(
\#\bigl(H_n(\gamma_0)\cap H_n(\bar\gamma_0)\bigr)
\le (1-\bar\delta_0)n
\right)
\le 2L^{-c_{\bar\delta_0}n}.
\end{equation}

The bad position conditions
$d(z_i,\bar z_i)\le\eta_1$ and
$z_i\in\mathcal G_{2\eta}$ or $\bar z_i\in\mathcal G_{2\eta}$
have stationary probabilities $O(\eta_1^2)$ and $O(\eta)$,
respectively. Moreover, for fixed $\eta_2>0$, the estimate for
the stationary distribution in Proposition~\ref{proposition1}
shows that the stationary probability of
$v_i\notin B_{\eta_2}(0)$ tends to zero as $L\to\infty$,
and the same holds for $\bar v_i$.

Given $\delta_0>0$, we may therefore choose $\eta_1$ and $\eta$
sufficiently small, and then $L$ sufficiently large, so that
the large deviation estimates in
Theorem~\ref{thm:large_deviations} yield constants $C,\gamma>0$
such that
\begin{equation}\label{nowforeta}
\mathbb P\left(
\#\left\{
1\le i\le n:
\begin{array}{l}
d(z_i,\bar z_i)\le\eta_1
\ \text{or}\ z_i\in\mathcal G_{2\eta}
\\
\text{or}\ \bar z_i\in\mathcal G_{2\eta}
\ \text{or}\ v_i\notin B_{\eta_2}(0)
\\
\text{or}\ \bar v_i\notin B_{\eta_2}(0)
\end{array}
\right\}
>\delta_0 n
\right)
\le Ce^{-\gamma n}.
\end{equation}

Formula \eqref{hyptimes} follows from \eqref{firsth} and
\eqref{nowforeta}, choosing
$\delta_0+\bar\delta_0<\bar\varepsilon$.
\end{proof}

Given positive integers $K_1,K_2$, define the $K_1+K_2$-sparse simultaneous hyperbolic times associated to $(\gamma_0,\bar \gamma_0)$ recursively by
\[
r_0:=0,
\qquad
r_{j+1}:=
\inf\left\{m\in\bigcup_{n\ge1}\widetilde H_n(\gamma_0,\bar \gamma_0):m>r_j+K_1+K_2\right\},
\]
where $\widetilde H_n(\gamma_0,\bar \gamma_0)$ is as in \eqref{simultaneoushyptimes}. Let $\widetilde H_{n,K_1,K_2}(\gamma_0,\bar \gamma_0)$ be the set of $K_1+K_2$-simultaneous hyperbolic times associated to $(\gamma_0,\bar \gamma_0)$ which are less or equal than $n.$ Applying \eqref{hyptimes} with $\bar\varepsilon/2$ in place
of $\bar\varepsilon$, we get 
\begin{equation}\label{eq:sparse-hyperbolic-times}
\mathbb P\left(
\#\widetilde H_{n,K_1,K_2}(\gamma_0,\bar \gamma_0)
\le
\frac{(1-\bar\varepsilon)n}{K_1+K_2+2}
\right)
\le Ce^{-cn}.
\end{equation}

Define the event
\[
\mathcal A_{k,l}(\eta;z_0,\bar z_0)
:=
\bigcap_{r=3k}^{3l-1}
\left\{
\begin{array}{l}
 f_\omega^r(z_0)\notin\mathcal G_\eta,\\
 f_\omega^r(\bar z_0)\notin\mathcal G_\eta
\end{array}
\right\}.
\]
Thus $\mathcal A_{k,l}$ requires critical avoidance at every
individual $f_\omega$-step between the $k$th and $l$th
iterates of $F_\omega$. Let $\mathcal{G}_{\eta}$ as in \eqref{criticaln} and fix two disjoint rectangles $I,
\bar I\subset \mathbb T^2\setminus\mathcal G_\eta.$ Let $E(z_0,\bar z_0;I,\bar I)$ denote the event that the orbits starting at $z_0,\bar z_0$ lie in $I$ and $\bar I$, respectively, at time $K_1$ under
$F_\omega$, and satisfy the critical avoidance condition
throughout the first $K_1+K_2$ iterates. More precisely, set
\begin{equation}\label{matheusset}
E(z_0,\bar z_0;I,\bar I)
:=
\left\{
F_\omega^{K_1}(z_0)\in I,
\quad
F_\omega^{K_1}(\bar z_0)\in\bar I
\right\}
\cap
\mathcal A_{0,K_1+K_2}(\eta;z_0,\bar z_0).
\end{equation}

\begin{definicao}\label{youngtimes}
Let $\gamma_0,\bar \gamma_0$ as above, with  $\gamma_0(0) =z_0, \bar \gamma_0(0) = \bar z_0$ with $z_0\neq \bar z_0,$ and rectangles $I,\bar I\subset 
\mathbb T^2\setminus \mathcal{G}_{\eta}$. We say that  $i$ is a $(\gamma_0,\bar \gamma_0,z_0\bar z_0,I,\bar I)$ Young time if $i$ is a $K_1+K_2$-sparse simultaneous hyperbolic time associated to $(\gamma_0,\bar \gamma_0)$ satisfying $\vartheta^{3i}\omega\in E(z_i,\bar z_i;I,\bar I).$
\end{definicao}
Let $Y_n(\gamma_0,\bar \gamma_0,z_0\bar z_0,I,\bar I)$, denote the set of the $(\gamma_0,\bar \gamma_0,z_0\bar z_0,I,\bar I)$-Young times less than or equal to $n$. When the curves, initial points and reference rectangles are clear from the context, we suppress these arguments and write simply $Y_n$. The following result gives a lower-tail estimate for the frequency $|Y_n|/n$.
\begin{proposicao}\label{prop:young-times}
Let $\gamma_0,\bar\gamma_0$ be as specified above, with
$\gamma_0(0)=z_0$ and $\bar\gamma_0(0)=\bar z_0$.
Let $I,\bar I\subset
\mathbb T^2\setminus\mathcal G_{\eta_0},$
for some $\eta_0>0$, and assume that
$m(I)m(\bar I)>p^*>0$.
Fix $K_2\ge1$. One can choose $K_1$ sufficiently large, then $\eta>0$ sufficiently small, and $L_0$ sufficiently large, so that
there exist $b,C_1,c_1,C_{\eta}>0$ such that, for every $L\ge L_0$ and $n\ge1$,
\begin{equation}\label{cardinalityyoungtimes}
\mathbb P\left(\#Y_n\le bn\right)
\le C_1e^{-c_1n}.
\end{equation}
Furthermore, if $i$ is a Young time, there exist subcurves $\gamma_0^+\subset\gamma_0$ and
$\bar\gamma_0^+\subset\bar\gamma_0$, containing $z_0$ and
$\bar z_0$, respectively, such that
\begin{equation}\label{curvereachref}
z_{i+K_1}\in F_\omega^{i+K_1}(\gamma_0^+)\cap I,
\qquad
\bar z_{i+K_1}\in F_\omega^{i+K_1}(\bar\gamma_0^+)\cap \bar I.
\end{equation}
Each image curve extends for a length of at least $1/16$
on each side of its marked point, and all tangent directions
of both curves lie in $B_{C_\eta/L}(0)$.
\end{proposicao}

\begin{proof}
We first prove that there exists $\rho>0$ such that
\begin{equation}\label{eq:young-success-probability}
\inf_{\substack{
d(z,\bar z)>\eta_1\\
 z,\bar z\notin\mathcal G_{2\eta}
}}
\mathbb P\bigl(E(z,\bar z)\bigr)
\ge\rho.
\end{equation}

By Theorem~\ref{mixingtwopointmotion}, applied at time $3K_1$,
choosing $K_1$ sufficiently large gives
\[
\mathbb P\left(
F_\omega^{K_1}(z)\in I,
\quad
F_\omega^{K_1}(\bar z)\in \bar I
\right)
\ge\frac{m(I)m(\bar I)}{2}>\frac{p^*}{2},
\]
uniformly for $d(z,\bar z)>\eta_1$ and $L$ sufficiently large.

If $z_0\notin\mathcal G_{2\eta}$, then, for $L$
sufficiently large, the map $\omega_0\longmapsto f(x_0+\omega_0)-y_0 \mod 1$ has 
bounded distortion on its injective components. Consequently,
\begin{equation}\label{consequently}
\mathbb P\left(
f_\omega(z)\in\mathcal G_{2\eta}
\right)
\le C_0\eta,
\end{equation}
uniformly for $z\notin\mathcal G_{2\eta}$ .

For $L$ sufficiently large, failure of critical avoidance requires one of the unperturbed
horizontal coordinates to enter $\mathcal G_{2\eta}$.
The bound in \eqref{consequently} then gives
\[
\mathbb P\left(
\mathcal A_{0,K_1+K_2}(\eta;z,\bar z)^c
\right)
\le6C_0(K_1+K_2)\eta.
\]
Therefore,
\[
\mathbb P\bigl(E(z,\bar z)\bigr)
\ge
\frac{p^*}{2}-6C_0K\eta
\ge
\frac{p^*}{4}
=:\rho,
\]
after choosing $\eta$ sufficiently small. This proves
\eqref{eq:young-success-probability}.

At every candidate $K_1+K_2$-sparse simultaneous hyperbolic time, we have
$d(z_i,\bar z_i)>\eta_1$ and
$z_i,\bar z_i\notin\mathcal G_{2\eta}$. The conditional exponential-moment argument
in \cite{LTTGeneral}, together with \eqref{eq:sparse-hyperbolic-times}
and \eqref{eq:young-success-probability}, then gives \eqref{cardinalityyoungtimes}.

We now prove the geometric assertion in \eqref{curvereachref}.
If $i$ is a Young time, Definition~\ref{youngtimes} and
Proposition~\ref{hyperbolict_times} give a subcurve of
$F_\omega^i(\gamma_0)$ through $z_i$, which extends for a length of at least $L^{-b_0}$ on each side of $z_i$ and with curvature at most $L^{a_0}$, for some $a_0,b_0>0$.
Moreover, $v_i\in B_{\eta_2}(0)$, where $\eta_2>0$ is sufficiently
small. Choose $q>\max\{a_0,b_0\}$. By the curvature bound, the
subcurve within distance $L^{-q}$ of $z_i$ is almost
horizontal for all sufficiently large $L$.

For every $r\le 3K_1$, we have $f^r_{\vartheta^{3i}(\omega)}(z_i),f^r_{\vartheta^{3i}(\omega)}(\bar z_i)
\notin \mathcal{G}_{\eta}.$ At distances comparable to $\eta$ from the critical set, almost
horizontal curves remain almost horizontal and are expanded by a
factor of at least $c_\eta L$, for some $c_\eta>0$.
Consequently, if $3K_1>q+1$, then, for sufficiently large $L$,
the image $F_\omega^{i+K_1}(\gamma_0)$ contains an almost horizontal
subcurve through $z_{i+K_1}$ that reaches large scale.
Since $z_{i+K_1}\in I$, this proves the first assertion
in \eqref{curvereachref}, together with the claimed lower bound
on the length of the image curve.
The same argument applied to $\bar\gamma_0$, with
$\bar z_{i+K_1}\in\bar I$, proves the second.
\end{proof}

\subsection{Stable times}\label{stable-times}

For $(\omega,z)\in\Omega\times\mathbb T^2$, let $A_m$ be as in \eqref{sequenceofmatrices} with $m =n.$ The following is essentially \cite[Definition 4.2]{DD24}.

\begin{definicao}\label{def:tempered-norms}
Let $C\in\mathbb R$, $\alpha,\varepsilon>0$, and
$n\in\mathbb N\cup\{\infty\}$.
We say that $\{A_m\}_{m=0}^{n}$ has
$(C,\alpha,\varepsilon)$-tempered norms if
\[
\|A_m\|
\ge
e^C L^{\alpha(m-k)-\varepsilon k}\|A_k\|
\]
for all integers $0\le k<m\le n$.
\end{definicao}

In \cite{DD24}, tempered norms are used to obtain tempered
splittings and to control the size of finite-time stable
manifolds through graph transform arguments; see
\cite[Proposition~4.6 and Appendix~B]{DD24}.
We will use this notion for a similar purpose in the horseshoe construction.

For $j\ge0$, write
\begin{equation}\label{eq:shifted-derivative}
A_m^{(j)}(\omega,z)
:=
A_m(\vartheta^{3j}\omega,z_j), \qquad 
m\ge0.
\end{equation}
Let $e^s(\omega,z)$ be the stable direction defined in
Section~\ref{stabledirsec}, and write $e_j^s:=e^s(\vartheta^{3j}\omega,z_j).$

Fix $\alpha,\varepsilon,u>0$ as in the preceding growth
and direction estimates. We define the set of stable
configurations by
\begin{equation}\label{eq:stable-configurations}
S
:=
\left\{
(z,\omega)\in\mathbb T^2\times\Omega:
\begin{array}{l}
\{A_m(\omega,z)\}_{m\ge0}
\text{ has }(0,\alpha,\varepsilon)
\text{-tempered norms},\\[1mm]
\angle(e^s(\omega,z),(0,1))\le L^{-u}
\end{array}
\right\},
\end{equation}
where we restrict to the invariant set of full measure
on which $e^s$ is defined.

\begin{definicao}\label{def:stable-times}
The set of stable times for $(\omega,z)$ is
\begin{equation}\label{eq:stable-times}
\mathcal S(\omega,z)
:=
\{j\ge0:(z_j,\vartheta^{3j}\omega)\in S\}.
\end{equation}
\end{definicao}

\begin{proposicao}\label{stableprop}
For the parameter choices above, there exist
$L_0,C,\beta>0$ such that, for every $L\ge L_0$,
\begin{equation}\label{eq:stable-configurations-measure}
(m\otimes\mathbb P)(S)
\ge 1-CL^{-\beta}.
\end{equation}
The constants are independent of $L$.

Consequently, for every $\varepsilon'>0$, choosing $L$
sufficiently large gives
\begin{equation}\label{densitystabledire}
m\otimes\mathbb P\text{-a.s.},
\qquad
\liminf_{n\to\infty}
\frac{
\#\bigl(\mathcal S(\omega,z)\cap\{0,\ldots,n-1\}\bigr)
}{n}
>
1-\varepsilon'.
\end{equation}
\end{proposicao}

\begin{proof}
The summation argument in the proof of
\cite[Propositions~4.4 and~4.7]{DD24}, using
Theorem~\ref{lyapunov_large_deviations} in place of their
large deviation estimate, gives
\begin{equation}\label{eq:stable-norms-probability}
\mathbb P\left(
\{A_m(\omega,z)\}_{m\ge0}
\text{ does not have }
(0,\alpha,\varepsilon)\text{-tempered norms}
\right)
\le CL^{-\gamma},
\end{equation}
for some $C,\gamma>0$, uniformly in $z$.

By Corollary~\ref{Stabledef}, there exists $v>0$ such that
\[
(m\otimes\mathbb P)\left(
\angle(e^s(\omega,z),(0,1))>L^{-u}
\right)
\le CL^{-v}.
\]
Combining these estimates proves
\eqref{eq:stable-configurations-measure}, with
$\beta:=\min\{\gamma,v\}$ and after increasing $C$. Finally,  \eqref{densitystabledire} follows from
Birkhoff's ergodic theorem and
\eqref{eq:stable-configurations-measure}.
\end{proof}

\subsection{Concluding the proof of Theorem \ref{thm:typical-horseshoes}}\label{finito}

Given $\chi>0,$ let $I,\bar I$ be rectangles  in $\chi$-horseshoe
position, and let $I_c,\bar I_c$ be their cores, as defined right after Definition \ref{def:horseshoe-position}.
For $(z,\bar z)\in I\times \bar I$ and $z\neq \bar z,$ let $\gamma_z,\bar\gamma_{\bar z}$
be horizontal arcs with  $\gamma_z(0)=z, \bar\gamma_{\bar z}(0) =\bar z$, crossing, respectively, $I,\bar I$
and satisfying the geometric hypotheses of
Proposition~\ref{hyperbolict_times}, for some $K_0,\ell_0.$ 

For $K_2\ge1$ and $K_1$ sufficiently large, both to be fixed
later in the proof, Proposition~\ref{prop:young-times}
provides $\eta>0$ such that the set of Young-times $Y_n=Y_n(\gamma_z,\bar\gamma_{\bar z},z,\bar z,I_c,\bar I_c)$ satisfies \eqref{cardinalityyoungtimes}.
The dependence of $Y_n$ on $\eta$ is implicit in
Definition~\ref{youngtimes} and \eqref{matheusset}.
By \eqref{cardinalityyoungtimes} and the Borel-Cantelli lemma,
there exists $\beta_\chi>0$ such that, for
$(m\otimes m)$-almost every $(z,\bar z)\in I\times \bar I$,
\begin{equation}\label{firstcardinality}
\mathbb P\left(
\liminf_{n\to\infty}\frac{\#Y_n(\omega,z,\bar z)}{n}
>\beta_\chi
\right)=1.
\end{equation}
We write
\[
Y:=\bigcup_{n\ge1}Y_n=\{\tilde t_j\}_{j\ge1}.
\]
Fix $0<\varepsilon_1<\beta_\chi$. By \eqref{densitystabledire}, if $L$ is sufficiently large, we have that
\begin{equation}\label{pairwisestable}
\liminf_{n\to\infty}
\frac{\#\bigl(\mathcal S(\omega,z)\cap
\mathcal S(\omega,\bar z)\cap\{1,\ldots,n\}\bigr)}{n}
>1-\varepsilon_1
\end{equation}
for $(m\otimes m\otimes\mathbb P)$-almost every
$(z,\bar z,\omega)\in I\times \bar I\times\Omega$.

Consider the return times
\begin{equation}\label{horseshoestime}
\mathcal H
:=
\left\{
\tilde t_j+K_1:
\tilde t_j+K_1+K_2\in
\mathcal S(\omega,z)\cap\mathcal S(\omega,\bar z)
\right\}.
\end{equation}
By \eqref{firstcardinality} and \eqref{pairwisestable}, this set
has lower density greater than $\beta_\chi-\varepsilon_1$ almost
surely. Let $D_0\ge1$ be an integer, to be increased finitely many times during the proof, and define
\begin{equation}\label{heicj}
\bar h_1:=\inf\mathcal H,
\qquad
\bar h_{j+1}:=\inf\{\bar h\in\mathcal H: \bar h>\bar h_j+K_1+K_2+D_0\}.
\end{equation}

The same sparsification argument as in
\eqref{eq:sparse-hyperbolic-times} gives
\begin{equation}\label{horseshoetimesdensity}
\liminf_{n\to\infty}
\frac{\#\{j\ge1:\bar h_j\le n\}}{n}
>\bar\beta,
\qquad
\bar\beta:=
\frac{\beta_\chi-\varepsilon_1}{K_1+K_2+D_0+1}>0,
\end{equation}
for $(m\otimes m\otimes\mathbb P)$-almost every
$(z,\bar z,\omega)\in I\times \bar I\times\Omega$.

We now establish the geometric and hyperbolic properties of the times $\{\bar h_j\}_{j\ge1}$ that will imply that $h_j:=3\bar h_j$ are horseshoe times in the sense of Definition~\ref{def:typical-horseshoe}.
\begin{proposicao}\label{prop:horseshoe-times}
Fix $\chi,\delta>0$ sufficiently small.
For suitable choices of $K_1,K_2$ and $\eta$ as above,
there exist $D_*\ge1$ and $\beta_*>0$ such that,
for every $D_0\ge D_*$, there exists $L_0\ge1$
with the following property.

For every $L\ge L_0$ and every pair of rectangles
$I,\bar I$ in $\chi$-horseshoe position, with cores
$I_c,\bar I_c$, the following holds for
$(m\otimes m\otimes\mathbb P)$-almost every
$(z,\bar z,\omega)\in I_c\times\bar I_c\times\Omega$:
the sequence $\{\bar h_j\}_{j\ge1}$ defined in
\eqref{heicj} satisfies \eqref{horseshoetimesdensity}
and has the following properties.

Write
\[
n_j:=\bar h_{j+1}-\bar h_j,\qquad
P_j:=F_{\vartheta^{3\bar h_j}\omega}^{n_j},
\]
and
\[
z_j:=F_\omega^{\bar h_j}(z),\qquad
\bar z_j:=F_\omega^{\bar h_j}(\bar z).
\]

\begin{enumerate}

\item\label{jesus1} For every $j\ge1$,
\[
z_j\in I_c,\qquad \bar z_j\in\bar I_c,
\]
and there exist subcurves
$\gamma_j\subset\gamma_z$ and
$\bar\gamma_j\subset\bar\gamma_{\bar z}$,
containing $z$ and $\bar z$, respectively,
such that
\begin{equation}\label{curvesprop9}
\Gamma_j:=F_\omega^{\bar h_j}(\gamma_j),
\qquad
\bar\Gamma_j:=F_\omega^{\bar h_j}(\bar\gamma_j)
\end{equation}
are simultaneous horizontal crossings with margin
of $I,\bar I$.

\item\label{jesus2} For every $j\ge1$, there exist
subcurves $\Gamma'_j\subset\Gamma_j$ and
$\bar\Gamma'_j\subset\bar\Gamma_j$ satisfying
\[
P_j(\Gamma'_j)=\Gamma_{j+1},
\qquad
P_j(\bar\Gamma'_j)=\bar\Gamma_{j+1},
\]
and vertical strips $H_j\subset I$ and
$\bar H_j\subset\bar I$ admitting a foliation of $C^1$ curves
\begin{equation}\label{thefinalfoliation}
H_j=\bigsqcup_{p\in\Gamma'_j}W_j(p),
\qquad
\bar H_j=
\bigsqcup_{p\in\bar\Gamma'_j}\bar W_j(p)
\end{equation}
by almost vertical crossings of $I$ and
$\bar I$, respectively, such that
$$
W_j(p)\cap\Gamma'_j=\{p\},
\qquad
\bar W_j(p)\cap\bar\Gamma'_j=\{p\}.
$$

Moreover, $P_j(H_j)$ and $P_j(\bar H_j)$
are horizontal strips  crossing both $I$ and $\bar I$ with margin.

\item\label{jesus3} For every $q\in H_j\cup\bar H_j$
and every nonzero $v\in\mathcal C^h$,
\begin{equation}\label{geoemtricone}
DP_j(q)v\in\operatorname{int}\mathcal C^h,
\qquad
\|DP_j(q)v\|
\ge L^{\beta_*n_j}\|v\|.    
\end{equation}
Similarly, for every
$q\in P_j(H_j\cup\bar H_j)$ and every
nonzero $w\in\mathcal C^v$,
\begin{equation}\label{geometric2}
DP_j^{-1}(q)w\in\operatorname{int}\mathcal C^v,
\qquad
\|DP_j^{-1}(q)w\|
\ge L^{\beta_*n_j}\|w\|.
\end{equation}

\end{enumerate}
\end{proposicao}

The proof of the above Proposition is given in Appendix \ref{gym}.

\begin{proof}[Proof of Theorem \ref{thm:typical-horseshoes}] 

Fix $\chi>0$ sufficiently small, and let $I,\bar I$ be
rectangles in $\chi$-horseshoe position. Choose
$K_1,K_2,\eta$ and $D_0\ge D_*$ as above, and let
$L\ge L_0$ be sufficiently large.

Set $h_j:=3\bar h_j$. By \eqref{horseshoetimesdensity},
these times have lower density greater than
$d_*:=\bar\beta/3$ for
$(m\otimes m\otimes\mathbb P)$-almost every
$(z,\bar z,\omega)\in I_c\times\bar I_c\times\Omega$.
Since $F_\omega=f_\omega^3$, the properties established
in Proposition~\ref{prop:horseshoe-times} imply that
$\{h_j\}_{j\ge1}$ satisfies
items~\eqref{item1}, \eqref{item2} and~\eqref{itemthree}
of Definition~\ref{def:typical-horseshoe},
with expansion exponent $\beta_*/3$.

By Fubini's theorem, these conclusions hold for
$\mathbb P$-almost every $\omega$ and
$(m\otimes m)$-almost every
$(z,\bar z)\in I_c\times\bar I_c$ with distinct
$y$-coordinates. Thus $(I,\bar I)$ is a typical horseshoe.
\end{proof}

\appendix

\section{Elementary probabilistic identities}
\label{app:conditioning}

We record two probabilistic facts used in the estimates of the quantities $I_2$ and $I_3$ defined in \eqref{i2} and \eqref{i3}. The first characterizes, for $n\ge 4$, the distribution of the earlier noise coordinates $u=(\omega_0,\dots,\omega_{n-3})$, conditioned on the existence of admissible final coordinates $(\omega_{n-2},\omega_{n-1})$ completing $u$ to a noise realization $\omega$ for which $f_\omega^n(x_0,y_0)=(x,y)$. The second describes the marginals and conditional laws of a uniform distribution on a measurable set.
Their proofs follow from the change-of-variables formula and Fubini's theorem and are omitted.

\begin{proposicao}\label{fromrmtogoodset}
Let $(\omega_k)_{k\ge0}$ be i.i.d.\ random variables uniformly
distributed on $[-L^{-\delta},L^{-\delta}]$, and write
\[
f_\omega^n(x_0,y_0)=F(u,v),
\qquad
u=(\omega_0,\ldots,\omega_{n-3}),
\qquad
v=(\omega_{n-2},\omega_{n-1}).
\]
Assume that:
\begin{enumerate}
\item for each fixed $u$, the equation $F(u,v)=(x,y)$
has at most one solution
$v\in[-L^{-\delta},L^{-\delta}]^2$;

\item
\[
|\det D_vF(u,v)|=1
\qquad\text{whenever }F(u,v)=(x,y).
\]
\end{enumerate}
Define the set of admissible earlier noise coordinates by
\[
\Gamma_{x,y}
:=
\left\{
u:\exists v\in[-L^{-\delta},L^{-\delta}]^2
\text{ such that }F(u,v)=(x,y)
\right\}.
\]
If $\mathbb P(u\in\Gamma_{x,y})>0$, then conditional law of $u$ given $F(u,v)=(x,y)$ is its law conditioned on 
$u\in \Gamma_{x,y}$. Thus, for every
$A\in\mathcal F_0^{n-3}$,
\[
\mathbb P\bigl(A\mid F(u,v)=(x,y)\bigr)
=
\mathbb P\bigl(A\mid u\in\Gamma_{x,y}\bigr).
\]
\end{proposicao}

\begin{proposicao}\label{from_normalized_to_others}
Let $Q=[-L^{-\delta},L^{-\delta}]^n$, and let $\mu$ be
normalized Lebesgue measure on $Q$. For a measurable set
$\Gamma\subset Q$ with $\mu(\Gamma)>0$, define
\[
\mu_\Gamma(A):=\frac{\mu(A\cap\Gamma)}{\mu(\Gamma)}.
\]
Write $z=(u,v)$, where
$u\in[-L^{-\delta},L^{-\delta}]^k$ and
$v\in[-L^{-\delta},L^{-\delta}]^{n-k}$, and define the fibres
\[
\Gamma_u
:=
\left\{
v\in[-L^{-\delta},L^{-\delta}]^{n-k}:
(u,v)\in\Gamma
\right\}.
\]
Here $|\Gamma_u|$ denotes the Lebesgue measure of the fibre.
Then:
\begin{enumerate}
\item\label{projeothers} The marginal of $\mu_\Gamma$ on the $u$-coordinates
has density
\[
u\longmapsto
\frac{|\Gamma_u|}{\int|\Gamma_{u'}|\,du'}
\]
with respect to Lebesgue measure.

\item\label{conditionalothers} For Lebesgue-almost every $u$ with $|\Gamma_u|>0$,
the conditional law of $v$ given $u$ is normalized
Lebesgue measure on $\Gamma_u$:
\[
\mu_\Gamma(dv\mid u)
=
\frac{\mathbf1_{\Gamma_u}(v)}{|\Gamma_u|}\,dv.
\]

\item\label{othersothers} The measure $\mu_\Gamma$ admits the disintegration
\[
\mu_\Gamma(du,dv)
=
\frac{|\Gamma_u|}{\int|\Gamma_{u'}|\,du'}\,du
\cdot
\frac{\mathbf1_{\Gamma_u}(v)}{|\Gamma_u|}\,dv.
\]
\end{enumerate}
\end{proposicao}
\section{Proofs of the geometric estimates}
\label{app:geometric-estimates}

In this appendix, we prove Propositions~\ref{four-step-trans}
and~\ref{uniform_goodness}. Both proofs are based on the following observation. Let 
\begin{equation}\label{inverseformula}
f_{\omega_0}^{-1}(x,y)
=
(y-\omega_0,f(y)-x),
\qquad
\omega_0\in[-L^{-\delta},L^{-\delta}],
\end{equation}
be the inverse map associated to  \eqref{randomstandardmap}
and write
\[
G_{\omega_1,\omega_2,\omega_3}
:=
f_{\omega_1}^{-1}\circ
f_{\omega_2}^{-1}\circ
f_{\omega_3}^{-1}.
\]
Then
\begin{equation}\label{Keuobs}
f_\omega^4(x_0,y_0)=(x,y)
\quad\Longleftrightarrow\quad
(x_1(\omega_0),y_1(\omega_0))
=
G_{\omega_1,\omega_2,\omega_3}(x,y).
\end{equation}
By \eqref{inverseformula},
\begin{equation}\label{inversemapF}
\begin{aligned}
G_{\omega_1,\omega_2,\omega_3}(x,y)
=
\Bigl(
&H_1(x,y)+\omega_3-\omega_1,\\
&f(H_1(x,y)+\omega_3)-f(y)+x+\omega_2
\Bigr),
\end{aligned}
\end{equation}
where
\begin{equation}\label{hacca1xy}
H_1(x,y):=f(f(y)-x)-y.
\end{equation}

\begin{proof}[Proof of Proposition~\ref{four-step-trans}]
Let $A_{(x_0,y_0)}^{(x,y)}$ be as in \eqref{ax0y0xy}.
If $\omega_0\in A_{(x_0,y_0)}^{(x,y)}$, then necessarily,
by \eqref{Keuobs}, \eqref{inversemapF} and \eqref{hacca1xy}, $x_1(\omega_0)=H_1(x,y)+\omega_3-\omega_1.$
In particular,
\begin{equation}\label{actualkeyinclusion}
A_{(x_0,y_0)}^{(x,y)}
\subset
\left\{
\omega_0:
x_1(\omega_0)\in
[H_1(x,y)-2L^{-\delta},H_1(x,y)+2L^{-\delta}]
\right\}.
\end{equation}
To construct the set
$\tilde A_{(x_0,y_0)}^{(x,y)}$ with the properties stated in
Proposition~\ref{four-step-trans}, the idea is to select those
$\omega_0$ for which $x_1(\omega_0)$ lies sufficiently 
inside the interval
$[H_1(x,y)-2L^{-\delta},H_1(x,y)+2L^{-\delta}]$
and $|x_1'(\omega_0)|$ is sufficiently large.

More precisely, we first introduce the strip
\begin{equation}\label{jei1}
J_1:=
\left[
H_1(x,y)-2L^{-\delta}+\tfrac12L^{-2\delta},
H_1(x,y)+2L^{-\delta}-\tfrac12L^{-2\delta}
\right],
\end{equation}
the boundary region $\partial J
:=
[H_1(x,y)-2L^{-\delta},H_1(x,y)+2L^{-\delta}]
\setminus J_1,$ and  the critical region $\mathcal C_{99/100}
:=
\left\{
\omega_0:|f'(x_0+\omega_0)|\le L^{99/100}
\right\},$
and define
\begin{equation}\label{key_set}
\tilde A_{(x_0,y_0)}^{(x,y)}
:=
\left\{
\omega_0\in\left[-L^{-\delta},L^{-\delta}\right]\setminus\mathcal C_{99/100}:
x_1(\omega_0)\in J_1
\right\}.
\end{equation}

By \cite[Proposition 11]{BXY18}, after discarding boundary
branches of length $\asymp L^{-99/100}$, the set
$[-L^{-\delta},L^{-\delta}]\setminus\mathcal C_{99/100}$
can be partitioned into $\gtrsim L^{99/100-\delta}$
connected components $I_j$, each mapped onto the full circle
by $\omega_0\mapsto x_1(\omega_0)$, with distortion bounded
by a constant $C\sim 1+L^{-98/100}$.
In particular, this proves \eqref{coolitem}.

The distortion bound implies that, for each component $I_j$,
\[
\operatorname{Leb}
\{\omega_0\in I_j:x_1(\omega_0)\in J_1\}
\asymp |J_1|\,|I_j|.
\]
Writing $\tilde I_j=I_j\cap x_1^{-1}(J_1)$, and accounting for the discarded boundary branches, we therefore obtain \[
\operatorname{Leb}\bigl(\tilde A_{(x_0,y_0)}^{(x,y)}\bigr)
=
\sum_j|\tilde I_j|+O(L^{-99/100})
\asymp
|J_1|\sum_j|I_j|
\asymp L^{-2\delta},
\]
since $|J_1|\asymp L^{-\delta}$ and the total length of
the components $I_j$ is comparable to $L^{-\delta}$.
Similarly, using \eqref{actualkeyinclusion}, we also get $\operatorname{Leb}\bigl(
A_{(x_0,y_0)}^{(x,y)}
\setminus\tilde A_{(x_0,y_0)}^{(x,y)}
\bigr)
\lesssim L^{-3\delta},$ which completes the proof of \eqref{measurethesitmates}.

Item \eqref{remainingclaim} follows from the lemma below.
\end{proof}

\begin{lema}\label{lemma_minor}
If $\omega_0\in\tilde A_{(x_0,y_0)}^{(x,y)}$, then the set $B_{(x_0,y_0)}^{(x,y)}(\omega_0),$ as defined in \eqref{bixoyoxy}, contains at least $\sim L^{99/100-2\delta}$ connected
components, each of size at least $L^{-1-\delta}$.
\end{lema}

\begin{proof}
The proof proceeds in three steps. In Step~1, given $\omega_0 \in \tilde A_{(x_0,y_0)}^{(x,y)},$ we construct many pairs $(\omega_1,\omega_3)$ satisfying
$G_{\omega_1,0,\omega_3}(x,y)=(x_1(\omega_0),y_1(\omega_0))$.
In Step~2, we show that their $\omega_1$-coordinates are
well separated. Finally, in Step~3, we show that, after
discarding $O(1)$ pairs, each of these coordinates is the
centre of an interval contained in
$B_{(x_0,y_0)}^{(x,y)}(\omega_0)$.
These intervals are pairwise disjoint, and their total
length gives the required estimate.

\noindent\textbf{Step 1.}
We first claim that, for every
$\omega_0\in\tilde A_{(x_0,y_0)}^{(x,y)}$,
\begin{equation}\label{cardinality}
\#\left\{
(\omega_1,\omega_3):
(x_1(\omega_0),y_1(\omega_0))
=
G_{\omega_1,0,\omega_3}(x,y),\omega_3\notin H_L
\right\}
>
L^{99/100-2\delta},
\end{equation}
where $H_L
:=
\left\{
\omega_3:
|f'(H_1(x,y)+\omega_3)|\le L^{99/100}
\right\}.$

Fix $\omega_0\in\tilde A_{(x_0,y_0)}^{(x,y)}$.
Then $x_1(\omega_0)\in J_1$, where $J_1$ is defined in
\eqref{jei1}. Since the gaps corresponding to $H_L$
have negligible length, the horizontal set
\begin{equation}\label{horizontalball}
\left(
B_{L^{-\delta}}(x_1(\omega_0))
\cap
\left[H_1(x,y)+
\bigl([-L^{-\delta},L^{-\delta}]\setminus H_L\bigr)\right]
\right)\times\{y_1(\omega_0)\}
\end{equation}
contains a connected component with length $\gtrsim L^{-2\delta}$.

By \eqref{inversemapF}, we have the identity
\begin{equation}\label{displacement}
\omega_3\longmapsto G_{\omega_1,0,\omega_3}(x,y)
=
(-\omega_1,0)+G_{0,0,\omega_3}(x,y).
\end{equation}
Thus, varying $\omega_1$ simply shifts the curve
$\omega_3\mapsto G_{0,0,\omega_3}(x,y)$ horizontally.

Again by \eqref{inversemapF}, the curve
$\omega_3\mapsto G_{0,0,\omega_3}(x,y)$, restricted to
$\omega_3\notin H_L$, is the graph $(x,\gamma(x))$ of a function $\gamma$
defined for 
$x \in H_1(x,y)+\bigl([-L^{-\delta},L^{-\delta}]\setminus H_L\bigr)$.
On each connected component, $\gamma$ has a monotone lift
to $\mathbb R$ satisfying $|\gamma'|\ge L^{99/100}$.
Its graph therefore intersects the largest connected
component of the horizontal set in \eqref{horizontalball}
$\gtrsim L^{99/100-2\delta}$ times.

At each intersection, the vertical coordinate is already
$y_1(\omega_0)$, while the horizontal coordinate lies within
$L^{-\delta}$ of $x_1(\omega_0)$.
By \eqref{displacement}, we can therefore choose
$\omega_1\in[-L^{-\delta},L^{-\delta}]$ to shift the
intersection horizontally to $(x_1(\omega_0),y_1(\omega_0))$.
Each intersection gives a distinct pair $(\omega_1,\omega_3)$,
proving \eqref{cardinality}.

\medskip
\noindent\textbf{Step 2.}
Let $\{(\omega_1^k,\omega_3^k)\}_{k\in I}$ be the values obtained
in the previous step, corresponding to solutions of
$G_{\omega_1^k,0,\omega_3^k}(x,y)
=
(x_1(\omega_0),y_1(\omega_0)).$ We claim that
\begin{equation}\label{proofdisplacement}
i\ne j
\quad\Longrightarrow\quad
|\omega_1^i-\omega_1^j|\gtrsim L^{-1}.
\end{equation}
Write $G_{0,0,\omega_3^k}(x,y)=(x_3^k,y_3^k).$ Then, for $i\ne j$, $y_3^i=y_3^j=y_1(\omega_0).$ The corresponding preimages $\omega_3^i,\omega_3^j\in H_L^c$ lie on distinct monotone branches. Hence,
\[
|\omega_3^i-\omega_3^j|
=
|x_3^i-x_3^j|
\gtrsim L^{-1}.
\]
By \eqref{displacement},
\[
|\omega_1^i-\omega_1^j|
=
|x_3^i-x_3^j|
\gtrsim L^{-1},
\]
which proves \eqref{proofdisplacement}.

\medskip
\noindent\textbf{Step 3.}
We claim that, if $\omega_0\in\tilde A_{(x_0,y_0)}^{(x,y)}$
and $\{\omega_1^i\}_{i\in I}$ is the projection into the $\omega_1$-coordinates of the pairs identified in step $1$, then there exists $c>0$ such that
\begin{equation}\label{lastclaim}
B_{cL^{-1-\delta}}(\omega_1^i)
\subset
B_{(x_0,y_0)}^{(x,y)}(\omega_0),
\qquad
\forall\,i\in I.
\end{equation}
By Step~2, these balls are pairwise disjoint.
Together with the count from Step~1, this claim will
complete the proof of the lemma.

Consider local $XY$ coordinates centred at
$(x_1(\omega_0),y_1(\omega_0))$.
After discarding $O(1)$ values of $\omega_1^i$, we can restrict to those $i\in I$ such that
\[
(x_1(\omega_0),y_1(\omega_0))
=
G_{\omega_1^i,0,\omega_3^i}(x,y)
\]
and there is a small interval $\bar J_i$ around $\omega_3^i$,
disjoint from $H_L$, whose endpoints are mapped by
$\omega_3\mapsto G_{\omega_1^i,0,\omega_3}$ to points with
vertical coordinates $-L^{-\delta}$ and $L^{-\delta}$ in
the new coordinates. Discarding a further $O(1)$ indices if necessary, we may also assume that $\operatorname{dist}\bigl(
\omega_1^i,\partial[-L^{-\delta},L^{-\delta}]
\bigr)
\ge L^{-1-\delta}.$

Writing, for $\omega_3\in\bar J_i$, $(X(\omega_3),Y(\omega_3)):=G_{\omega_1^i,0,\omega_3}(x,y),$ we have, by \eqref{inversemapF}
\[
L^{99/100}\lesssim|Y'(\omega_3)|\lesssim L,
\qquad
|X'(\omega_3)|=1.
\]
As a result,
\[
L^{-1-\delta}\lesssim|\bar J_i|
\lesssim L^{-99/100-\delta}.
\]
Choose $c>0$ such that $B_{cL^{-1-\delta}}(\omega_3^i)\subset\bar J_i$ and, if $\omega_3\in B_{cL^{-1-\delta}}(\omega_3^i)$,
\[
|X(\omega_3)|\lesssim L^{-1-\delta},
\qquad
|Y(\omega_3)|\le L^{-\delta}.
\]

We now prove \eqref{lastclaim}.
Fix $\omega_1\in B_{cL^{-1-\delta}}(\omega_1^i)$ and choose
$\omega_3\in\bar J_i$ such that
$X(\omega_3)=\omega_1-\omega_1^i$.
Since $|Y(\omega_3)|\le L^{-\delta}$, the choice $\omega_2=-Y(\omega_3)$  gives
\[
(x_1(\omega_0),y_1(\omega_0))
=
G_{\omega_1,\omega_2,\omega_3}(x,y).
\]
This proves \eqref{lastclaim}.

\end{proof}

\begin{proof}[Proof of Proposition~\ref{uniform_goodness}]
(i) If $\omega_1$ corresponds to a bad value
$\omega_3\in\mathcal C_3^{\mathrm{aug}}$, then the first
coordinate equation gives $x_1(\omega_0)+\omega_1=H_1(x,y)+\omega_3.$ Hence, $\omega_1
\in
H_1(x,y)-x_1(\omega_0)+\mathcal C_3^{\mathrm{aug}}.$ Thus the set of bad $\omega_1$'s is contained in a translate of $\mathcal C_3^{\mathrm{aug}}$, and therefore has measure bounded by $\operatorname{Leb}(\mathcal C_3^{\mathrm{aug}})$.

(ii) The proof is the same as Step 3 of the proof of
Lemma \ref{lemma_minor}.
\end{proof}

\section{Comparison of the conditional slope laws}
\label{app:slope-comparison}
\begin{proof}[Proof of Proposition~\ref{Llarge}]

Fix $(x_0,y_0),(x,y)\in\mathbb T^2$, and initial slopes
$s_0,s_0'\in\mathbb P^1$. For $i=1,2$, let $s_i$ and $s_i'$ denote the slopes at time $i$
starting from $(x_0,y_0,s_0)$ and $(x_0,y_0,s_0')$,
respectively, under the same noise realization.
Set
\[
w_1=f'(x_1+\omega_1)-s_1,
\qquad
w_1'=f'(x_1+\omega_1)-s_1'.
\]
By \eqref{recursion}, $s_2=1/w_1$ and $s_2'=1/w_1'$.
Since push-forwards do not increase total variation distance,
to prove Proposition~\ref{Llarge} it suffices to show that
\begin{equation}\label{newgoal}
\|p_1(\omega_0)-p_1'(\omega_0)\|_{\mathrm{TV}}
:=
\left\|
\mathbb P\left(
w_1\in\cdot
\,\middle|\,
\Omega_{(x_0,y_0)}^{(x,y)},\omega_0
\right)
-
\mathbb P\left(
w_1'\in\cdot
\,\middle|\,
\Omega_{(x_0,y_0)}^{(x,y)},\omega_0
\right)
\right\|_{\mathrm{TV}}
\lesssim L^{-\delta},
\end{equation}
uniformly over $\omega_0\in\tilde G$, where $\tilde G$
is defined in \eqref{ineedtis}.

The proof proceeds in three steps.

\noindent
\textbf{Step 1: Decomposition into subprobability measures.}

For $\omega_0\in\tilde G$, let $p_1(\omega_0)$ and
$p_1'(\omega_0)$ be as defined above, and let $Q_{\omega_0}$
denote the conditional law of $\omega_1$ given
$\Omega_{(x_0,y_0)}^{(x,y)}$ and $\omega_0$, so that $p_1(\omega_0)=(w_1)_*Q_{\omega_0},$ $p_1'(\omega_0)=(w_1')_*Q_{\omega_0}.$

Define the bad set
\[
C(\omega_0):=
\left\{
\omega_1\in B_{(x_0,y_0)}^{(x,y)}(\omega_0):
\begin{array}{l}
|f''(x_1+\omega_1)|<L^{2/3},\ \text{or}\\
\omega_3\in\mathcal C_3^{\mathrm{aug}},\ \text{or}\\
\operatorname{dist}(\omega_1,\partial[-L^{-\delta},L^{-\delta}])
<c'L^{-1-\delta}
\end{array}
\right\},
\]
where $\mathcal C_3^{\mathrm{aug}}$ is defined in \eqref{aug}
and $\omega_3$ is the unique final noise coordinate
determined by $(\omega_0,\omega_1)$ and the endpoint constraint.

Then we can write
\begin{equation}\label{pgpb}
p_1(\omega_0)=p_1^g(\omega_0)+p_1^b(\omega_0),
\qquad
p_1'(\omega_0)=p_1^{\prime,g}(\omega_0)+p_1^{\prime,b}(\omega_0),
\end{equation}
where  $p_1^b(\omega_0)
=(w_1)_*\bigl(Q_{\omega_0}|_{C(\omega_0)}\bigr),$ $p_1^{\prime,b}(\omega_0)
=(w_1')_*\bigl(Q_{\omega_0}|_{C(\omega_0)}\bigr)$ and the remaining measures $p_1^g(\omega_0)$ and $p_1^{\prime,g}(\omega_0)$
are then defined by \eqref{pgpb}.
The conditional density formula in Proposition~\ref{from_normalized_to_others},
\eqref{conditionalothers}, gives
\begin{align*}
p_1^b(\omega_0)(\mathbb R)+p_1^{\prime,b}(\omega_0)(\mathbb R)
&\lesssim Q_{\omega_0}(C(\omega_0))
\le\frac{|C(\omega_0)|}{|B_{(x_0,y_0)}^{(x,y)}(\omega_0)|}\\
&\lesssim
L^{1/100+3\delta}
\left(
L^{-1/3}+L^{-1-\delta}+L^{-1/2}
\right)
\lesssim L^{-\delta},
\end{align*}
where the term $L^{-1/2}$ comes from Proposition~\ref{uniform_goodness}.
Therefore, by \eqref{pgpb},
\begin{equation}\label{pigpigiprime}
\|p_1(\omega_0)-p_1'(\omega_0)\|_{\mathrm{TV}}
\lesssim
L^{-\delta}
+
\|p_1^g(\omega_0)-p_1^{\prime,g}(\omega_0)\|_{\mathrm{TV}}.
\end{equation}
\medskip
\noindent

\textbf{Step 2: Construction of the coupling partition.}
By Proposition \ref{uniform_goodness}, after enlarging
$C(\omega_0)$ to include all residual components of
$B_{(x_0,y_0)}^{(x,y)}(\omega_0)\setminus C(\omega_0)$ whose lengths are less
than a sufficiently small fixed multiple of $L^{-1-\delta}$,
we can write
\[
B_{(x_0,y_0)}^{(x,y)}(\omega_0)\setminus C(\omega_0)
=
\bigcup_{j\in J}I_j,
\qquad
|I_j|\gtrsim L^{-1-\delta}.
\]
The additional discarded probability mass is $O(L^{-\delta})$,
so this enlargement does not affect the bound in
\eqref{pigpigiprime}.

We now trim these intervals $I_j$. More precisely, for each $j$, let
\begin{equation}\label{translations}
U_j:=f'(x_1+I_j)-s_1,
\qquad
U_j':=f'(x_1+I_j)-s_1'    
\end{equation}
be the images of $I_j$ under $w_1$ and $w_1'$, respectively and define the sets
\begin{align*}
\widetilde U_j&:=U_j\cap U_j',\\
\widetilde I_j
&:=
\left\{
\omega_1\in I_j:w_1(\omega_1)\in\widetilde U_j
\right\},\\
\widetilde I_j'
&:=
\left\{
\omega_1\in I_j:w_1'(\omega_1)\in\widetilde U_j
\right\}.
\end{align*}

The restricted branches $\widetilde I_j$ and
$\widetilde I_j'$ have the same image $\widetilde U_j$
under $w_1$ and $w_1'$, respectively.

By \eqref{translations}, we have $U_j=U'_j+s'_1-s_1.$ Since $|s_1-s_1'|\le L^{-4/3},$ because $\omega_0\in\tilde G$, and $|f''(x_1+\omega_1)|\ge L^{2/3}$ on $I_j$, because $\omega_1\notin C(\omega_0),$ 
we have that $|I_j\setminus\widetilde I_j|
+
|I_j\setminus\widetilde I_j'| =O(L^{-2}).$
Consequently, the total discarded $Q_{\omega_0}$-probability mass is bounded by
\[ 
\frac{
\sum_j
\bigl(
|I_j\setminus\widetilde I_j|
+
|I_j\setminus\widetilde I_j'|
\bigr)
}{
|B_{(x_0,y_0)}^{(x,y)}(\omega_0)|
}
\lesssim
L^{-1+\delta}
\frac{\sum_j|I_j|}{|B_{(x_0,y_0)}^{(x,y)}(\omega_0)|}
\le
L^{-1+\delta}.
\]
By \eqref{pigpigiprime} and the trimming estimate,
\begin{equation}\label{p1minusp1prime}
\|p_1(\omega_0)-p_1'(\omega_0)\|_{\mathrm{TV}}
\lesssim
L^{-\delta}
+
\sum_j\int_{\widetilde U_j}
|p_{1,j}^g(\omega_0,u)-p_{1,j}^{\prime,g}(\omega_0,u)|\,du,
\end{equation}
where $p_{1,j}^g(\omega_0,\cdot)$ is the density of
$(w_1)_*(Q_{\omega_0}|_{\widetilde I_j})$, and
$p_{1,j}^{\prime,g}(\omega_0,\cdot)$ is defined analogously
using $w_1'$ and $\widetilde I_j'$.

\medskip
\noindent
\textbf{Step 3: Conclusion of the proof.}

We now claim that, for every $u\in\widetilde U_j$,
\begin{equation}\label{plsletmefinish}
|p_{1,j}^g(\omega_0,u)-p_{1,j}^{\prime,g}(\omega_0,u)|
\lesssim
\frac{L^{-1}|s_1-s_1'|}
{|B_{(x_0,y_0)}^{(x,y)}(\omega_0)|}.
\end{equation}

Assuming this estimate holds, summing over all branches gives
\begin{align*}
\sum_j\int_{\widetilde U_j}
|p_{1,j}^g(\omega_0,u)-p_{1,j}^{\prime,g}(\omega_0,u)|\,du
&\lesssim
\sum_j
\frac{L^{-1}|s_1-s_1'|}
{|B_{(x_0,y_0)}^{(x,y)}(\omega_0)|}
|\widetilde U_j|\\
&\lesssim
\sum_j
\frac{|s_1-s_1'|}
{|B_{(x_0,y_0)}^{(x,y)}(\omega_0)|}
|\widetilde I_j|\\
&\le |s_1-s_1'|
\le L^{-4/3},
\end{align*}
where the second inequality follows from
$|\widetilde U_j|\lesssim L|\widetilde I_j|$.
Together with \eqref{p1minusp1prime},
\eqref{pigpigiprime}, and \eqref{newgoal},
this proves Proposition~\ref{Llarge}.

To prove \eqref{plsletmefinish}, note that the
change-of-variables formula gives, for $u\in\widetilde U_j$,
\begin{equation}\label{densities}
\begin{aligned}
p_{1,j}^g(\omega_0,u)
&=
\frac{1}{|B_{(x_0,y_0)}^{(x,y)}(\omega_0)|}
\frac{1}{|f''(x_1+\omega_j(u))|},\\
p_{1,j}^{\prime,g}(\omega_0,u)
&=
\frac{1}{|B_{(x_0,y_0)}^{(x,y)}(\omega_0)|}
\frac{1}{|f''(x_1+\omega_j'(u))|},
\end{aligned}
\end{equation}
where $\omega_j(u)\in\widetilde I_j$ and
$\omega_j'(u)\in\widetilde I_j'$ are the unique preimages
of $u$ under $w_1$ and $w_1'$, respectively.

Since
\[
f'(x_1+\omega_j'(u))-f'(x_1+\omega_j(u))
=s_1'-s_1
\]
and $|f''|\ge L^{2/3}$ on $x_1+I_j$, the mean value theorem gives
\[
|\omega_j'(u)-\omega_j(u)|
\le L^{-2/3}|s_1-s_1'|.
\]
Together with $|f'''|\lesssim L$, this implies
\[
|f''(x_1+\omega_j'(u))-f''(x_1+\omega_j(u))|
\lesssim L^{1/3}|s_1-s_1'|.
\]
Using the lower bound on $|f''|$ once more, we obtain
\[
\left|
\frac{1}{|f''(x_1+\omega_j(u))|}
-
\frac{1}{|f''(x_1+\omega_j'(u))|}
\right|
\lesssim
\frac{L^{1/3}|s_1-s_1'|}{L^{4/3}}
=
L^{-1}|s_1-s_1'|.
\]
Substituting into \eqref{densities} proves
\eqref{plsletmefinish}.
\end{proof}

\section{Proof of the Lyapunov estimates}
\label{app:two_point_lyapunov}
In this appendix, we prove estimates \eqref{first_lyapunov} and \eqref{concoutsidediagonal} from
Proposition~\ref{Lyapunov function}. In order to do that, we need the following large-deviation estimate.
Let $(x_n,y_n),(\bar x_n,\bar y_n)$ satisfy
\eqref{twopointmotion}.
For the following local estimates, choose compatible real
lifts of the initial points and iterate the lifted recurrence.
Define
\begin{equation}\label{uenvienden}
u_n:=x_n+\bar x_n,
\qquad
v_n:=x_n-\bar x_n,
\qquad
d_n:=d_{\mathbb{T}}(x_n,\bar x_n).
\end{equation}
Then
\begin{equation}\label{unnevienne}
\begin{aligned}
u_n
&=
2L\cos\bigl(\pi(u_{n-1}+2\omega_{n-1})\bigr)
\cos(\pi v_{n-1})
-(y_{n-1}+\bar y_{n-1}),\\
v_n
&=
-2L\sin\bigl(\pi(u_{n-1}+2\omega_{n-1})\bigr)
\sin(\pi v_{n-1})
-(y_{n-1}-\bar y_{n-1}),
\end{aligned}
\end{equation}
and, if $|v_n|<\frac12$, then $d_n=|v_n|$.

Given the two orbits up to time $s$, define $\xi_s$
using the next noise coordinate $\omega_s$:
\begin{equation}\label{defviesse}
\xi_s:=
\begin{cases}
\frac{99}{100},
& |v_{s+1}|\ge L^{99/100}|v_s|,\\
0,
& |v_s|\le|v_{s+1}|<L^{99/100}|v_s|,\\
-n,
& L^{-n}|v_s|\le|v_{s+1}|<L^{-n+1}|v_s|,
\quad n\ge1.
\end{cases}
\end{equation}
By definition,
\begin{equation}\label{substitution}
|v_{s+1}|\ge L^{\xi_s}|v_s|.
\end{equation}
For $i,k\ge0$, define
\begin{equation}\label{sumxi1}
S_k^i=\xi_i+\dots+\xi_{i+k}.
\end{equation}

\begin{proposicao}\label{smalldev}
There exists $\lambda>0$ such that, for every fixed $r\ge0$
and integers $i,k\ge0$ with $k\ge2r$, the following holds
for all sufficiently large $L$ and sufficiently small
$\delta>0$.

If $0<d(z_0,\bar z_0)\le L^{-i-k-2}$, then
\begin{equation}\label{distancedev}
\mathbb P\{S_k^i\le r-n\}
\lesssim L^{-\lambda(k-r+n)},
\qquad n\ge0.
\end{equation}
\end{proposicao}

\begin{proof}
Note that if $d(z_0,\bar z_0)\le L^{-i-k-2}$, then the
Lipschitz bound on $f_\omega$ gives $|v_s|\lesssim L^{-1}$
for all $0\le s\le i+k+1$.
Consequently, for sufficiently large $L$, we have
$d_s=|v_s|$.

Formula \eqref{unnevienne} and the following standard
sublevel estimate for the sine function,
\begin{equation}\label{standardfsin}
\sup_{a\in\mathbb R}
\operatorname{Leb}\left\{
x\in[0,2\pi]:|\sin x-a|\le\varepsilon
\right\}
\le C\sqrt{\varepsilon},
\qquad 0<\varepsilon\le1,
\end{equation}
for some constant $C>0$, give, for any $\gamma<1$,
\[
\mathbb P\left(
d_{s+1}\le L^\gamma d_s
\,\middle|\,
\mathcal F_0^{s-1}
\right)
\lesssim L^{\delta+(\gamma-1)/2}.
\]
Taking $\gamma=99/100$ yields
\begin{equation}\label{inequality1}
\mathbb P\left(
\xi_s=\frac{99}{100}
\,\middle|\,
\mathcal F_0^{s-1}
\right)
\ge1-O\left(L^{\delta-1/200}\right),
\end{equation}
while taking $\gamma=m+1$ gives
\begin{equation}\label{inequalitytwo}
\mathbb P\left(
\xi_s=m
\,\middle|\,
\mathcal F_0^{s-1}
\right)
\lesssim L^{\delta+m/2},
\qquad m\le-1.
\end{equation}

For some sufficiently small $\beta>0$, these estimates imply
\[
\mathbb E\left[
L^{-\beta\xi_s}
\,\middle|\,
\mathcal F_0^{s-1}
\right]
\le L^{-49\beta/50}.
\]
Successive conditioning and Markov's inequality therefore give
\[
\mathbb P\{S_k^i\le r-n\}
\le
L^{-\beta\left(\frac{49}{50}(k+1)-r+n\right)}.
\]
Since $k\ge2r$, this implies \eqref{distancedev}
with $\lambda=\beta/2$.
\end{proof}

\begin{proof}[Proof of Proposition~\ref{Lyapunov function}]
The proof is divided into two steps.

\medskip
\noindent
\textbf{Step 1: Proof of \eqref{first_lyapunov}.}

If $(z,\bar z)\in\Delta_t$, then
\begin{equation}\label{pielgpt}
\begin{aligned}
\mathcal P_2^lG_{t,p}(z,\bar z)
&\le
\mathbb E\left[
d(f_\omega^l(z),f_\omega^l(\bar z))^{-p}
\right]\\
&=
G_{t,p}(z,\bar z)
\mathbb E\left[J_l(\omega,z,\bar z)^{-p}\right],
\end{aligned}
\end{equation}
where 
\begin{equation}\label{jel}
J_l(\omega,z,\bar z)
:=
\frac{d(f_\omega^l(z),f_\omega^l(\bar z))}
{d(z,\bar z)}.
\end{equation}
Set $v:=\frac{z-\bar z}{\|z-\bar z\|}$.
Using Taylor expansion and the bound 
$\|f_\omega\|_{C^2}\lesssim L$, we obtain $J_l(\omega,z,\bar z)
=
\|Df_\omega^l(z)v\|\,\|e_l+\rho_l\|,$ where
$e_l=
\frac{Df_\omega^l(z)v}{\|Df_\omega^l(z)v\|},$ and 
$\|\rho_l\|\lesssim L^{2l-t}.$
Consequently,
\begin{equation}\label{jeielminusp}
J_l(\omega,z,\bar z)^{-p}
=
\|Df_\omega^l(z)v\|^{-p}\,
\|e_l+\rho_l\|^{-p}.
\end{equation}

We use the Taylor expansions
\begin{align*}
e^{-px}&=1-px+O(p^2x^2),\\
(1+x)^{-p}&=1+O(px),
\end{align*}
for the first and second factors in \eqref{jeielminusp},
respectively. This gives
\begin{align*}
J_l(\omega,z,\bar z)^{-p}
={}&
1-p\log\|Df_\omega^l(z)v\|\\
&+
O\left(p^2\log^2\|Df_\omega^l(z)v\|\right)
+
O(pL^{2l-t}).
\end{align*}
Choosing $p=\frac{\varepsilon_0}{\log L}$, with
$\varepsilon_0>0$ sufficiently small, we obtain
\[
J_l(\omega,z,\bar z)^{-p}
=
1-p\log\|Df_\omega^l(z)v\|
+
O(\varepsilon_0^2)+O(L^{2l-t}).
\]

By the expansion-on-average estimate in
Theorem~\ref{blubluaverage}, there exists $l_0 \in \mathbb{N},$ so that, for every $l\ge l_0$ and $L$ sufficiently large,
\[
\inf_{z,v}
\mathbb E\log\|Df_\omega^l(z)v\|
\ge\frac34\log L.
\]
Taking expectations therefore gives
\begin{align*}
\mathbb E\left[J_l(\omega,z,\bar z)^{-p}\right]
&\le
1-\frac34\varepsilon_0
+
O(\varepsilon_0^2)+O(L^{2l-t})\\
&\le
1-\frac12\varepsilon_0
=:\alpha<1,
\end{align*}
provided $\varepsilon_0$ is sufficiently small and $L$
sufficiently large. The claim follows from \eqref{pielgpt}.

\medskip
\noindent
\textbf{Step 2: Proof of \eqref{concoutsidediagonal}.}

We estimate $\mathcal P_2^lG_{t,p}(z,\bar z)$ when
$d(z_0,\bar z_0)>L^{-t}$.
Note that, in this case, either $d_0\gtrsim L^{-t-2}$ or, if this fails,
$|y_0-\bar y_0|\gtrsim L^{-t}$, in which case
\eqref{unnevienne} gives $d_1\gtrsim L^{-t}$.
Thus, either $d_0\gtrsim L^{-t-2}$ or
$d_1\gtrsim L^{-t}$.
Let us assume $d_0\ge L^{-t-2}$.
The other case is similar and is left to the reader.

Note that, if
$d(z_0,\bar z_0)>(2\pi L+1)^lL^{-t}$,
the Lipschitz bound on $f^{-1}_\omega$ gives
$d(z_k,\bar z_k)>L^{-t}$ for $0\le k\le l$.
In particular, $G_{t,p}(z_l,\bar z_l)=1$.
We may therefore assume that
\[
L^{-t}<d(z_0,\bar z_0)\le(2\pi L+1)^lL^{-t}.
\]
It follows that
\begin{equation}\label{icanapppp44}
L^{-t}<d(z_0,\bar z_0)\lesssim L^{-t+l}.
\end{equation}

By the definition of $G_{t,p}$,
\begin{equation}\label{eq:lyapunov-tail-bound}
\begin{aligned}
\mathcal P_2^lG_{t,p}(z,\bar z)
&\le
1+
\int_{\{d(z_l,\bar z_l)\le L^{-t}\}}
d(z_l,\bar z_l)^{-p}\,d\mathbb P(\omega)\\
&\le
1+
\int_{\{|v_l|\le L^{-t}\}}
|v_l|^{-p}\,d\mathbb P(\omega)\\
&\le
1+
\sum_{n=0}^\infty
\int_{\{L^{-t-n}>|v_l|>L^{-t-n-1}\}}
|v_l|^{-p}\,d\mathbb P(\omega)\\
&\le
1+
\sum_{n=0}^\infty
L^{(n+t+1)p}\,
\mathbb P\{|v_l|\le L^{-t-n}\}\\
&\le
1+\sum_{n=0}^{\infty}
L^{(n+t+1)p}
\mathbb P\{|v_l|\le L^{-n+2}|v_0|\}.
\end{aligned}
\end{equation}
By \eqref{substitution},
\[
|v_l|\ge L^{S_{l-1}^0}|v_0|,
\qquad
S_{l-1}^0=\xi_0+\dots+\xi_{l-1}.
\]
Therefore,
\[
\{|v_l|\le L^{-n+2}|v_0|\}
\subset
\{S_{l-1}^0\le-n+2\}.
\]

By \eqref{icanapppp44}, we can apply Proposition~\ref{smalldev} with $k=l-1$ and $r=2$
(so that $k\ge2r$), and obtain
\[
\mathbb P\{S_{l-1}^0\le-n+2\}
\lesssim L^{-\lambda(l-3+n)}.
\]
Substituting this estimate into
\eqref{eq:lyapunov-tail-bound} gives
\eqref{concoutsidediagonal}, provided $L$ is sufficiently
large that $p<\lambda$ and $(t+1)p<\frac{\lambda}{2}(l-3).$
\end{proof}

\section{Extending the small set}
\label{subsec:two_point_extension}

In this appendix, we prove Proposition~\ref{uniform_bound_on_T},
establishing its three items separately.

\begin{proof}[Proof of item \eqref{itemforgotten}]
Fix $(z_0,\bar z_0)\in\mathcal C$, with $\mathcal C$
as in \eqref{two_point_sublevel}, and let
$d_n,u_n,v_n$ be as in \eqref{uenvienden}.
 As argued in Appendix \ref{app:two_point_lyapunov}, in the proof of \eqref{concoutsidediagonal}, if
$d(z_0,\bar z_0)\ge L^{-2t}$, then either
$d_0\ge L^{-2(t+1)}$ or $d_1\ge L^{-2(t+1)}$.
Starting at time zero or one accordingly and reindexing,
we may therefore assume that $d_0\ge L^{-2(t+1)}$.
Let $\tau:=\min\{n\ge0:d_n>\eta\}.$ A more refined sine sub-level estimate version of \eqref{standardfsin} gives
\[
\sup_{a,\theta \in \mathbb R}\mathbb P\left(
d_{\mathbb T}\bigl(A\sin(2\pi\omega_s+\theta)+a,0\bigr)
\le r
\right)
\le Cr+CL^\delta\sqrt{\frac{r}{|A|}}, \qquad \forall A\neq 0,r\in (0,1/4).
\]

Whenever $0<d_s\le\eta$, we have $2L|\sin(\pi(x_s-\bar x_s))|\asymp Ld_s.$ The above sub-level estimate, along with \eqref{unnevienne}, gives
\begin{equation}\label{goingaroung}
\mathbb P\left(
d_{s+1}>\min\{L^{99/100}d_s,\eta\}
\,\middle|\,
\mathcal F_0^{s-1}
\right)
\ge 1-C\eta-CL^{\delta-1/200}
\ge\iota>0,
\end{equation}
for $\eta,\delta$ sufficiently small and $L$ sufficiently large.

Define
\[
G_i:=
\{\tau\le i\}
\cup
\{d_{i+1}>\min\{L^{99/100}d_i,\eta\}\},
\]
and choose $K_0\in\mathbb N$ such that $\frac{99}{100}K_0>2(t+1).$

We claim that
\begin{equation}\label{largeincusion}
\bigcap_{i=0}^{K_0-1}G_i\subset\{\tau\le K_0\}.
\end{equation}
Indeed, if all the events $G_i$ occur but $\tau>K_0$,
then $d_i\le\eta$ for every $i\le K_0$, and the definition
of $G_i$ forces $d_{i+1}>L^{99/100}d_i$.
Iterating gives
\[
d_{K_0}
\ge L^{99K_0/100}d_0
\ge L^{-2(t+1)+99K_0/100}>\eta,
\]
contradicting $\tau>K_0$.

Successive conditioning therefore gives $\mathbb P(\tau\le K_0)\ge\iota^{K_0},$ with $\iota>0$ as in \eqref{goingaroung}. Formula \eqref{prop18} then follows by the pigeonhole principle choosing $b:=\frac{\iota^{K_0}}{K_0+2}.$
\end{proof}

\begin{proof}[Proof of item \eqref{thm:twostep_smallset}]

Suppose that
$d_{\mathbb T}(x_0,\bar x_0)\ge\eta/4$.
Recall that, with $u_i:=x_i+\bar x_i$ and
$v_i:=x_i-\bar x_i$, equation~\eqref{unnevienne} gives
\[
v_i
=
-2L\sin\bigl(\pi(u_{i-1}+2\omega_{i-1})\bigr)
\sin(\pi v_{i-1})
-(y_{i-1}-\bar y_{i-1}).
\]
Since $d_{\mathbb T}(x_0,\bar x_0)\ge\eta/4,$ there exists
$C_\eta>0$ such that
\[
|v_1'(\omega_0)|
\ge C_\eta L|\cos(\pi(u_0+2\omega_0))|.
\]
Arguing as in the proof of Proposition~\ref{four-step-trans},
the same bounded distortion estimates give, for every
interval $E\subset\mathbb T$ and some constant $C_d>0$
\begin{equation}\label{distortioncoolness}
\mathbb P\bigl(v_1\bmod 1\in E\bigr)
\le \bigl(C_d|E|+C_{\eta}L^{\delta-1/2}\bigr).
\end{equation}
Fix $\rho>0$ sufficiently small and define $E_\rho:=
\left\{
v\in\mathbb T:
d_{\mathbb T}(v,0)>\rho, \,\,
d_{\mathbb T}\left(v,\frac12\right)>\rho
\right\}.$ Applying \eqref{distortioncoolness} to the
connected components of $\mathbb T \setminus E_\rho$, we obtain that
\begin{equation}\label{mathbbg1}
 \mathbb P(v_1\bmod 1\in E_\rho)
\ge 1-C_d\rho-C_\eta L^{\delta-1/2}.
\end{equation}
On $E_{\rho}$, we have $|\sin(\pi v_1)|\ge c_\rho,$ and $|\cos(\pi v_1)|\ge c_\rho.$ Recall the second-step formulas:
\begin{align*}
x_2
&=L\cos(2\pi(x_1+\omega_1))-(x_0+\omega_0),\\
\bar x_2
&=L\cos(2\pi(\bar x_1+\omega_1))-(\bar x_0+\omega_0),\\
u_2
&=2L\cos(\pi(u_1+2\omega_1))
\cos(\pi v_1)-(u_0+2\omega_0),\\
v_2
&=-2L\sin(\pi(u_1+2\omega_1))
\sin(\pi v_1)-v_0.
\end{align*}
If $v_1 \bmod1 \in E_{\rho},$ the same distortion argument gives  for every
interval $E\subset\mathbb T$,
\begin{equation}\label{distortionisamaizng}
\mathbb P(g(\omega_1)\bmod 1\in E\mid v_1 \bmod1 \in E_{\rho})
\le C_d|E|+C_\rho L^{\delta-\frac12},
\end{equation}
for $g=x_2,\bar x_2,u_2,v_2$.

Now observe that the event $\{(z_2,\bar z_2)\notin S\}$ is contained
in the union of the following four events:
\begin{align*}
B_x&:=\{|\sin(2\pi x_2)|<c\},\\
B_{\bar x}&:=\{|\sin(2\pi\bar x_2)|<c\},\\
B_u&:=\{|\cos(\pi u_2)|<\sqrt c\},\\
B_v&:=\{|\sin(\pi v_2)|<\sqrt c\}.
\end{align*}
By \eqref{distortionisamaizng} and \eqref{mathbbg1},
\begin{align*}
\mathbb P((z_2,\bar z_2)\in S)
&\ge
\mathbb P((z_2,\bar z_2)\in S,\ v_1 \bmod1 \in E_{\rho})\\
&=
\mathbb P(v_1 \bmod1 \in E_{\rho})
\mathbb P((z_2,\bar z_2)\in S\mid v_1 \bmod1 \in E_{\rho})\\
&\ge
\left(1-C_d\rho-C_\eta L^{\delta-\frac12}\right)
\left(1-C_d\sqrt c-C_\rho L^{\delta-\frac12}\right).
\end{align*}
\end{proof}
\begin{proof}[Proof of  item \eqref{lem:S_to_eta_one_step}]

Let $E:=\{v\in\mathbb T:d_{\mathbb T}(v,0)\le\eta\}.$ Since $(z_0,\bar z_0)\in S$, estimate \eqref{distortioncoolness} and the definition of $v_1$ in \eqref{uenvienden} gives
\[
\mathbb P(d_{\mathbb T}(x_1,\bar x_1)>\eta)
\ge1-C_d\eta-C_\eta L^{\delta-\frac12},
\]
which concludes the proof.
\end{proof}

\section{Local minorization}
\label{subsec:two_point_local_minorization}

In this appendix, we prove the density estimate
\eqref{foundlowerdense} in Proposition~\ref{First_step_small_set}.
Since the proof involves several steps, we begin by explaining
its main ideas.

Our aim is to show that, for every initial pair $(z_0,\bar z_0) \in S$, the six-step transition law of the
two-point motion has a density bounded below by a uniformly
positive constant almost everywhere on $C_{\mathrm{good}}$.
Here $S$ and $C_{\mathrm{good}}$ are defined in
\eqref{defSc} and \eqref{defCc}, respectively.
To this end, we first fix the middle noise variables
$\beta=(\omega_2,\omega_3)$ and consider the conditional law of the process
\[
h_\beta(\omega_0,\omega_1,\omega_4,\omega_5)(z_0,\bar z_0)
:=
\bigl(f_\omega^6(z_0),f_\omega^6(\bar z_0)\bigr).
\]
Denote the density of this conditional law by $p_\beta$.
By \eqref{piiiiii111}, the desired estimate
\eqref{foundlowerdense} follows from a positive lower bound
for $p_\beta$ almost everywhere on $C_{\mathrm{good}}$,
provided that this bound is uniform in $\beta$, the
admissible initial pair, and the parameters $L,\delta$.
The change-of-variables formula \eqref{pp111112222} expresses
this density, almost everywhere, as a sum over
the noise configurations reaching that pair:
\[
p_\beta(w,\bar w)
\asymp
L^{4\delta}
\sum_{h_\beta(\omega)=(w,\bar w)}
\frac{1}{|\det Dh_\beta(\omega)|}.
\]
We prove in \eqref{Weclaimppp} that
$|\det Dh_\beta|\lesssim L^6$.
Consequently, the desired density bound follows once we find
at least $\gtrsim L^{6-4\delta}$ solutions of
$h_\beta(\omega)=(w,\bar w)$.
This solution count is established in
Proposition~\ref{two_point_cardinality}, using the geometric
constructions of Lemmas~\ref{forward_small}
and~\ref{inverse_small}.

To count the noise configurations $\omega$ such that 
\[
f_\omega^6(z_0)=w,
\qquad
f_\omega^6(\bar z_0)=\bar w,
\]
we equivalently count, as done in Appendix \ref{app:geometric-estimates}, those for which
\begin{equation}\label{intuitveexpl}
f_\omega^2(z_0)
=
\bigl(f_{\vartheta^2\omega}^4\bigr)^{-1}(w),
\qquad
f_\omega^2(\bar z_0)
=
\bigl(f_{\vartheta^2\omega}^4\bigr)^{-1}(\bar w).    
\end{equation}
The following geometric argument allows us to count such
noise realizations. Expansion away from the critical region allows us to find
$\gtrsim L^{3-2\delta}$ pairwise disjoint branches in the
$(\omega_0,\omega_1)$ noise domain, each of which parametrizes,
under two forward iterates and on suitable lifts, an almost
horizontal surface whose projection onto the $(x,\bar x)$-torus
covers every point exactly twice.

Similarly, with $\omega_2,\omega_3$ fixed, we construct
$\gtrsim L^{3-2\delta}$ pairwise disjoint branches in the
$(\omega_4,\omega_5)$ noise domain, each of which parametrizes,
under four inverse iterates and on suitable lifts, an almost
vertical surface whose projection onto the $(y,\bar y)$-torus
covers every point exactly twice.

A transversality argument, formalized with a fixed-point construction, ensures that the image of every forward branch intersects the image of every backward branch at least once. Since the noise branches are pairwise disjoint, this yields
$\gtrsim L^{6-4\delta}$ distinct noise configurations solving \eqref{intuitveexpl}.

For graph representations, we group the coordinates in the order $(x,\bar x,y,\bar y)$. The proof of Proposition~\ref{First_step_small_set} relies on
the following two lemmas.
\begin{lema}\label{forward_small}
Let $S$ be as in \eqref{defSc}.
If $(z,\bar z)\in S$, there exist $\asymp L^{3-2\delta}$
pairwise disjoint half-open branch sets
$\mathcal G_{k,m}\subset[-L^{-\delta},L^{-\delta}]^2$
such that, for each $k,m$, the map
\begin{align*}
(\omega_0,\omega_1)&\longmapsto
\bigl(x_2(\omega_0,\omega_1),
      \bar x_2(\omega_0,\omega_1)\bigr)
\end{align*}
is two-to-one from $\mathcal G_{k,m}$ onto $\mathbb T^2$.

On suitable lifts, the image of $\overline{\mathcal G_{k,m}}$
under the full two-point map $(x_2,\bar x_2,y_2,\bar y_2)$
is the graph
\[
\left\{
\bigl(x,\bar x,\xi(x,\bar x)\bigr):
(x,\bar x)\in\mathcal D_{k,m}
\right\}\subset\mathbb R^4,
\]
where
\begin{equation}\label{mathcaldkm}
\mathcal D_{k,m}
=
\left\{
(m+u,v):
\begin{array}{l}
0\le u\le1,\\
c_{k,m}(u)\le v\le c_{k,m}(u)+2
\end{array}
\right\},
\end{equation}
for some function $c_{k,m}$ satisfying
\begin{equation}\label{cdistortion}
\|c_{k,m}'\|_\infty=O(1).
\end{equation}
Moreover, for some fixed $\rho>0$,
\begin{equation}\label{inclusionxi}
\xi(\mathcal D_{k,m})\subset[\rho,1-\rho]^2,
\end{equation}
and
\begin{equation}\label{normofxi}
\|D\xi\|=O(L^{-1}),
\end{equation}
uniformly in $k,m$.
\end{lema}

The proof of Lemma~\ref{forward_small} is deferred to the end of this appendix. 

For fixed $(\omega_2,\omega_3)$ and a terminal pair $(w,\bar w)$,
write
\[
\begin{aligned}
\bigl(x_4^-(\omega_4,\omega_5),y_4^-(\omega_4,\omega_5)\bigr)
&:=
\bigl(f_{\omega_2}^{-1}\circ f_{\omega_3}^{-1}
\circ f_{\omega_4}^{-1}\circ f_{\omega_5}^{-1}\bigr)(w),\\
\bigl(\bar x_4^-(\omega_4,\omega_5),\bar y_4^-(\omega_4,\omega_5)\bigr)
&:=
\bigl(f_{\omega_2}^{-1}\circ f_{\omega_3}^{-1}
\circ f_{\omega_4}^{-1}\circ f_{\omega_5}^{-1}\bigr)(\bar w).
\end{aligned}
\]

\begin{lema}\label{inverse_small}
Let $C_{\mathrm{good}}$ be as in \eqref{defCc}.
If $(w,\bar w)\in C_{\mathrm{good}}$, then, for every fixed
$(\omega_2,\omega_3)\in[-L^{-\delta},L^{-\delta}]^2$,
there exist $\asymp L^{3-2\delta}$ pairwise disjoint half-open
branch sets
$\mathcal G^-_{k,m}\subset[-L^{-\delta},L^{-\delta}]^2$
such that, for each $k,m$, the map

\begin{align*}
(\omega_4,\omega_5)&\longmapsto
\bigl(y_4^-(\omega_4,\omega_5),
      \bar y_4^-(\omega_4,\omega_5)\bigr)
\end{align*}

is two-to-one.

On suitable lifts, the image of $\overline{\mathcal G^-_{k,m}}$
under the full inverse two-point map
$(x_4^-,\bar x_4^-,y_4^-,\bar y_4^-)$ is the graph
\[
\left\{
\bigl(\xi^-(y,\bar y),y,\bar y\bigr):
(y,\bar y)\in\mathcal D^-_{k,m}
\right\}\subset\mathbb R^4,
\]
where
\begin{equation}\label{anotherdixi}
\mathcal D^-_{k,m}
=
\left\{
(m+u,v):
\begin{array}{l}
0\le u\le1,\\
c^-_{k,m}(u)\le v\le c^-_{k,m}(u)+2
\end{array}
\right\},
\end{equation}
for some function $c^-_{k,m}$ satisfying
\begin{equation}\label{anothercdist}
\|(c^-_{k,m})'\|_\infty=O(1).
\end{equation}
Moreover, for some fixed $\rho>0$,
\begin{equation}\label{anotheravoidance}
\xi^-(\mathcal D^-_{k,m})\subset[\rho,1-\rho]^2,
\end{equation}
and
\begin{equation}\label{normofxi_inverse}
\|D\xi^-\|=O(L^{-1}),
\end{equation}
uniformly in $k,m$ and the fixed noise variables.
\end{lema}

The above lemma is the analogue of Lemma~\ref{forward_small}
for the inverse dynamics, and its proof follows the same
strategy. Indeed, with $\omega_2,\omega_3$ fixed, the fourth
inverse iterate has an analytic form analogous to that of
the second forward iterate, with the coordinate axes
interchanged. Since we give a detailed proof of
Lemma~\ref{forward_small} at the end of this appendix,
we leave its adaptation to the inverse dynamics to the reader.

\begin{proof}[Proof of Proposition \ref{First_step_small_set}]
Fix an initial pair $(z_0,\bar z_0)\in S$.
For each fixed $\beta=(\omega_2,\omega_3)$, write $\widehat\omega:=(\omega_0,\omega_1,\omega_4,\omega_5)$
and define
\[
h_\beta^{(z_0,\bar z_0)}:
\left[-L^{-\delta},L^{-\delta}\right]^4
\longrightarrow \mathbb T^2\times\mathbb T^2,
\qquad
h_\beta^{(z_0,\bar z_0)}(\widehat\omega)
:=
\bigl(f_\omega^6(z_0),f_\omega^6(\bar z_0)\bigr).
\]
Let $\mu_\beta$ denote the law of $h_\beta^{(z_0,\bar z_0)}.$ Then
\begin{equation}\label{piiiiii111}
\bar{\mathcal P}^{\,6}_{2,(z_0,\bar z_0)}(A)
=
\frac{L^{2\delta}}{4}
\int_{\left[-L^{-\delta},L^{-\delta}\right]^2}
\mu_\beta(A)\,d\beta.
\end{equation}
For fixed $\beta$, the change-of-variables formula gives,
for almost every $w\in C_{\mathrm{good}}$,
\begin{equation}\label{pp111112222}
\frac{d\mu_\beta}{d(m\otimes m)}(w)
=
\frac{L^{4\delta}}{16}
\sum_{h_\beta(\widehat\omega)=w}
\frac{1}{|\det Dh_\beta^{(z_0,\bar z_0)}(\widehat\omega)|}
\gtrsim
\frac{L^{4\delta}\,
\#\{\widehat\omega:h_\beta^{(z_0,\bar z_0)}(\widehat\omega)=w\}}
{\sup_{\widehat\omega}|\det Dh_\beta^{(z_0,\bar z_0)}(\widehat\omega)|}.
\end{equation}
Formula \eqref{foundlowerdense} follows once we prove that, there are constants $c_*,C_{\text{det}}>0,$  such that for all initial pairs $(z_0,\bar z_0)\in S$, $\beta \in \left[-L^{-\delta},L^{-\delta}\right]$ and $w \in C_{\text{good}}$
\begin{equation}\label{conditional_solution_count}
\#\{\widehat\omega:h_\beta(\widehat\omega)=w\}
\ge c_*L^{6-4\delta},
\end{equation}
and
\begin{equation}\label{Weclaimppp}
|\det Dh_\beta(\widehat\omega)|
\le C_{\mathrm{det}}L^6,
\end{equation}

To prove \eqref{Weclaimppp}, write
\[
A_n=
\begin{pmatrix}
a_n&-1\\
1&0
\end{pmatrix},
\qquad
\mathcal A_n=\operatorname{diag}(A_n,\bar A_n),
\qquad
\mathbf v_n=(a_n,1,\bar a_n,1)^T,
\]
where
$a_n=f'(x_n+\omega_n)$ and
$\bar a_n=f'(\bar x_n+\omega_n)$.
We have $|a_n|,|\bar a_n|\le2\pi L$.

By the chain rule, the columns of $Dh_\beta$, corresponding to
$\omega_0,\omega_1,\omega_4,\omega_5$, are
\[
Dh_\beta=
\left[
\mathcal A_5\mathcal A_4\mathcal A_3\mathcal A_2\mathcal A_1\mathbf v_0
\;\middle|\;
\mathcal A_5\mathcal A_4\mathcal A_3\mathcal A_2\mathbf v_1
\;\middle|\;
\mathcal A_5\mathbf v_4
\;\middle|\;
\mathbf v_5
\right].
\]
Since $\det\mathcal A_n=1$, multiplication on the left by
$(\mathcal A_5\mathcal A_4)^{-1}$ gives
\[
\det Dh_\beta=\det[p,q,e,r],
\]
where
\[
p=\mathcal A_3\mathcal A_2\mathcal A_1\mathbf v_0,
\qquad
q=\mathcal A_3\mathcal A_2\mathbf v_1,
\]
and
\[
e=(1,0,1,0)^T,
\qquad
r=(0,-1,0,-1)^T.
\]
Here we used
$\mathcal A_n^{-1}\mathbf v_n=e$ and
$\mathcal A_n^{-1}e=r$.

Writing $p=(p_x,p_y,\bar p_x,\bar p_y)^T$ and similarly for $q$,
subtraction of the first two rows from the last two yields
\[
|\det Dh_\beta|
=
\left|
\det
\begin{pmatrix}
p_x-\bar p_x&q_x-\bar q_x\\
p_y-\bar p_y&q_y-\bar q_y
\end{pmatrix}
\right|.
\]
The form of $A_n$ and the bound $|a_n|\le2\pi L$ imply
\[
p_x,\bar p_x=O(L^4),
\qquad
p_y,\bar p_y=O(L^3),
\]
and
\[
q_x,\bar q_x=O(L^3),
\qquad
q_y,\bar q_y=O(L^2),
\]
with constants independent of $L$. Consequently,
\begin{align*}
|\det Dh_\beta|
&\le
|p_x-\bar p_x|\,|q_y-\bar q_y|
+
|p_y-\bar p_y|\,|q_x-\bar q_x|\\
&\le C_{\mathrm{det}}L^6,
\end{align*}
which proves \eqref{Weclaimppp}.

Formula \eqref{conditional_solution_count} follows from the following Proposition.
\begin{proposicao}\label{two_point_cardinality}
Let $(z,\bar z)\in S$ and $(w,\bar w)\in C_{\mathrm{good}}$.
For every fixed $(\omega_2,\omega_3)\in\left[-L^{-\delta},L^{-\delta}]\right]^2$,
the equations
\begin{equation}\label{whatineedtosolve}
f_\omega^6(z)=w,
\qquad
f_\omega^6(\bar z)=\bar w
\end{equation}
have at least $c_*L^{6-4\delta}$ solutions in the variables
$(\omega_0,\omega_1,\omega_4,\omega_5)$, where $c_*>0$
is independent of $L,\delta$.
\end{proposicao}

\begin{proof}
Fix $\omega_2,\omega_3\in[-L^{-\delta},L^{-\delta}]$,
$(z,\bar z)\in S$, and $(w,\bar w)\in C_{\mathrm{good}}$.
By Lemma~\ref{forward_small}, there are
$\asymp L^{3-2\delta}$ pairwise disjoint forward branches
$\mathcal G_{k,m}$. Similarly, by Lemma~\ref{inverse_small},
there are $\asymp L^{3-2\delta}$ pairwise disjoint inverse
branches $\mathcal G^-_{j,m'}$.
To prove Proposition~\ref{two_point_cardinality}, it therefore
suffices to show that, for every choice of a forward branch
$\mathcal G_{k,m}$ and an inverse branch $\mathcal G^-_{j,m'}$,
there is at least one solution of \eqref{whatineedtosolve}
with $(\omega_0,\omega_1)\in\mathcal G_{k,m}$ and
$(\omega_4,\omega_5)\in\mathcal G^-_{j,m'}$.

Let us then fix a forward branch $\mathcal G_{k,m}$ and an inverse branch
$\mathcal G^-_{j,m'}$, with corresponding graph domains  $\mathcal D_{k,m},\mathcal D^-_{j,m'}$  and graph maps $\xi,\xi^-$.

By \eqref{mathcaldkm}, \eqref{anotherdixi}, and the bounds
for $c_{k,m},c^-_{j,m'}$ in \eqref{cdistortion} and
\eqref{anothercdist}, both $\mathcal D_{k,m}$ and
$\mathcal D^-_{j,m'}$ have uniformly bounded
diameter. Together with the derivative bounds
\eqref{normofxi} and \eqref{normofxi_inverse}, this gives
\begin{equation}\label{ldiam}
\operatorname{diam}\xi(\mathcal D_{k,m})
+
\operatorname{diam}\xi^-(\mathcal D^-_{j,m'})
=O(L^{-1}).
\end{equation}

Choose $x_*\in\xi^-(\mathcal D^-_{j,m'})$ and
$y_*\in\xi(\mathcal D_{k,m})$. Since, by \eqref{inclusionxi} and \eqref{anotheravoidance},
both points lie in $[\rho,1-\rho]^2$, the explicit description for $\mathcal D_{k,m},\mathcal D^-_{j,m'}$ in \eqref{mathcaldkm} and \eqref{anotherdixi}, together with
the uniform derivative bounds \eqref{cdistortion} and
\eqref{anothercdist}, imply that there exist integer vectors $\nu_+,\nu_-\in\mathbb Z^2$ and a fixed $r>0$ such that
\[
\nu_++\overline{B_r(x_*)}\subset\operatorname{int}\mathcal D_{k,m},
\qquad
\nu_-+\overline{B_r(y_*)}\subset\operatorname{int}\mathcal D^-_{j,m'}.
\]
Define
\[
\Phi(x):=\xi(x+\nu_+),
\qquad
\Psi(y):=\xi^-(y+\nu_-).
\]
By \eqref{ldiam}, for $L$ sufficiently large,
\[
\Phi\bigl(\overline{B_r(x_*)}\bigr)\subset B_r(y_*),
\qquad
\Psi\bigl(\overline{B_r(y_*)}\bigr)\subset B_r(x_*).
\]
Moreover, the derivative bounds give
\[
\operatorname{Lip}(\Psi\circ\Phi)=O(L^{-2})
\]
on $\overline{B_r(x_*)}$.
Thus, $\Psi\circ\Phi$ is a contraction of this closed ball
into itself. Its unique fixed point determines pre-images
$(\omega_0,\omega_1)\in\mathcal G_{k,m}$ and
$(\omega_4,\omega_5)\in\mathcal G^-_{j,m'}$
satisfying \eqref{whatineedtosolve} for the fixed
$\omega_2,\omega_3$.
\end{proof}
\end{proof}

We now prove Lemma~\ref{forward_small}, completing the branch
construction used in Appendix~\ref{subsec:two_point_local_minorization}.
To this end, we construct the half-open branches $\mathcal G_{k,m}$
required by the lemma and establish their claimed properties.
The proof proceeds in five steps, whose main arguments we
briefly outline below.

We first construct a family of intervals $H_k$ with the following properties:
\begin{itemize}
\item The family consists of $\asymp L^{2-\delta}$ intervals,
each of length comparable to $L^{-2}$ (Step~1).
\item For each $k$, after passing to  a real lift $X_2$ of $x_2,$ the intervals
$X_2(\omega_0,[-L^{-\delta},L^{-\delta}])$,
$\omega_0\in H_k$, have a common intersection of length
$\asymp L^{1-\delta}$. Taking the integers $m$ for which
$[m,m+1]$ lies in this intersection yields a set $M_k$
of $\asymp L^{1-\delta}$ admissible integers (Step~2).
\end{itemize}
For fixed $k$, an admissible integer $m\in M_k$, and
$u\in[0,1)$, Step~3 constructs a function $h_m^u$ satisfying
$X_2(\omega_0,h_m^u(\omega_0))=u+m.$  Using the lower bound on the length of $H_k$ provided in Step~1, we then
show that the map
$\omega_0\mapsto\bar x_2(\omega_0,h_m^u(\omega_0))$
covers the circle at least twice as $\omega_0$ varies over $H_k$.
These level sets are used in Step~4 to define
$\mathcal G_{k,m}$ and prove that they are pairwise disjoint
and that the map $(\omega_0,\omega_1)\mapsto(x_2,\bar x_2)$
is two-to-one from each $\mathcal G_{k,m}$ onto $\mathbb T^2$.
Finally, Step~5 establishes the graph representation
\eqref{mathcaldkm} and the bounds \eqref{cdistortion},
\eqref{inclusionxi}, and \eqref{normofxi}.

\begin{proof}[Proof of Lemma \ref{forward_small}]
The proof is divided into five steps.

In the following, we fix representatives
$x_0,y_0,\bar x_0,\bar y_0\in[0,1)$ of the initial conditions.
For $i=1,2$, we denote by $x_i,y_i,\bar x_i,\bar y_i\in[0,1)$
the representatives of the torus iterates, and by
$X_i,Y_i,\bar X_i,\bar Y_i\in\mathbb R$ their real-valued lifts.

\medskip
\noindent\textbf{Step 1: Construction of the intervals $H_k$.}

Given $\eta>0$, define the intervals
\begin{equation}\label{i1i2}
I_1(\eta):=\left[\eta,\frac12-\eta\right],
\qquad
I_2(\eta):=\left[\frac12+\eta,1-\eta\right].
\end{equation}
Fix $(z_0,\bar z_0)\in S$. We construct a family of intervals $\{H_k\}_{k \in K}$ as follows.
\begin{itemize}
\item Discarding at most two boundary pieces, we first partition $[-L^{-\delta},L^{-\delta}]$ into
$\asymp L^{1-\delta}$ intervals $I_j$, each mapped onto the
full circle by $x_1$. Since, $|X_1'|,|\bar X_1'|\ge \frac{\pi c}{2}L$ for $(z_0,\bar z_0)\in S$ and $L$ large enough, we may shorten each $I_j$, if necessary, so that its
image under $\bar X_1$ has length at most $1$, while $I_j$
itself has length at least $c'L^{-1}$ for some uniform $c'>0$.

\item We then remove from $I_j$ all points for which
$x_1\notin I_1(\eta)\cup I_2(\eta)$ or
$\bar x_1\notin I_1(\eta')\cup I_2(\eta')$,
and retain only those remaining interval whose
length is at least $hL^{-1}$.
Since the images of $I_j$ under both $X_1$ and $\bar X_1$
have length at least $2\pi cc'$, at least one such component
remains provided $\eta,\eta'$ and $h$ are chosen
sufficiently small.
\item Finally, we subdivide each remaining interval into half-open subintervals
of length between $\widetilde cL^{-2}$ and
$2\widetilde cL^{-2}$, where $\widetilde c>1$ will be fixed at the end of step $3$, and remove the short boundary components. We denote the resulting family by $\{H_k\}$;
it consists of $\asymp L^{2-\delta}$ intervals.
\end{itemize}

\medskip
\noindent\textbf{Step 2: A uniform family of admissible integers.}

Define the set of admissible integers by
\begin{equation}\label{admissibleset}
M_k:=\left\{
m\in\mathbb Z:[m,m+1]\subset
\bigcap_{\omega_0\in H_k}X_2(\omega_0,\left[-L^{-\delta},L^{-\delta}]\right])
\right\}.
\end{equation}
In this step, we show that
\begin{equation}\label{sizeofnumbers}
|M_k|\asymp L^{1-\delta}.
\end{equation}
Indeed, fix $k\in K$ and consider $H_k$ as defined in the previous step. By construction, for every $\omega_0\in H_k$,
we have $X_1\bmod 1\in I_1(\eta)\cup I_2(\eta)$. Thus the map
\[
\omega_1\longmapsto X_2(\omega_0,\omega_1)
=f(X_1+\omega_1)-Y_1
\]
is monotone and satisfies $|\partial_{\omega_1}X_2| \asymp L.$ Hence, for every fixed $\omega_0\in H_k$, the interval $X_2(\omega_0,\left[-L^{-\delta},L^{-\delta}]\right])$ has length
$\asymp L^{1-\delta}$.

Furthermore, observe that if $(x_1,y_1)=f_{\omega_0}(z_0)$ and
$(\widetilde x_1,\widetilde y_1)=f_{\widetilde\omega_0}(z_0)$,
with $\omega_0,\widetilde\omega_0\in H_k$, then
$|X_1-\widetilde X_1|\lesssim L^{-1}$ and
\[
\bigl|f(X_1+\omega_1)-Y_1
-f(\widetilde X_1+\omega_1)+\widetilde Y_1\bigr|
\lesssim L|X_1-\widetilde X_1|+O(1)=O(1).
\]
Consequently, for every $\omega_0\in H_k$, the interval
$X_2(\omega_0,[-L^{-\delta},L^{-\delta}])$ contains $X_2(\widetilde\omega_0,[-L^{-\delta},L^{-\delta}])$
after removing a segment of uniformly bounded length from
each end of the latter, which proves \eqref{sizeofnumbers}.

\medskip
\noindent\textbf{Step 3: The dynamics of $\bar X_2$ along a level set of $X_2$.}
Given $k\in K$, fix $m\in M_k$ and $u\in\mathbb T$,
represented in $[0,1)$. On $H_k$, translate the lifts
$X_1,\bar X_1$ by constant integers so that they take
values in $(0,1)$.

We record two facts that will be used in the next step to construct the sets $\mathcal{G}_{k,m}.$

Define the functions
\begin{align*}
\tau_{m,1}(u,\omega_0)
&=\frac{1}{2\pi}\arccos\left(\frac{u+m+x_0+\omega_0}{L}\right),\\
\tau_{m,2}(u,\omega_0)
&=1-\frac{1}{2\pi}\arccos\left(\frac{u+m+x_0+\omega_0}{L}\right),
\end{align*}

The first fact is that, if $\omega_0\in H_k,$ then
\begin{equation}\label{x2equalsu}
X_2(\omega_0,\omega_1)=u+m \iff \omega_1=h_m^u(\omega_0)
:=\tau_{m,i}(u,\omega_0)-X_1(\omega_0),
\end{equation}
where the index $i$ is determined by whether $X_1$ lies in $I_1(\frac{\eta}{2})$ or
$I_2(\frac{\eta}{2})$.

This follows from the fact that, if $\omega_0\in H_k$ and $L$ large enough, $X_1(\omega_0)+\omega_1\in I_1(\frac{\eta}{2})\cup I_2(\frac{\eta}{2})$, and $f$ is invertible on each of these intervals.

Then we show that, if we define
\[
\phi_u(\omega_0)
:=\bar X_2(\omega_0,h_m^u(\omega_0))
=f(\bar X_1+h_m^u(\omega_0))-\bar x_0-\omega_0.
\]
then we can choose $\widetilde c$ in the definition of $H_k$ in step $1$ so that
\begin{equation}\label{secondfact}
|\phi_u(H_k)|>2
\end{equation}
In particular, the projection of $\phi_u(H_k)$ onto $\mathbb T$ covers the full circle at least twice. This is just a computation. Indeed, differentiating \eqref{x2equalsu} with respect to $\omega_0$ gives
\[
(h_m^u)'(\omega_0)
=-f'(x_0+\omega_0)+\frac{1}{f'(X_1+\omega_1)}.
\]

Differentiating $\phi_u$ by the chain rule and substituting the expression for $(h_m^u)'$ obtained above, we get
\begin{align*}
\phi_u'(\omega_0)
&=f'(\bar X_1+h_m^u(\omega_0))
\left(\bar X_1'(\omega_0)+(h_m^u)'(\omega_0)\right)-1\\
&=f'(\bar X_1+h_m^u(\omega_0))
\left(
f'(\bar x_0+\omega_0)-f'(x_0+\omega_0)
+\frac{1}{f'(X_1+\omega_1)}
\right)-1.
\end{align*}
Because $\bar X_1\in I_1(\eta')\cup I_2(\eta')$ and
$(z_0,\bar z_0)\in S$, we have
$|\phi_u'(\omega_0)|\ge c'L^2$ for some $c'>0$.
Choosing $\widetilde c$ sufficiently large that
$c'\widetilde c>2$, we then have $|\phi_u(H_k)|>2,$ which proves \eqref{secondfact}.

\noindent\textbf{Step 4: Defining the sets $\mathcal G_{k,m}$}

Fix $k$ and $m\in M_k$. For every $u\in\mathbb T$,
let $H_k(m,u)$ denote the uniquely defined half-open interval with the same left
endpoint as $H_k$ such that
the projection of $\phi_u(H_k(m,u))$ onto $\mathbb T$ covers the circle exactly twice. Define
\begin{equation}\label{bigG}
\mathcal G_{k,m}
:=\bigcup_{u\in\mathbb T}
\left\{(\omega_0,h_m^u(\omega_0)):
\omega_0\in H_k(m,u)\right\}.
\end{equation}
By \eqref{sizeofnumbers}, and since there are
$\asymp L^{2-\delta}$ intervals $H_k$, there are
$\asymp L^{3-2\delta}$ sets of this form.

To complete the proof of Lemma~\ref{forward_small}, we need
to show that the sets $\mathcal G_{k,m}$ are pairwise disjoint,
that the map
$(\omega_0,\omega_1)\mapsto
(x_2(\omega_0,\omega_1),\bar x_2(\omega_0,\omega_1))$
is two-to-one from each $\mathcal G_{k,m}$ onto $\mathbb T^2$,
and that \eqref{mathcaldkm}, \eqref{cdistortion},
\eqref{inclusionxi}, and \eqref{normofxi} hold.
In this step, we prove the first two assertions; the remaining
properties are established in the next step.

We first show that these sets are pairwise disjoint. Suppose that
\[
(\omega_0,\omega_1)\in
\mathcal G_{k,m}\cap\mathcal G_{k',m'}.
\]
Then $\omega_0\in H_k\cap H_{k'}$, so $k=k'$. Moreover,
\eqref{bigG} gives
\[
\omega_1=h_m^u(\omega_0)=h_{m'}^{u'}(\omega_0)
\]
for some $u,u'\in\mathbb T$. Applying $X_2$ yields
\[
u+m=X_2(\omega_0,h_m^u(\omega_0))
=X_2(\omega_0,h_{m'}^{u'}(\omega_0))=u'+m'.
\]
Since $u,u'\in[0,1)$ represent the fractional parts and $m,m'$
the integer parts, we obtain $u=u'$ and $m=m'$.

Next, we show that the map
\[
\mathcal G_{k,m}\ni(\omega_0,\omega_1)
\longmapsto
\bigl(x_2(\omega_0,\omega_1),\bar x_2(\omega_0,\omega_1)\bigr)
\in\mathbb T^2
\]
is two-to-one, while the real-valued map $(X_2,\bar X_2)$
is one-to-one.
Given $u\in\mathbb T$, represented in $[0,1)$, a point of
$\mathcal G_{k,m}$ satisfying $x_2(\omega_0,\omega_1)=u$
has $X_2(\omega_0,\omega_1)=u+m$,
$\omega_0\in H_k(m,u)$ and $\omega_1=h_m^u(\omega_0)$.
By construction, the map
\[
\omega_0\longmapsto\bar x_2(\omega_0,h_m^u(\omega_0))
\]
covers the circle exactly twice on $H_k(m,u)$. Thus, for every
$v\in\mathbb T$, there are exactly two distinct values
$\omega_0^1,\omega_0^2\in H_k(m,u)$ such that
\[
\bar x_2(\omega_0^i,h_m^u(\omega_0^i))=v,
\qquad i\in\{1,2\}.
\]
This proves the two-to-one assertion.

\medskip
\noindent\textbf{Step 5: The graph representation and derivative bound.}

Write $P(\omega_0,\omega_1):=(X_2,\bar X_2)$ and
$Q(\omega_0,\omega_1):=(Y_2,\bar Y_2)$.
An argument analogous to the proof of the two-to-one property in Step~4 shows that $P$
is one-to-one on $\overline{\mathcal G_{k,m}}$. Thus, the image of $\overline{\mathcal G_{k,m}}$ under $(P,Q)$
is the graph $\left\{\bigl(x,\bar x,\xi(x,\bar x)\bigr):
(x,\bar x)\in\mathcal D_{k,m}\right\}$, where
$\xi:=Q\circ P^{-1}$, as claimed in Lemma~\ref{forward_small}. The critical avoidance imposed on $X_1,\bar X_1$ in step $1$ gives
\eqref{inclusionxi}.

Note that, by definition,
$\overline{H_k(m,u)}=[a,b(u)]$, where
$|\phi_u(b(u))-\phi_u(a)|=2$.
Since $|\phi_u'(b(u))|\asymp L^2$, the implicit function
theorem implies that $b(u)$ is continuously differentiable.
This observation, along with the definition of $\mathcal G_{k,m}$ in
\eqref{bigG}  gives
\[
\mathcal D_{k,m}
=
\left\{(m+u,v):
0\le u\le1,\quad
c_{k,m}(u)\le v\le c_{k,m}(u)+2
\right\},
\]
where
\[
c_{k,m}(u):=\min\{\phi_u(a),\phi_u(b(u))\}.
\]
This proves \eqref{mathcaldkm}.
Differentiating \eqref{x2equalsu} with respect to $u$, we obtain
\[
\partial_u\phi_u(\omega_0)
=
\frac{f'(\bar X_1+h_m^u(\omega_0))}
     {f'(X_1+h_m^u(\omega_0))}
=O(1),
\]
which gives \eqref{cdistortion}.
The derivative estimate \eqref{normofxi} follows from the next proposition.
\begin{proposicao}\label{functional estimate}
The map $Q\circ P^{-1}$ satisfies
\begin{equation}\label{contraction}
\|D(Q\circ P^{-1})\|\lesssim L^{-1}.
\end{equation}
\end{proposicao}
\begin{proof}
Note that 
\begin{align*}
Y_2&=X_1+\omega_1,\\
\bar Y_2&=\bar X_1+\omega_1,\\
X_2&=f(Y_2)-x_0-\omega_0,\\
\bar X_2&=f(\bar Y_2)-\bar x_0-\omega_0.
\end{align*}

The definition of $P$ and $Q$ implies that
\begin{align*}
DQ=\begin{pmatrix}a&1\\\bar a&1\end{pmatrix},\qquad DP=\begin{pmatrix}
ba-1&b\\
\bar b\bar a-1&\bar b
\end{pmatrix},
\end{align*}
where 
\[
a=f'(x_0+\omega_0),\qquad
\bar a=f'(\bar x_0+\omega_0),\qquad
b=f'(Y_2),\qquad
\bar b=f'(\bar Y_2).
\]
In particular, $\det DP=b\bar b(a-\bar a)+(b-\bar b).$ By construction, $|b|,|\bar b|\ge c''L$, and the transversality
condition gives $|a-\bar a|\ge c''L$. Hence $|b\bar b(a-\bar a)|\asymp L^3.$ Since $|b|,|\bar b|\le2\pi L$, it follows that
$|\det DP|\asymp L^3$, and 
\[
(DP)^{-1}=\frac{1}{\det DP}
\begin{pmatrix}
\bar b&-b\\
-(\bar b\bar a-1)&ba-1
\end{pmatrix}.
\]
For $\xi=Q\circ P^{-1}$, this gives
\[
D\xi=DQ\,(DP)^{-1}
=\frac{1}{\det DP}
\begin{pmatrix}
\bar b(a-\bar a)+1&-1\\
1&b(a-\bar a)-1
\end{pmatrix}.
\]
Since $|\bar b(a-\bar a)|+|b(a-\bar a)|\lesssim L^2$, we obtain
$\|D\xi\|=O(L^{-1})$.
\end{proof}

This completes the proof of Lemma \ref{forward_small}.
\end{proof}

\section{Geometric estimates for typical horseshoes}\label{gym}

In this appendix we prove Proposition~\ref{prop:horseshoe-times}.
Throughout this appendix, write
$z_t:=F_\omega^t(z_0)$ and
$\bar z_t:=F_\omega^t(\bar z_0)$.
Since the proof is involved, we begin with an informal overview
of the main ideas.

\subsection{Intuitive idea of the proof}

The idea is to first prove item~\eqref{jesus2}. Indeed, Item~\eqref{jesus1} is just a consequence of the definition of Young times and the cone conditions
\eqref{geoemtricone} and \eqref{geometric2} then follow
from the construction built to prove item \eqref{jesus2}.

More precisely, given $\chi>0$ sufficiently small and two rectangles
$I,\bar I$ in $\chi$-horseshoe position, we aim to prove that, for $(m\otimes m\otimes\mathbb P)$-almost every
$(z_0,\bar z_0,\omega)\in I_c\times\bar I_c\times\Omega$,
the sequence $\{\bar h_j\}_{j\ge0}$ defined in \eqref{heicj}
has the following geometric property: for each $j\ge0$, there exist two vertical strips, one containing $z_{\bar h_j}$ and crossing $I$ vertically, and the other containing $\bar z_{\bar h_j}$ and crossing $\bar I$ vertically, whose images under $F^{\bar h_{j+1}-\bar h_j}_{\vartheta^{3\bar h_j}\omega}$
each cross both $I$ and $\bar I$ horizontally with margin. Here, we describe the construction of the strip containing $z_{\bar h_j}$; the strip containing $\bar z_{\bar h_j}$ is obtained in the same way.

Fix $j\ge1$ and set
\begin{equation}\label{defoftimes}
t_0:=\bar h_j,\qquad t_1:=\bar h_j+K_2,\qquad
t_2:=\bar h_{j+1}-K_1,\qquad t_3:=\bar h_{j+1}.
\end{equation}

At time $t_0$, the point $z_{t_0}\in I$ lies on the curve $\Gamma_j$ defined in \eqref{curvesprop9}, which is almost horizontal and crosses both $I$ and $\bar I$ with margin. Moreover, $t_1$ is a stable time, $t_2$ is a Young time, and $t_2-t_1>D_0$, where $D_0$ is as in \eqref{heicj}.
At time $t_3$, $z_{t_3}$ lies on the almost horizontal curve $\Gamma_{j+1}$, which again crosses both $I$ and $\bar I$ with margin. By the definition of Young times, the orbit points $z_i$ remain bounded away from the critical region for every
integer $i\in[t_0,t_1]\cup[t_2,t_3]$.

Our strategy is to first establish the following geometric facts:
\begin{itemize}
\item we first show that a subcurve of $\Gamma_j$ is mapped into $\Gamma_{j+1}$;
\item we then construct a foliation by short vertical curves, each passing through a point
of $\Gamma_{j+1}$, and show that their images under the
inverse dynamics from time $t_3$ to time $t_0$ cross $I$
vertically with margin;
\item Finally, using these curves, we construct the desired strip and verify the geometric properties stated in \eqref{jesus2}.
\end{itemize}
The geometric facts stated above are proved by combining two different sources of hyperbolicity. The first is critical avoidance on the blocks $[t_0,t_1]$
and $[t_2,t_3]$, which ensures that almost horizontal curves grow under forward iteration at a rate $\asymp L$, while
almost vertical curves grow under backward iteration at
a comparable rate.
The second source is the geometric control
provided by the stable time $t_1$ and the Young time $t_2$ on the intermediate block $[t_1,t_2]$. On this interval, we combine the
reverse-temperedness provided by the Young time $t_2$ with
the temperedness provided by the stable time $t_1$, to establish a strong hyperbolicity property, which can be viewed as a form of
\emph{simultaneous temperedness} (see Lemma \ref{lem:long-common-splitting} and Proposition \ref{coolproposicao}):
for $i\in[t_1,t_2]$, the local dynamics near $z_i$ is
conjugated to a small perturbation of a linear hyperbolic map near $0$, so that at the endpoints $i\in\{t_1,t_2\}$ the associated conjugacy maps neighbourhoods of fixed size around $0$ to neighbourhoods around $z_i$ whose size is polynomial in $L^{-1}$, and sends the horizontal and vertical axes to almost horizontal and almost vertical curves, respectively.

\subsection{Preliminary geometric estimates}

Set $D:=t_2-t_1>0$ and recall $D>D_0,$ where $D_0$ is as in \eqref{heicj}. Set $\alpha_*:=\alpha-2\varepsilon,$ $\varepsilon_*:=3\varepsilon$ and take $\varepsilon$ sufficiently small that $\alpha_*>16\varepsilon_*$.
Let $A_m$ denote the sequence of matrices defined in \eqref{sequenceofmatrices} with $z=z_0.$ Since $t_2$ is a Young time, the sequence of matrices $\{A_m\}_{m=0}^{t_2}$ admits a $(C_1,\alpha,\varepsilon)$ reverse-tempered splitting with unit directions $s_{t_2},u_{t_2}$ and some $C_1>0$. As a result, if we define
\[
a_s:=\frac{A_{t_1}(\omega,z_0)s_{t_2}}
{\|A_{t_1}(\omega,z_0)s_{t_2}\|},
\qquad
a_u:=\frac{A_{t_1}(\omega,z_0)u_{t_2}}
{\|A_{t_1}(\omega,z_0)u_{t_2}\|},
\]
then, by Definition~\ref{def:reverse-tempered-splitting}, there exists $C_1>0$ such that, for
$0\le k\le m\le D$,
\begin{equation}\label{eq:restricted-reverse-angle}
\begin{aligned}
\frac{\|A_m^{(t_1)} a_s\|}{\|A_k^{(t_1)} a_s\|}
&\le e^{C_1}L^3 L^{-\alpha_*(m-k)+\varepsilon_*(D-k)},\\
\frac{\|A_m^{(t_1)} a_u\|}{\|A_k^{(t_1)} a_u\|}
&\ge e^{-C_1}L^{-3}L^{\alpha_*(m-k)-\varepsilon_*(D-k)},\\
\sin\angle(A_k^{(t_1)} a_s,A_k^{(t_1)} a_u)
&\ge e^{-C_1}L^{-3}L^{-\varepsilon_*(D-k)}.
\end{aligned}
\end{equation}
Furthermore, the norm-to-splitting
argument in \cite[Proposition~4.6]{DD24}, with the one-step
bound $\|DF_\omega^{\pm1}\|_\infty\le(2\pi L+1)^3$, gives
orthogonal unit directions $b_s,b_u$, with
$b_s=e^s(\vartheta^{3t_1}\omega,z_{t_1})$, such that
\begin{equation}\label{eq:long-forward-splitting}
\begin{aligned}
\frac{\|A_m^{(t_1)} b_s\|}{\|A_k^{(t_1)} b_s\|}
&\le e^{C_1}L^3 L^{-\alpha_*(m-k)+\varepsilon_* k},\\
\frac{\|A_m^{(t_1)} b_u\|}{\|A_k^{(t_1)} b_u\|}
&\ge e^{-C_1}L^{-3}L^{\alpha_*(m-k)-\varepsilon_* k},\\
\sin\angle(A_k^{(t_1)} b_s,A_k^{(t_1)} b_u)
&\ge e^{-C_1}L^{-3}L^{-\varepsilon_* k},
\end{aligned}
\end{equation}
up to increasing $C_1$ if necessary.

Set
\begin{equation}\label{twosidedsplitting}
e_s:=a_s,\qquad e_u:=b_u,
\end{equation}
and, for a nonzero vector $v$, write
$v(i):=A_i^{(t_1)}v/\|A_i^{(t_1)}v\|$, where the notation $A_m^{(j)}$ is defined in \eqref{eq:shifted-derivative}.

The following lemma formalizes the concept of \emph{simultaneous temperedness} described above.
\begin{lema}\label{lem:long-common-splitting}
Suppose that $\alpha_*>\frac32\varepsilon_*$ and
\begin{equation}\label{eq:long-block-length-condition}
e^{3C_1}L^{9}L^{-(\alpha_*-\frac32\varepsilon_*)D}\le\frac14.
\end{equation}
Set $d_i:=\min\{i,D-i\}$ and
$\sigma_i:=\sin\angle(e_s(i),e_u(i))$. Then
\begin{equation}\label{eq:long-common-angle}
\sigma_i\ge\frac12e^{-C_1}L^{-3-\varepsilon_*d_i},
\qquad 0\le i\le D,
\end{equation}
and, for $0\le k\le m\le D$,
\begin{align}
\frac{\|A_m^{(t_1)}e_s\|}{\|A_k^{(t_1)}e_s\|}
&\le2e^{2C_1}L^6
L^{-(\alpha_*-\varepsilon_*)(m-k)+2\varepsilon_*d_k},
\label{eq:long-common-stable}\\
\frac{\|A_m^{(t_1)}e_u\|}{\|A_k^{(t_1)}e_u\|}
&\ge\frac12e^{-2C_1}L^{-6}
L^{\alpha_*(m-k)-2\varepsilon_*d_k}.
\label{eq:long-common-unstable}
\end{align}
\end{lema}
The proof of this Lemma is given at the end of this appendix. A key consequence is that the splitting
$e_s,e_u$ satisfies reverse-tempered estimates for both
the forward and backward dynamics between $t_1$ and $t_2$,
analogous to those in
Definition~\ref{def:reverse-tempered-splitting}, but with
a polynomial prefactor $L^6$ in place of $L^3$.
In Proposition~\ref{hyperbolict_times}, we showed that
reverse-tempered norms imply growth to large scales
for curves whose tangent directions are bounded away
from the stable direction.
A closer inspection of the arguments in
\cite[Appendix~A]{DD24} yields the following stronger result.

\begin{proposicao}\label{coolproposicao}
Suppose that $\alpha_*>16\varepsilon_*$ and set
$\lambda:=\alpha_*-4\varepsilon_*$.
There exist $P>0$ and affine changes of coordinates $L_i$,
centred at $z_{t_1+i}$, such that
\begin{equation}\label{affineloss}
\|L_i\|\le L^P L^{3\varepsilon_*d_i},
\qquad
\|L_i^{-1}\|\le\sqrt2,
\qquad 0\le i\le D,
\end{equation}
and the following holds.
Fix either $G_i=g_i$ for all $i$ or $G_i=\widehat g_i$
for all $i$, where
\begin{equation}\label{inverseandnot}
g_i:=L_{i+1}\circ F_{\vartheta^{3(t_1+i)}\omega}\circ L_i^{-1},
\qquad
\widehat g_i:=Jg_{D-i-1}^{-1}J,
\end{equation}
with $J(x,y):=(y,x)$. Fix $q>0$ sufficiently large relative to $P$ and $\lambda$,
and choose $\bar P$ sufficiently large relative to $P$ and $q$.
Set
\begin{equation}\label{eq:lb-radii}
r_i:=L^{-\bar P}L^{-\lambda d_i/4},
\qquad
a_i:=\frac{r_i}{16(1+L^{P+1})},
\qquad
Q_i:=[-r_i,r_i]^2.
\end{equation}
Suppose that $G_i(\tilde z_i)=\tilde z_{i+1}$ for
$0\le i<D$ and
\begin{equation}\label{markedsuff}
|\tilde z_i|\le r_i/1600,
\qquad 0\le i\le D.
\end{equation}
Consider a curve $\tilde\gamma_0
=
\left\{
\tilde z_0+\bigl(\xi,\varphi_0(\xi)\bigr):
|\xi|\le l_0
\right\},$ where $\varphi_0(0)=0$, $0<l_0\le a_0$, and
$\|\varphi_0'\|_\infty\le L^{P+1}$.
Then, for $D$ sufficiently large, there exists a subcurve
$\gamma_0^*\subset\tilde\gamma_0$ containing $\tilde z_0$
whose successive images $\gamma_{i+1}^*:=G_i(\gamma_i^*)$,
$0\le i<D$, are such that $\gamma_i^*$ is a graph through
$\tilde z_i$ contained in $Q_i$, for $0\le i\le D$.
Write points on $\gamma_i^*$ as
$\tilde z_i+(\xi_i,\varphi_i(\xi_i))$, where
$\varphi_i(0)=0$, and set $U_i:=\|\varphi_i'\|_\infty$.
Then, we have
\begin{align}
U_{i+1}
&\le 2L^{-2\lambda}U_i+2L^{-\lambda-q},
&&0\le i<D,
\label{intermediate_slopes}\\
\left|\frac{d\xi_{i+1}}{d\xi_i}\right|
&\ge L^{\lambda/2},
&&0\le i<D.
\label{graph_expansion}
\end{align}
In particular, $\gamma_D^*$ can be chosen so that
\begin{equation}\label{final_bounds}
U_D\le5L^{-\lambda-q},
\qquad
l_D=\min\bigl\{L^{\lambda D/2}l_0,a_D\bigr\},
\end{equation}
where $l_D$ denotes the projected half-length of $\gamma_D^*$ on each side of $\tilde z_D$.
\end{proposicao}

\begin{proof}
We first use the estimates of Lemma~G.1 to construct
a family of adapted norms, following the argument of
\cite[Lemma~A.1]{DD24}.
Proceeding as at the beginning of the proof of
\cite[Proposition~A.13]{DD24}, we obtain the charts
$L_i$ satisfying \eqref{affineloss}, in place of
\cite[(A.36)]{DD24}.
For either choice of $G_i$ in \eqref{inverseandnot},
we have $G_i(0)=0$ and
$DG_i(0)=\operatorname{diag}(\mu_i,\nu_i)$, with
$|\mu_i|\ge L^\lambda$ and $|\nu_i|\le L^{-\lambda}$.
Using $\|D^2F_\omega^{\pm1}\|_\infty\le CL^5$,
the computation in \cite[Lemma~A.2]{DD24} gives
\begin{align*}
\|D^2G_i\|_\infty
\le CL^{P+5+3\varepsilon_*}L^{3\varepsilon_*d_i},
\end{align*}
which is our counterpart of \cite[(A.37)]{DD24}.
Restricting to $Q_i$ and applying the mean value theorem,
we obtain the counterpart of \cite[(A.42)]{DD24}:
\begin{align*}
\sup_{Q_i}\|DG_i-DG_i(0)\|
\le CL^{P+5+3\varepsilon_*-\bar P}
L^{-(\lambda/4-3\varepsilon_*)d_i}
\le L^{-q-P-2},
\end{align*}
where the last inequality follows by choosing $\bar P$
sufficiently large and then $L$ sufficiently large,
since $\lambda>12\varepsilon_*$.
The choices of $a_i$ and $Q_i$ in \eqref{eq:lb-radii}
ensure that any graph through $\tilde z_i$ with slope
at most $L^{P+1}$ and projected half-length at most
$a_i$ is contained in $Q_i$.
We can therefore apply the graph-transform estimates
\cite[(A.10) and (A.12)]{DD24} directly to these graphs,
obtaining \eqref{intermediate_slopes}, \eqref{graph_expansion} and
\eqref{final_bounds} by induction.
\end{proof}

\begin{proof}[Proof of Proposition~\ref{prop:horseshoe-times}]

Fix $\delta>0$ sufficiently small and choose $0<\chi\ll1/16$.
By \eqref{affineloss} and the definition of $a_D$, choose
$P^*>0$ sufficiently large that any curve of length at least
$a_D$ in either endpoint chart has length at least $L^{-P^*}$
in the original coordinates. Then choose $K_1,K_2\gg P^*$,
with $K_1$ also satisfying the requirements of
Proposition~\ref{prop:young-times}.

The construction in Section~\ref{finito} ensures that the times
$\{\bar h_j\}_{j\ge0}$ defined in \eqref{heicj} satisfy the lower
density bound \eqref{horseshoetimesdensity} for
$(m\otimes m\otimes\mathbb P)$-almost every
$(z_0,\bar z_0,\omega)\in I_c\times\bar I_c\times\Omega$.
Item~\eqref{jesus1} follows directly by combining
Proposition~\ref{prop:young-times} with the choice $0<\chi\ll1/16$. We now prove the remaining claims, carrying out the geometric construction
along the orbit of $z_0$; the construction along the orbit of $\bar z_0$ is analogous.

\medskip
\noindent\textbf{Step 1: A subcurve of $\Gamma_j$ mapping onto
$\Gamma_{j+1}$.}

Fix $j\ge0$ and let $\Gamma_j$ be as in \eqref{curvesprop9}. We claim that
\begin{equation}\label{keyinclusion}
F_{\vartheta^{3 \bar h_j}(\omega)}^{{\bar h_{j+1}-\bar h_j}}(\Gamma_j)
\supset \Gamma_{j+1},
\end{equation}
where, with a slight abuse of notation, we trim $\Gamma_{j+1}$
so that its endpoints lie on the outermost almost vertical sides of $I^+$ and $\bar I^+$.

Recall that $\Gamma_j$ is almost horizontal and extends a fixed length on both sides of $z_{t_0}$. By the cone estimates along the critical-avoiding block $[t_0,t_1]$, its image $F^{t_1-t_0}_{\vartheta^{3t_0}\omega}(\Gamma_j)$ contains an
almost horizontal subarc extending a fixed length on both
sides of $z_{t_1}$.

Since $t_1$ is a stable time, formula \eqref{eq:stable-configurations}, combined with the angular estimates between $e_s$ and $b_s$ in Lemma \ref{lem:long-splitting-comparison}, gives
$\angle(e_s(0),(0,1))\le CL^{-u}$.
By \eqref{affineloss} and \cite[Lemma A.8]{DD24},
$F^{t_1-t_0}_{\vartheta^{3t_0}\omega}(\Gamma_j)$ contains a
curve that, written in the chart at time $t_1$, is a graph
through the origin with slope at most $L^{P+1}$, where $P$
is as in Proposition~\ref{coolproposicao}, and projected
half-length $L^{-\lambda D/2}a_D$.
Applying Proposition~\ref{coolproposicao} with $G_i=g_i$, where $g_i$ is the forward map as in \eqref{inverseandnot}, and $\tilde z_i=0$
for all $i$, we obtain that
$F^{t_2-t_0}_{\vartheta^{3t_0}\omega}(\Gamma_j)$ contains a
subcurve whose representation in the chart at time $t_2$
is a graph with projected half-length $a_D$ and slope at
most $5L^{-\lambda-q}$.

Returning to the original coordinates, the choice of $P^*$ ensures that this curve extends a length of at least $L^{-P^*}$ on each side of $z_{t_2}$. Furthermore, it is almost horizontal. Indeed, by \eqref{final_bounds} and the choice $q\gg P$, all its tangent directions are arbitrarily close to $e_u(D)$ for $L$ sufficiently large. Almost horizontality then follows from the fact that, since $t_2$ is a Young time, the tangent direction at $z_{t_2}$ is almost horizontal by \eqref{simultaneoushyptimes}. Formula \eqref{keyinclusion} then follows from the choice
$K_1\gg P^*$ and critical avoidance along the block $[t_2,t_3]$.

\noindent\textbf{Step 2: Backward growth of vertical curves.}

Fix $w\in\Gamma_{j+1}$ and define
\begin{equation}\label{Dabliut}
w_t:=\left(F^{t_3-t}_{\vartheta^{3t}\omega}\right)^{-1}(w),
\qquad t_0\le t\le t_3.
\end{equation}
We claim that there exists a vertical curve $\ell_{t_3,w}$, symmetric around $w$, whose backward image at time $t_0$ contains an almost vertical crossing of $I$ with margin.

By cone preservation and backward expansion along the
critical-avoiding block $[t_2,t_3]$, we can choose a vertical curve $\ell_{t_3,w}$ such that its backward image $\ell_{t_2,w}$ is almost vertical and, written in the inverse chart at time $t_2$, is a graph centred at the coordinates of $w_{t_2}$ with projected half-length $a_D L^{-\lambda D/4}.$ By Step 1, $e_u(D)$ is almost horizontal. Thus, by \eqref{affineloss} and \cite[Lemma A.8]{DD24}, this graph has slope at most $L^{P+1}$.
Let $\tilde w_i:=JL_{D-i}(w_{t_2-i})$ denote the coordinates
of $w_{t_2-i}$ in the inverse charts, for $0\le i\le D$.
Since $w\in\Gamma_{j+1}$, increasing $K_1$ in Step~1
if necessary, we may assume that $|\tilde w_0|$ is
sufficiently small for \eqref{graph_expansion} and
\eqref{intermediate_slopes} to give
\begin{align*}
|\tilde w_i|
&\le \frac{r_0}{1600}L^{-\lambda i/2}
\le \frac{r_i}{1600},
\qquad 0\le i\le D.
\end{align*}
Applying Proposition~\ref{coolproposicao} with $G_i=\widehat g_i$ and $\tilde z_i=\tilde w_i$, we obtain a subarc of $\ell_{t_2,w}$ whose backward image at time
$t_1$, denoted by $\ell_{t_1,w}$, is a graph in the inverse
chart, centred at $\tilde w_D$, with projected half-length
$a_D$ and slope at most $5L^{-\lambda-q}$.

Returning to the original coordinates, the choice of $P^*$
ensures that $\ell_{t_1,w}$ extends a length of at least
$L^{-P^*}$ on each side of $w_{t_1}$. Furthermore, it is
almost vertical. Indeed, by \eqref{final_bounds} and the
choice $q\gg P$, all its tangent directions are arbitrarily close to $e_s(0)$ for $L$ sufficiently large, which is almost vertical by step $1$.

Finally, the choice $K_2\gg P^*$ and critical avoidance
along the block $[t_0,t_1]$ imply that backward iteration
preserves almost verticality and expands $\ell_{t_1,w}$
to an almost vertical crossing $\ell_{t_0,w}$ of $I$
with margin.

\medskip
\noindent\textbf{Step 3: Concluding the proof of item \eqref{jesus2}}

Let $\Gamma'_j\subset\Gamma_j$ satisfy $F_{\vartheta^{3\bar h_j}\omega}^{\bar h_{j+1}-\bar h_j}(\Gamma'_j)
=\Gamma_{j+1}.$ Such a sub-curve exists by \eqref{keyinclusion}. For every $w\in\Gamma_{j+1}$, let $w_{t_0}$ be as in
\eqref{Dabliut}. Then $w_{t_0}\in\Gamma'_j$, and the curves
$\ell_{t_0,w}$ cross $I$ vertically with margin. Define
\begin{equation}\label{boundaryargument}
H_j:=\bigcup_{w\in\Gamma_{j+1}}
\bigl(\ell_{t_0,w}\cap I\bigr).
\end{equation}
Every point of $H_j$ belongs to a retained fibre, so its orbit
through the middle block remains in the chart boxes.

To conclude the proof of item~\eqref{jesus2},
we need to prove that
$F^{t_3-t_0}_{\vartheta^{3t_0}\omega}(H_j)$ is a horizontal
strip. Let $\bar w$ and $\tilde w$ be the leftmost and rightmost
endpoints of $\Gamma_{j+1}$. By construction, $H_j$ is a
vertical strip whose vertical boundaries are
$\ell_{t_0,\bar w}$ and $\ell_{t_0,\tilde w}$, and whose
horizontal boundaries are the horizontal subarcs of
$\partial I$ between these two fibres. As a result, the vertical boundaries of $F^{t_3-t_0}_{\vartheta^{3t_0}\omega}(H_j)$ are $\ell_{t_3,\bar w}$ and $\ell_{t_3,\tilde w}$, which, by step $2,$ can be chosen to be almost vertical and exponentially small, whilst its remaining boundary components connect $\ell_{t_3,\bar w}$ with $\ell_{t_3,\tilde w}$ and are images of an almost horizontal curve. The cone estimates, along with the definition of $\bar w,\tilde w\in\Gamma_{j+1},$ show that these components are almost horizontal curves crossing $I$ and $\bar I$ with margin. This proves that $F^{t_3-t_0}_{\vartheta^{3t_0}\omega}(H_j)$ is a horizontal strip.

\noindent\textbf{Step 4: Concluding the proof of item \eqref{jesus3}.}

To prove the cone estimates in \eqref{geoemtricone}, fix
$z\in H_j$ and an almost horizontal vector
$v_{t_0}$ in the tangent space of $z$, and denote its successive images under the derivative cocycle by $v_i$.
On the initial block $[t_0,t_1]$, critical avoidance
ensures that the vector expands and remains almost horizontal.
On the intermediate block $[t_1,t_2]$, the orbit remains
inside the charts, so the tangent estimates used in Step~1
apply. In particular, \eqref{final_bounds} shows that
$v_{t_2}$ is close to the direction $e_u(D)$, which was
shown in Step~1 to be almost horizontal, while
\eqref{graph_expansion} gives expansion in the horizontal
coordinate. Critical avoidance on the final block
$[t_2,t_3]$ then gives further expansion and places
$v_{t_3}$ strictly inside the horizontal cone.
Taking $K_1$ and $K_2$ sufficiently large to absorb
the polynomial losses from the coordinate changes,
we obtain \eqref{geoemtricone} for some sufficiently small
$\beta_*>0$. The same argument applied to the inverse dynamics proves \eqref{geometric2}.
\end{proof}

\begin{proof}[Proof of Lemma \ref{lem:long-common-splitting}]
We first prove the following auxiliary result.
\begin{lema}\label{lem:long-splitting-comparison}
For every $0\le i\le D$,
\begin{align}
\sin\angle(a_u(i),b_u(i))
&\le e^{2C_1}L^6L^{-2\alpha_* i+\varepsilon_* D},
\label{eq:long-unstable-comparison}\\
\sin\angle(a_s(i),b_s(i))
&\le e^{2C_1}L^6L^{-2\alpha_*(D-i)+\varepsilon_* D}.
\label{eq:long-stable-comparison}
\end{align}
\end{lema}

\begin{proof}
Since $A_i^{(t_1)}$ preserves area and the initial representatives have unit
length,
\[
\sin\angle(a_u(i),b_u(i))
=\frac{|\det(a_u,b_u)|}{\|A_i^{(t_1)}a_u\|\,\|A_i^{(t_1)}b_u\|}
\le e^{2C_1}L^6L^{-2\alpha_* i+\varepsilon_* D},
\]
by \eqref{eq:long-forward-splitting} and
\eqref{eq:restricted-reverse-angle} with $k=0$.
For the stable directions, apply area preservation to the matrix $A_{D-i}^{(t_1+i)}$ defined in \eqref{eq:shifted-derivative}:
\begin{align*}
\sin\angle(a_s(i),b_s(i))
&=\|A_{D-i}^{(t_1+i)}a_s(i)\|\,\|A_{D-i}^{(t_1+i)}b_s(i)\|
\sin\angle(a_s(D),b_s(D))\\
&\le \frac{\|A_D^{(t_1)}a_s\|}{\|A_i^{(t_1)}a_s\|}
\frac{\|A_D^{(t_1)}b_s\|}{\|A_i^{(t_1)}b_s\|}
\le e^{2C_1}L^6L^{-2\alpha_*(D-i)+\varepsilon_*(D-i)+\varepsilon_* i},
\end{align*}
where in the last line we used \eqref{eq:long-forward-splitting} and \eqref{eq:restricted-reverse-angle}.
This proves the second estimate.
\end{proof}

We now prove Lemma \ref{lem:long-common-splitting}. Comparing $e_s(i)$ with $b_s(i)$ for $i\le D/2$, and
$e_u(i)$ with $a_u(i)$ for $i\ge D/2$, the preceding angle
estimates and \eqref{eq:long-block-length-condition} give
\[
\sigma_i
\ge e^{-C_1}L^{-3}L^{-\varepsilon_*d_i}
\left(1-2e^{3C_1}L^9L^{-(\alpha_*-\frac32\varepsilon_*)D}\right)
\ge \frac12e^{-C_1}L^{-3}L^{-\varepsilon_*d_i}.
\]
This proves \eqref{eq:long-common-angle}.

Write $S_i:=\|A_i^{(t_1)}e_s\|$ and $U_i:=\|A_i^{(t_1)}e_u\|$.
By area preservation, $S_iU_i\sigma_i$ is constant, so
\begin{equation}\label{eq:long-area-ratios}
\frac{S_m}{S_k}
=\frac{U_k}{U_m}\frac{\sigma_k}{\sigma_m},
\qquad
\frac{U_m}{U_k}
=\frac{S_k}{S_m}\frac{\sigma_k}{\sigma_m}.
\end{equation}
Combining these identities with
\eqref{eq:long-forward-splitting} when $k\le D/2$, and
\eqref{eq:restricted-reverse-angle} when $k\ge D/2$,
gives \eqref{eq:long-common-stable} and
\eqref{eq:long-common-unstable}, using
$d_m\le d_k+(m-k)$.
\end{proof}
\subsection*{Use of AI tools}
The author used ChatGPT (OpenAI) to assist with language editing
and rewriting parts of the manuscript for clarity. The author
takes full responsibility for the content of the paper
\section*{Acknowledgements}
I thank Alex Blumenthal for helpful discussions on the two-point motion.

I acknowledge support from the Engineering and
Physical Sciences Research Council (EPSRC) through grant EP/Z533658/1, ``Predicting Critical Transitions in Complex
Dynamical Networks: Reduction and Learning.''

\clearpage
\bibliographystyle{amsplain}
\bibliography{standard}

\end{document}